\PassOptionsToPackage{sort&compress}{natbib}
\documentclass[final,3p,times,10pt]{elsarticle}
\usepackage[centertags]{amsmath}
\usepackage{amsfonts,amssymb,amsthm,mathtools,epstopdf,newlfont,graphicx,subfigure,booktabs}
\usepackage{bm,cancel,algorithm,algpseudocode}
\usepackage{appendix}
\usepackage[colorlinks=true]{hyperref}
\newtheorem{thm}{Theorem}[section]

\newtheorem{cor}{Corollary}[section]
\newtheorem{conj}{Conjecture}[section]

\numberwithin{algorithm}{section}
\numberwithin{equation}{section}
\renewcommand{\theequation}{\thesection.\arabic{equation}}

\newcommand{\bmx}{\bm{x}}
\newcommand{\bmu}{\bm{u}}
\newcommand{\bmv}{\bm{v}}
\newcommand{\bmy}{\bm{y}}
\newcommand{\bmz}{\bm{z}}
\newcommand{\bmf}{\bm{f}}
\newcommand{\bmh}{\bm{h}}
\newcommand{\bmt}{\bm{t}}
\newcommand{\bmT}{\bm{T}}
\newcommand{\hk}{\hat{k}}
\newcommand{\amin}{\alpha_{c}^-}
\newcommand{\amax}{\alpha_{c}^+}
\newcommand{\bmtau}{\bm{\tau}}
\newcommand{\MB}[1]{\mathbb{#1}}
\newcommand{\hF}[1]{\hat{\F{#1}}}
\newcommand{\MBR}{\MB{R}}
\newcommand{\MBS}{\MB{S}}
\newcommand{\F}[1]{\mathbf{#1}}
\newcommand{\C}[1]{\mathcal{#1}}
\newcommand{\MF}[1]{\mathfrak{#1}}

\newcommand{\tx}{\tilde{x}}
\newcommand{\tu}{\tilde{u}}
\newcommand{\tbmx}{\tilde{\bmx}}
\newcommand{\tbmu}{\tilde{\bmu}}
\newcommand{\bmone}{\bm{\mathit{1}}}
\newcommand{\bmzer}{\bm{\mathit{0}}}

\def\simgt{\,\hbox{\lower0.6ex\hbox{$>$}\llap{\raise0.3ex\hbox{$\sim$}}}\,}
\def\simlt{\,\hbox{\lower0.6ex\hbox{$<$}\llap{\raise0.3ex\hbox{$\sim$}}}\,}
\def\simgteq{\,\hbox{\lower0.6ex\hbox{$\ge$}\llap{\raise0.6ex\hbox{$\sim$}}}\,}
\def\simlteq{\,\hbox{\lower0.6ex\hbox{$\le$}\llap{\raise0.6ex\hbox{$\sim$}}}\,}
\makeatletter
\def\user@resume{resume}
\def\user@intermezzo{intermezzo}
\newcounter{previousequation}
\newcounter{lastsubequation}
\newcounter{savedparentequation}
\renewenvironment{subequations}[1][]{%
      \def\user@decides{#1}%
      \setcounter{previousequation}{\value{equation}}%
      \ifx\user@decides\user@resume 
           \setcounter{equation}{\value{savedparentequation}}%
      \else  
      \ifx\user@decides\user@intermezzo
           \refstepcounter{equation}%
      \else
           \setcounter{lastsubequation}{0}%
           \refstepcounter{equation}%
      \fi\fi
      \protected@edef\theHparentequation{%
          \@ifundefined {theHequation}\theequation \theHequation}%
      \protected@edef\theparentequation{\theequation}%
      \setcounter{parentequation}{\value{equation}}%
      \ifx\user@decides\user@resume 
           \setcounter{equation}{\value{lastsubequation}}%
         \else
           \setcounter{equation}{0}%
      \fi
      \def\theequation  {\theparentequation  \alph{equation}}%
      \def\theHequation {\theHparentequation \alph{equation}}%
      \ignorespaces
}{%
  \ifx\user@decides\user@resume
       \setcounter{lastsubequation}{\value{equation}}%
       \setcounter{equation}{\value{previousequation}}%
  \else
  \ifx\user@decides\user@intermezzo
       \setcounter{equation}{\value{parentequation}}%
  \else
       \setcounter{lastsubequation}{\value{equation}}%
       \setcounter{savedparentequation}{\value{parentequation}}%
       \setcounter{equation}{\value{parentequation}}%
  \fi\fi
  \ignorespacesafterend
}

\makeatother
\begin{document}
\begin{frontmatter}
\title{A Direct Integral Pseudospectral Method for Solving a Class of Infinite-Horizon Optimal Control Problems Using Gegenbauer Polynomials and Certain Parametric Maps}
\author[fahd]{Kareem T. Elgindy}
\ead{kareem.elgindy@(kfupm.edu.sa; gmail.com)}
\author[sohag]{Hareth M. Refat\corref{cor1}}
\ead{harith_refaat@science.sohag.edu.eg; hareth.mohamed.refat@gmail.com}
\cortext[cor1]{Corresponding author}
\address[fahd]{Mathematics Department, College of Computing and Mathematics, King Fahd University of Petroleum \& Minerals, Dhahran 31261, Kingdom of Saudi Arabia} 
\address[sohag]{Mathematics Department, Faculty of Science, Sohag University, Sohag 82524, Egypt} 
\begin{abstract}
We present a novel direct integral pseudospectral (PS) method (a direct IPS method) for solving a class of continuous-time infinite-horizon optimal control problems (IHOCs). The method transforms the IHOCs into finite-horizon optimal control problems (FHOCs) in their integral forms by means of certain parametric mappings, which are then approximated by finite-dimensional nonlinear programming problems (NLPs) through rational collocations based on Gegenbauer polynomials and Gegenbauer-Gauss-Radau (GGR) points. The paper also analyzes the interplay between the parametric maps, barycentric rational collocations based on Gegenbauer polynomials and GGR points, and the convergence properties of the collocated solutions for IHOCs. Some novel formulas for the construction of the rational interpolation weights and the GGR-based integration and differentiation matrices in barycentric-trigonometric forms are derived. A rigorous study on the error and convergence of the proposed method is presented. A stability analysis based on the Lebesgue constant for GGR-based rational interpolation is investigated. Two easy-to-implement pseudocodes of computational algorithms for computing the barycentric-trigonometric rational weights are described. Two illustrative test examples are presented to support the theoretical results. We show that the proposed collocation method leveraged with a fast and accurate NLP solver converges exponentially to near-optimal approximations for a coarse collocation mesh grid size. The paper also shows that typical direct spectral/PS- and IPS-methods based on classical Jacobi polynomials and certain parametric maps usually diverge as the number of collocation points grow large, if the computations are carried out using floating-point arithmetic and the discretizations use a single mesh grid whether they are of Gauss/Gauss-Radau (GR) type or equally-spaced.  
\end{abstract}


\end{frontmatter}
\section{Introduction}
\label{Int}
Arguably, one of the most impactful numerical methods for solving continuous-time optimal control problems (CTOCPs) in the 20th century has been direct pseudospectral (PS) methods, which can accurately reduce CTOCPs into optimization problems of standard forms that can be easily treated using typical optimization methods. The key success of these methods lie in their ability to converge to sufficiently smooth solutions with exponential rates using relatively coarse mesh grids. PS methods are considered to be \textit{``one of the biggest technologies for solving PDEs''} that were largely developed about half a century ago since the pioneering works of \citet{orszag1971accurate} and \citet{patterson1971spectral}. They have been continuously refined and extended in later decades to solve many problems in various scientific areas that were only tractable by these techniques. Perhaps one of the brightest moments in the course of their development appeared on March 3, 2007, when an international space station completed a 180-degree maneuver without using any propellant via tracking an attitude trajectory developed with PS optimal control theory; thus, saving NASA {\$1M} \cite{kang2007pseudospectral}. PS methods are closely related to the popular class of spectral methods, but they expand the solutions in terms of their grid point values by means of interpolation in lieu of global and usually orthogonal basis polynomials. Such a nodal representation is extremely useful in the sense that the solution values are immediately available at the collocation points once the full discretization is implemented, as the the governing equations are satisfied pointwise in the physical space, whereas modal representations require a further step of computing the modal approximation after calculating the coefficients of the expansion basis functions \cite{elgindy2019high}. This places PS methods at the front of highly accurate methods that are particularly easy to apply to equations with variable coefficients and nonlinearities \cite{fornberg1994review}. Clear expositions of spectral and PS methods, exhibiting a wide range of outlooks on the subject, include the books \cite{fornberg1998practical,hesthaven2007spectral,canuto1987springer,canuto2007spectral}. A robust variant of PS methods is the class of integral PS (IPS) methods (aka PS integration methods), which is closely related to PS methods, but it requires an initial step of reformulating the dynamical system equations in their integral form first before the collocation phase starts; thus, avoids the degradation of precision often caused by numerical differentiation processes. The integral reformulation can be performed by either a direct integration of the dynamical system equations if they have constant coefficients, or by approximating the solution's highest-order derivative involved in the problem by a nodal finite series in terms of its grid point values, and then solve for those grid point values before successively integrating back in a stable manner to obtain the sought solution grid point values. The spectral approximation of the integral form of differential equations was put forward in the 1960s by \citet{clenshaw1960method} in the spectral space and by \citet{el1969chebyshev} in the physical space; cf. \cite{lee1997fast,greengard1991spectral,elgindy2013solving,elgindy2016high,elgindy2018high,elgindy2018high1}.

Among the many classes of CTOCPs, infinite-horizon optimal control problems (IHOCs) and optimal control problems defined on sufficiently large intervals have attracted a lot of research interest due to their size of applications in economics, engineering, computer science, business and management science, bio-medicine, aerospace, energy, etc.; cf. \cite{ling2012envelope,barucci2001technology,ross2012review,gao2019online,wang2022intelligent,parandehgheibi2015value,janova2016optimal,pang2022human}. Some classical results on the existence of solutions for IHOCs can be found in \cite{baum1976existence,bates1978lower,haurie1980existence,carlson1987infinite}. One of the most general and well-known results on the existence of solutions to IHOC problems was proved by \citet{balder1983existence} using the notion of uniform integrability. Sufficient conditions for the existence of a finitely optimal solution for a class of nonlinear IHOCs were derived by \citet{carlson1986existence} under minimal convexity and seminormality conditions. An existence and uniqueness theorem for a class of IHOC problems was proved by \citet{wang2005existence} under certain conditions. Existence and uniqueness results for a class of linear-quadratic, convex IHOC problems in weighted Sobolev spaces for the state and weighted Lebesgue spaces for the control were obtained by \citet{pickenhain2015infinite}. A recent extension to the existence results of \citet{balder1983existence} to the case in which the integral functional is understood as an improper integral was proved by \citet{besov2018balder} using the notion of uniform boundedness of pieces of the objective functional that was proposed earlier by \citet{dmitruk2005existence}. \citet{aseev2018existence} derived some sufficient conditions for the existence and boundedness of optimal controls for a class of generally nonlinear IHOC problems with not necessarily bounded set of control constraints. \citet{basco2019hamilton} obtained some existence and uniqueness results of weak solutions of nonautonomous Hamilton–Jacobi–Bellman equation associated with a class of IHOC problems for the class of lower semicontinuous functions vanishing at infinity and under certain conditions of controllability. The most important and well-known necessary conditions of optimality were first derived by \citet{halkin1974necessary}; cf. also \cite[Theorem 2.3]{carlson1987infinite}. 

While many direct PS methods appeared in the literature for solving finite-horizon optimal control problems (FHOCs), we could only find a few works on IHOCs using this class of methods. In particular, we recognize the Legendre–Gauss (LG) and Legendre–Gauss–Radau (LGR) PS methods of \citet{garg2002gauss,garg2011pseudospectral,garg2011direct} and the transformed Legendre spectral method of \citet{shahini2018transformed}. Although Legendre polynomials are commonly used in PS methods designed to solve IHOCs, we shall explore in our work the possibility to whether we could achieve better accuracy and convergence rates using Gegenbauer polynomials (aka ultraspherical polynomials). There are a number of reasons that prompt us to consider this family of polynomials as a viable alternative to perform discretizations of IHOCs: (i) First, observe that Gegenbauer polynomials include both Chebyshev and Legendre polynomials as part of its bigger family, so all theoretical and experimental results on Gegenbauer polynomials directly apply on Chebyshev and Legendre polynomials by definition, (ii) being part of Gegenbauer polynomials allow us to apply any of Chebyshev and Legendre polynomials with a single selection of the Gegenbauer parameter (index) $\alpha$-- simply set $\alpha = 0$ or $1/2$ in your code! Thus, giving us more flexibility, (iii) Gegenbauer polynomials are very useful in eliminating the Gibbs phenomenon and recovering the spectral accuracy up to the discontinuity points \cite{gottlieb1995gibbs,gottlieb1997gibbs,kamm2010application}, (iv) one measure for assessing the quality of spectral and PS methods in numerical discretizations is concerned with how large the number of terms is required in a spectral/PS expansion to achieve a certain level of accuracy. Of course, the smaller the number of terms the more efficient the method is in terms of speed and computational complexity. In applications like CTOCPs, this property leads to optimization problems of small-scale which can be solved very quickly with reduced computational work at a concrete level using modern optimization software \cite{elgindy2013fast,elgindy2020distributed,elgindy2018high}. Now, with this being mentioned, it is important to realize that Chebyshev and Legendre polynomials are usually optimal for large spectral/PS expansions under the Chebyshev and Euclidean norms, respectively, but they are not necessarily optimal for a small/medium range-- an observation that was proven numerically in a number of papers for certain polynomial- and rational-interpolations and collocations; cf. \cite{doha1990accurate,abd2014new,elgindy2013optimal,elgindy2013solving,
Elgindy2016,elgindy2016high,fornberg1994review}, which give us another reason to apply Gegenbauer polynomials as a proper basis polynomials that may provide faster convergence rates. (v) A stability analysis conducted in \cite{elgindy2018high1} and grounded in the Lebesgue constant for polynomial interpolations in Lagrange-basis form based on flipped-GGR (FGGR) points showed that the Lebesgue constant is not minimal for Chebyshev polynomials but rather was minimal for Gegenbauer polynomials associated with negative $\alpha$-values. This analysis proved with no doubt that some Gegenbauer polynomials with negative $\alpha$-values could be more plausible to employ in basis-form polynomial interpolation/collocation for short/medium range of mesh grid sizes. This observation is consistent with an earlier work of \citet{light1978comparison} who proved in the late 1970's that the Chebyshev and Legendre projection operators cannot be minimal as the norms of Gegenbauer projection operators increase monotonically with $\alpha$ for small expansions. 

In light of the above arguments, we are motivated in this work to develop a novel direct IPS method for solving IHOCs using Gegenbauer polynomials and study its convergence. To this end, we derive some accurate and numerically stable GGR-based rational interpolation formulas, and describe two computational algorithms for constructing them. We show also how to derive the associated quadrature formulas required for numerical integration in time. We shall then use these numerical instruments to approximate the optimal state and control variables after transforming the IHOC into a FHOC in integral form (FHOCI) by means of certain parametric maps and rational collocation. During the course of our paper presentation, we shall try also to investigate a number of interesting relevant questions to our work. For instance, which parametric map is more suitable for GGR-based rational collocations? How should we choose the Gegenbauer parameter values to carry out collocations in practice? A \textit{``poor''} choice of $\alpha$ can largely ruin the accuracy of the numerical scheme, while a \textit{``good''} choice can in many cases furnish superb approximations with higher accuracy than those enjoyed by Chebyshev and Legendre polynomials for sufficiently smooth functions using relatively coarse mesh grids! Do PS- and IPS-methods based on Chebyshev, Legendre, and Gegenbauer polynomials generally converge to the solutions of IHOCs for large mesh sizes? If they do not, then what are the causes? Through rigorous stability, error, and convergence analyses, we shall prove that such methods often converge with exponential rate to near exact solutions using relatively small mesh grids, but they usually diverge for fine meshes under certain parametric maps.

The rest of the article is organized as follows: Sections \ref{sec:PS1} and \ref{sec:TOTIHOCP1} describe the IHOC under study and its transformation into a FHOCI via various parametric maps. Section \ref{sec:NOIHOC} presents the discretization scheme of the FHOCI passing through the construction of the needed barycentric rational interpolants and their quadratures, and closing with a set up of the IPS rational collocation at the GGR points in Sections \ref{sec:BIATGP}-\ref{subsec:PSRC2404}. Section \ref{subsubsec:TLcfggr1} is devoted to analyze the stability and sensitivity of GGR-based rational interpolation/collocation developed in this paper. The optimality necessary conditions of the obtained NLP through IPS rational collocation are derived in Section \ref{sec:NCOOFTNP}. Rigorous error and convergence analyses are conducted in Section \ref{sec:errb}. Some divergence results of typical IPS collocation schemes of the FHOCI for fine meshes of Gauss Type using certain parametric maps are derived in Section \ref{subsec:DOTPCS}. Simulation results are shown in Section \ref{sec:NE1} followed by some conclusions and future works
in Section \ref{sec:CAFW1}. The derivation of the barycentric rational formulas necessary for constructing the GGR-based differentiation matrix is shown in \ref{sec:TDIM}. Two easy-to-implement pseudocodes of computational algorithms for computing the barycentric weights of the our new rational interpolation method are described in \ref{app:Alg1}.

\section{The Problem Statement}
\label{sec:PS1}
Consider the following nonlinear, autonomous control system of ordinary differential equations
\begin{equation}
\dot{\bmx}(t)= \bmf(\bmx(t),\bmu(t)), \quad t \in [0,\infty), 
\label{eq:2}
\end{equation}
subject to the system of initial conditions
\begin{equation}
\bmx(0)=\bmx_{0},
\label{eq:3}
\end{equation}
where $\bmx(t) = \left[x_{1}(t), x_{2}(t), \ldots, x_{n_{x}}(t)\right]^t \in \MBR^{n_{x}}, \bmu(t)= \left[u_{1}(t), u_{2}(t), \ldots, u_{n_{u}}(t)\right]^t \in \MBR^{n_{u}}, \bmx_0 = [x_{1,0}, x_{2,0}, \ldots, x_{n_x,0}]^t \in \MBR^{n_{x}}$ is a constant specified vector, and $\bmf:\MBR^{n_{x}} \times \MBR^{n_{u}} \rightarrow \MBR^{n_{x}}: \bmf = [f_1, f_2, \ldots, f_{n_x}]^t$. The problem is to find the optimal control $\bmu(t)$ and the corresponding state trajectory $\bmx(t)$ on the semi-infinite-domain $[0, \infty)$ that satisfy Eqs. \eqref{eq:2} and \eqref{eq:3} while minimizing the functional 
\begin{equation}
J = \int_{0}^{\infty} g(\bmx(t),\bmu(t))\, dt,
\label{eq:1}
\end{equation}
where $g: \MBR^{n_{x}} \times \MBR^{n_{u}} \rightarrow \MBR$. We assume that $\bmf$ and $g$ are generally nonlinear, continuously differentiable functions with respect to their arguments, and the nonlinear IHOC \eqref{eq:2}–\eqref{eq:1} has a unique solution. In the rest of the article, for any row/column vector $\bmy = (y_i)_{1 \le i \le n}$ with $y_i \in \MBR\, \forall i$ and real-valued function $h: \bm{\Omega} \subseteq \mathbb{R} \to \mathbb{R}$, the notation $h(\bmy)$ stands for a vector of the same size and structure of $\bmy$ such that $h(y_i)$ is the $i$th element of $h(\bmy)$. Moreover, by $\bmh(\bmy)$, we mean $[h_1(\bmy), \ldots, h_m(\bmy)]^t$, for any $m$-dimensional column vector function $\bmh$, with the realization that the definition of each array $h_i(\bmy)$ follows the former notation rule for each $i$.

\section{Transformation of the IHOC into a FHOC}
\label{sec:TOTIHOCP1}
Given a differentiable, strictly monotonic mapping $T: [0, \infty) \to [-1,1)$ defined by $T(\tau) = t$, one can transform the IHOC \eqref{eq:2}–\eqref{eq:1} into the following FHOC:
\begin{subequations}
\begin{equation}
\min J= \int_{-1}^{1} \displaystyle{
T'(\tau) g\left(\tbmx(\tau),\tbmu(\tau)\right) }\, d\tau,
\label{eq:5}
\end{equation}
subject to
\begin{equation}
\dot{\tbmx}(\tau)= T'(\tau) \bmf(\tbmx(\tau),\tbmu(\tau)),\quad \tau \in [-1,1),
\label{eq:6}
\end{equation}
\vspace{-2mm}
\begin{equation}
\tbmx(-1)=\bmx_{0},
\label{eq:7}
\end{equation}
where $\bmf\left(\tbmx(\tau),\tbmu(\tau)\right)=\left[f_{1}\left(\tbmx(\tau),\tbmu(\tau)\right),\ldots,f_{n_{x}}\left(\tbmx(\tau),\tbmu(\tau)\right)\right]^t$, and $\tilde{\boldsymbol{\eta}}(\tau) = \boldsymbol{\eta}\left(T(\tau)\right)\,\forall \boldsymbol{\eta} \in \lbrace \bmx, \bmu \rbrace$. To take advantage later of the well-conditioning of numerical integration operators during the collocation phase, we rewrite Eq. \eqref{eq:6} in its integral formulation as follows:
\begin{equation}
\tbmx(\tau)=  \int_{-1}^{\tau} \displaystyle{T'(z) \bmf\left(\tbmx(z),\tbmu(z)\right) }\, dz + \bmx_{0}, \quad \tau \in [-1,1).
\label{eq:8}
\end{equation}
\end{subequations}
We refer to the FHOC described by Eqs. \eqref{eq:5}, \eqref{eq:7}, and \eqref{eq:8} by the FHOCI. A wide variety of defining formulas exist for the mapping $T$. Five common defining formulas of such a mapping that occurred in the literature are as follows:
\begin{subequations}
\begin{align}
  &{T_{1,L}}(\tau ) = \frac{L(1 + \tau)}{{1 - \tau }},\quad (\text{\citet{boyd1987orthogonal}})\label{eq:num1}\\
  &{T_{2,L}}(\tau ) = L \ln \left( {\frac{2}{{1 - \tau }}} \right),\quad (\text{\citet{canuto1987springer}})\label{eq:num2}\\
  &{\varphi_a}(\tau ) = \frac{{1 + \tau }}{{1 - \tau }},\quad (\text{\citet{fahroo2008pseudospectral}})\label{eq:fahroo2008p}\\
  &%\!\!% first approximation of  misalignment compensation ...
                \negthickspace% for compensation of rcases misalignment  
  \begin{rcases}  
  {\varphi_b}(\tau ) = \ln \left( \displaystyle{{\frac{2}{{1 - \tau }}}} \right),\label{garg2011a}\\
  {\varphi_c}(\tau ) = \ln \left( \displaystyle{{\frac{4}{{{{(1 - \tau )}^2}}}}} \right),
  \end{rcases}\quad (\text{\citet{garg2011advances}})  
\end{align}
\end{subequations}
where $L \in \MBR^+$ is a scaling parameter that can stretch the image of the interval $[-1,1]$ in the codomain $[0, \infty)$ as desired. An \textit{``optimal''} choice of $L$-value can significantly improve the quality of the discrete approximations as we shall demonstrate later in Section \ref{sec:NE1}. We refer to $L$ by \textit{``the map scaling parameter.''} The parametric maps $T_{i,L}, i = 1, 2$ are often referred to by the \textit{``algebraic''} and \textit{``logarithmic''} maps in the literature, respectively. Notice that the maps defined by Eqs. \eqref{eq:fahroo2008p} and \eqref{garg2011a} are special cases of the parametric maps $T_{i,L}, i = 1, 2$; in particular, $T_{1, 1} = \varphi_a, T_{2,1} = \varphi_b$, and $T_{2,2} = \varphi_c$. Since the value of either parametric map varies when $\alpha$ varies for arguments of type GGR points, it is more convenient to denote them by $T_{i,L}^{(\alpha)}, i = 1, 2$ to emphasize this fact. Mesh-like surfaces of both parametric mappings are shown in Figures \ref{fig:IHtp1} and \ref{fig:IHtp2}, for several values of $n, L$, and $\alpha$. Both figures show that the parametric mappings (i) increase monotonically for decreasing values of $\alpha$ while holding $n$ and $L$ fixed, (ii) increase monotonically for increasing values of $L$ while holding $n$ and $\alpha$ fixed, and (iii) increase monotonically for increasing values of $n$ while holding $L$ and $\alpha$ fixed. Moreover, near $\tau = 1$, the rate of increase of ${T_{1,L}^{(\alpha)}}$ with respect to any of the arguments $n, L$, and $\alpha$ while holding the others fixed is much larger than that of ${T_{2,L}^{(\alpha)}}$ which grows very slowly. Loosely put, the stretching of the mesh grid near $\tau = 1$ is stronger for $T_{1,L}^{(\alpha)}$ than $T_{2,L}^{(\alpha)}$.

\begin{figure}[h]
\includegraphics[scale=0.35]{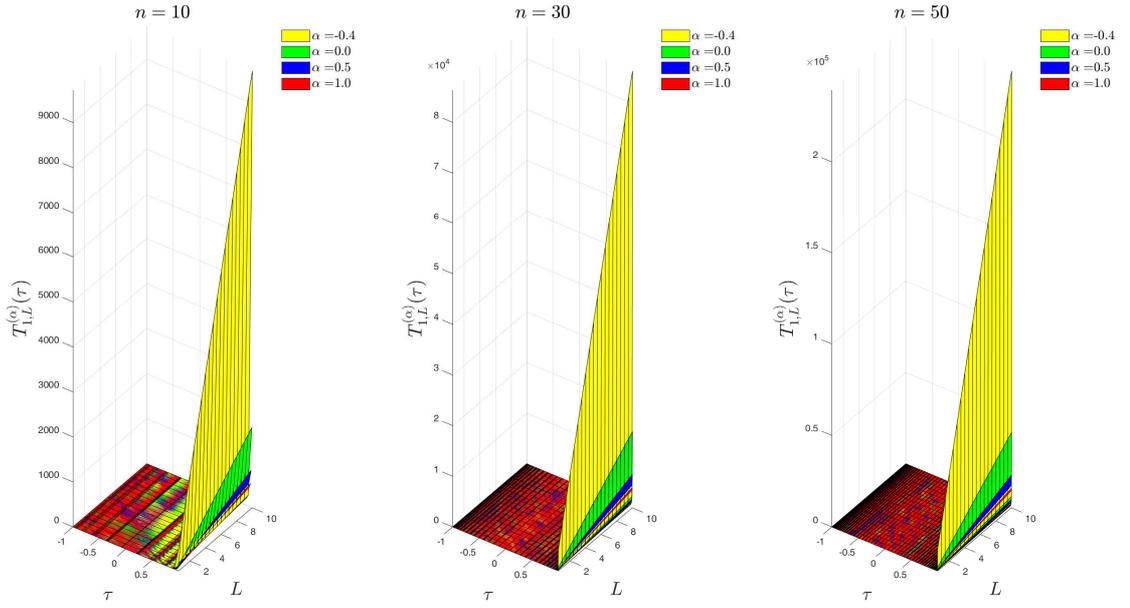}
\caption{Mesh-like surfaces of the parametric mapping $T_{1,L}^{(\alpha)}$ on the discrete rectangular domain $\F{\Omega}_n = \{(\tau_i, L): L = 0.5(0.5)10, i = 0, \ldots,n\}$, for $n = 10(20)50$ and $\alpha = -0.4, 0, 0.5, 1$.}
\label{fig:IHtp1}
\end{figure}
\begin{figure}[h]
\includegraphics[scale=0.35]{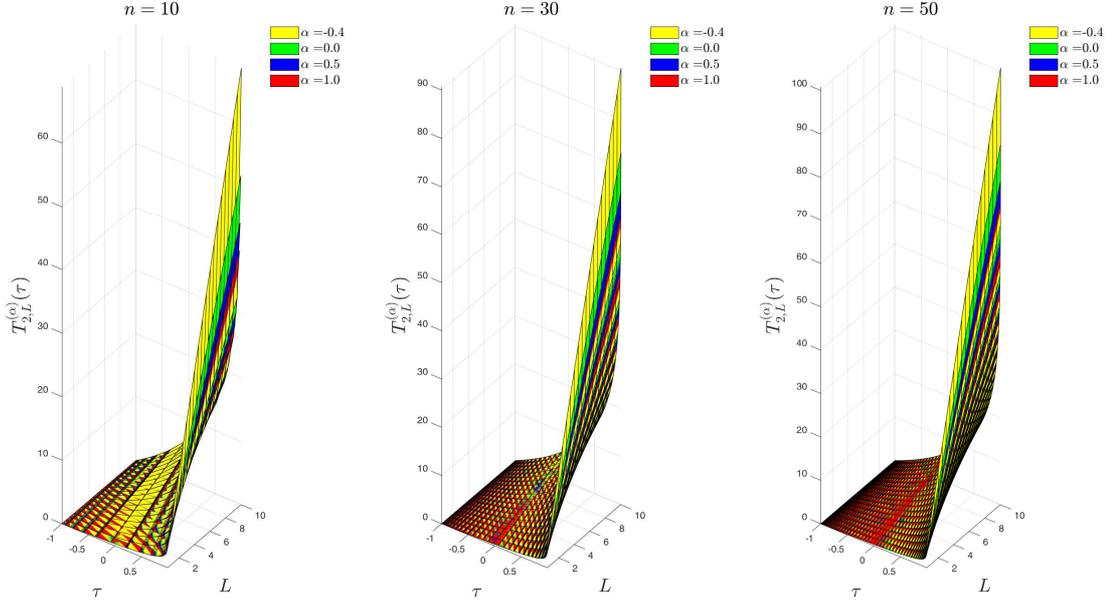}
\caption{Mesh-like surfaces of the parametric mapping $T_{2,L}^{(\alpha)}$ on the discrete rectangular domain $\F{\Omega}_n = \{(\tau_i, L): L = 0.5(0.5)10, i = 0, \ldots,n\}$, for $n = 10(20)50$ and $\alpha = -0.4, 0, 0.5, 1$.}
\label{fig:IHtp2}
\end{figure}

\section{Numerical Discretization of the FHOCI}
\label{sec:NOIHOC}
In this section, we provide a description of the proposed numerical discretization of the FHOCI using an IPS method based on Gegenbauer polynomials and GGR points. 

\subsection{Barycentric Rational Interpolation at the GGR Points}
\label{sec:BIATGP}
Let $\MB{Z}^{+}$ be the set of positive integers, $G_{n}^{(\alpha)}(\tau)$ be the $n$th-degree Gegenbauer polynomial with $\alpha > -1/2$ and $G_{n}^{(\alpha)}(1) = 1,\;\forall n \in \MB{Z}_{0}^{+}= \MB{Z}^{+}\cup \lbrace 0 \rbrace$, and $\MBS_{n} = \left\{\tau_{k}: \C{G}_{n+1}^{(\alpha)}\left(\tau_k\right) = 0\text{ for }k = 0, \ldots, n,\text{ and } -1 = \tau_0 < \tau_1 < \ldots < \tau_n\right\}$ be the set of GGR nodes, where $\C{G}_{n+1}^{(\alpha)}(\tau) = G_{n}^{(\alpha)}(\tau) + G_{n+1}^{(\alpha)}(\tau)$. The orthonormal Gegenbauer basis polynomials are defined by $\phi_{j}^{(\alpha)}(\tau) = G_{j}^{(\alpha)}(\tau)/\sqrt{\lambda_{j}}$, where
\begin{equation}\label{eq:11n}
\lambda_{j} = \frac{2^{2\alpha-1} j! \Gamma^{2}{(\alpha+\frac{1}{2})}}{(j+\alpha)\Gamma(j+2 \alpha)},\quad j = 0, \ldots, n.
\end{equation}
They satisfy the discrete orthonormality relation
\begin{equation}\label{eq:13app}
\sum \limits_{j=0}^{n} \varpi_j\, \phi_{s}^{(\alpha)} (\tau_{j}) \phi_{k}^{(\alpha)} (\tau_{j}) = \delta_{sk},\quad s,k = 0, \ldots, n,
\end{equation}
where $\varpi_j, j = 0, 1, \ldots, n$, are the corresponding Christoffel numbers of the GGR quadrature formula on the interval $[-1,1]$ defined by
\begin{subequations}\label{subeq:14}
\begin{align}
\varpi_{0} &= \left(\alpha +\frac{1}{2}\right) \vartheta_{0},\label{subeq:14_1}\\
 \varpi_j &= \vartheta_{j},\quad j = 1, 2, \ldots, n,\label{subeq:14_2}
\end{align}
with
\begin{equation}\label{eq:15}
\vartheta_{j} = 2^{2 \alpha-1} \frac{\Gamma^{2}{\left(\alpha + \frac{1}{2}\right)}\,n!}{\left(n+\alpha+\frac{1}{2}\right) \Gamma{(n+2\alpha+1)}} \left(1-\tau_{j}\right) \left(G_{n}^{(\alpha)}(\tau_{j})\right)^{-2},\quad j = 0, 1, \ldots, n.
\end{equation}
\end{subequations}
Given a set of $n + 1$ data points $\left\{(\tau_i,f_i)\right\}_{i=0}^n$, the Gegenbauer polynomial interpolant $P_{n}f$ in Lagrange form is defined by 
\begin{equation}
P_{n}f(\tau) = \sum_{i=0}^{n} f_{i} \C{L}_{n,i}(\tau),
\label{eq:new3}
\end{equation}  
where $\C{L}_{n,i}$ are the Lagrange polynomials given by
\[\C{L}_{n,i}(\tau) = \frac{{\prod\nolimits_{k \ne i} {(\tau - {\tau _k})} }}{{\prod\nolimits_{k \ne i} {({\tau _i} - {\tau _k})} }},\quad \forall i.\]
$P_{n}f$ can be evaluated fast and more stably by evaluating Lagrange polynomials through the \textit{``true''} barycentric formula
\begin{equation}
\C{L}_{n,i}(\tau) = \frac{\xi_{i}}{\tau-\tau_{i}}/\sum_{j=0}^{n}\frac{\xi_{j}}{\tau-\tau_{j}}, \quad \forall i,
\label{eq:new1}
\end{equation}
which brings into play the barycentric weights $\xi_{i}, i = 0, \ldots, n$, given by
\begin{equation}
\xi_{i} = \frac{1}{\prod_{i\ne j}^{n} \left(\tau_{j}-\tau_{i}\right)}, \quad \forall i.
\label{eq:new2}
\end{equation}
An interpolation in Lagrange form with Lagrange polynomials defined by Eqs. \eqref{eq:new1} is often referred to by \textit{``a barycentric rational interpolation.''} The barycentric weights associated with the GGR points can be expressed explicitly in terms of the corresponding Christoffel numbers through the following theorem.
\begin{thm}\label{th:new1}
The barycentric weights for the GGR points are given by
\begin{subequations}\label{subeq:new}
\begin{align}
\xi_{0} &=-\sqrt{(2\alpha+1)\varpi _{0}},\label{subeq:news1}\\
\xi_{i}&=(-1)^{i-1}\sqrt{\left(1-\tau_{i}\right)\varpi _{i}}, \quad i=1,2,\ldots,n.\label{subeq:news2}
\end{align}
\end{subequations}
\end{thm}
\begin{proof}
Let $P_n^{(\alpha, \beta)}(\tau)$ be the Jacobi polynomial of degree $n$ and associated with the parameters $\alpha, \beta > -1$ as normalized by \citet{szeg1939orthogonal}. Through \citet[Eq. (4.5.4)]{szeg1939orthogonal} and \citet[Eq. (A.1)]{elgindy2013fast}, we have
\begin{align*}
&(1 + \tau)P_n^{(\alpha  - 1/2,\alpha  + 1/2)}(\tau) = \frac{2}{{2n + 2\alpha  + 1}}\left[ {(n + \alpha  + 1/2)\,P_n^{(\alpha  - 1/2,\alpha  - 1/2)}(\tau) + (n + 1)\,P_{n + 1}^{(\alpha  - 1/2,\alpha  - 1/2)}(\tau)} \right] \hfill \\
   &= \frac{2}{{2n + 2\alpha  + 1}}\left[ {\frac{{(n + \alpha  + 1/2)\,\Gamma (n + \alpha  + 1/2)}}{{n!\,\Gamma (\alpha  + 1/2)}} G_n^{(\alpha )}(\tau ) + \frac{{(n + 1)\,\Gamma (n + \alpha  + 3/2)}}{{(n + 1)!\,\Gamma (\alpha  + 1/2)}}G_{n + 1}^{(\alpha )}(\tau )} \right] = \frac{{\Gamma (n + \alpha  + 1/2)}}{{n!\,\Gamma (\alpha  + 1/2)}}\C{G}_{n + 1}^{(\alpha )}(\tau).
\end{align*}
Therefore, $(1+\tau_i) P_{n+1}^{(\alpha-1/2,\,\alpha+1/2)}(\tau_i) = \C{G}_{n+1}^{(\alpha)}(\tau_i)\,\forall i$, and Formulas \eqref{subeq:news1} and \eqref{subeq:news2} can be derived from \cite[Theorem 3.6]{wang2014explicit} by replacing both $\alpha$ and $\beta$ with $\alpha-1/2$.
\end{proof}
Recall that the GGR points cluster near $\pm 1$ as $n \rightarrow \infty$, so Eq. \eqref{subeq:news2} may suffer from cancellation errors for values of $\tau_i$ sufficiently close to $1$. The following theorem provides two alternative trigonometric forms of Formula \eqref{subeq:news2}.
\begin{thm}\label{th:new2}
The barycentric weights corresponding to the interior GGR points are given by
\begin{subequations}
\begin{align}
\xi_{i} &=(-1)^{i-1} \sin{\left(\frac{1}{2}\cos^{-1}{\tau_{i}}\right)} \sqrt{2 \varpi _{i}},\label{subeq:new2s2}\\
        &=(-1)^{i-1} \sin{\left(\cos^{-1}{\tau_{i}}\right)} \sqrt{\frac{\varpi _{i}}{1+\tau_{i}}}, \quad i=1,2,\ldots,n.\label{subeq:new3s2}
\end{align}
\end{subequations}
\end{thm}
\begin{proof}
Through the change of variables $\tau = \cos{\theta}$ and the double angle rule for the cosine function, we find that 
\[\sqrt {1 - {\tau _i}}  = \sqrt {1 - \cos \theta _i}  = \sqrt {2{{\sin }^2}\left( {\frac{1}{2}{\theta _i}} \right)}  = \sqrt 2 \sin \left({\frac{1}{2}{{\cos }^{ - 1}} {\tau _i}} \right),\quad i=1,2,\ldots, n,\]
from which Formula \eqref{subeq:new2s2} is derived. Formula \eqref{subeq:new3s2} is established by realizing that
\[\sqrt {1 - {\tau _i}}  = \sqrt {\frac{1 - {\tau _i}^{2}}{1 + {\tau _i}}}= \frac{\sin{\left(\cos^{-1}{\tau_{i}}\right)}}{\sqrt{1 + {\tau _i}}},\quad i=1,2,\ldots, n,\]
which completes the proof.  
\end{proof}
We refer to Formulas \eqref{subeq:new2s2} and \eqref{subeq:new3s2} by the trigonometric-barycentric weights. Formula \eqref{subeq:news2} is faster to compute and requires a smaller number of arithmetic operations compared with Formulas \eqref{subeq:new2s2} and \eqref{subeq:new3s2}, but the latter two formulas may possibly produce smaller errors near $\tau = 1$ as we observed through numerical experiments. This suggests to perform Formula \eqref{subeq:news2} for all values of $\tau$, except when $\tau$ is sufficiently close to $1$, where we switch to the other trigonometric forms. To this end, we introduce \textit{``a switching parameter,''} $0 < \varepsilon \ll 1$, at which the interchange of formulas is performed. The crossover value of $\varepsilon$, where it becomes more accurate to use the trigonometric form, will depend on the implementation; a prescription of this strategy is outlined in Algorithms \ref{algorithm:1} and \ref{algorithm:2}. We refer to the explicit formulas used in Algorithms \ref{algorithm:1} and \ref{algorithm:2} to compute the barycentric weights by the \textit{``first- and second-switching formulas''} of the barycentric weights for the GGR points, respectively. We also denote the errors in computing the barycentric weights using Formula \eqref{subeq:news2}, Algorithm \ref{algorithm:1}, and Algorithm \ref{algorithm:2}, by $E_{i,n}^{(\alpha)}$, for $i = 1, 2,$ and $3$, respectively. Figures \ref{fig:E1E2_vareps1} and \ref{fig:E1E2_vareps2} show comparisons between the three strategies in terms of error, for a certain value range of $\varepsilon, n$, and $\alpha$. While Algorithm \ref{algorithm:1} does not look promising against the usual Formula \eqref{subeq:news2} relative to the given input data as clearly seen from Figure \ref{fig:E1E2_vareps1}, Figure \ref{fig:E1E2_vareps2} manifests that Algorithm \ref{algorithm:2} is more numerically stable near $\varepsilon = 0.1$ and provides better approximations. We shall therefore use the latter algorithm with $\varepsilon = 0.1$ for the computation of the barycentric weights, and refer to the rational interpolation with barycentric weights obtained through the second switching formulas by the \textit{``switched rational (SR) interpolation.''}

\begin{figure}[h]
\includegraphics[scale=0.35]{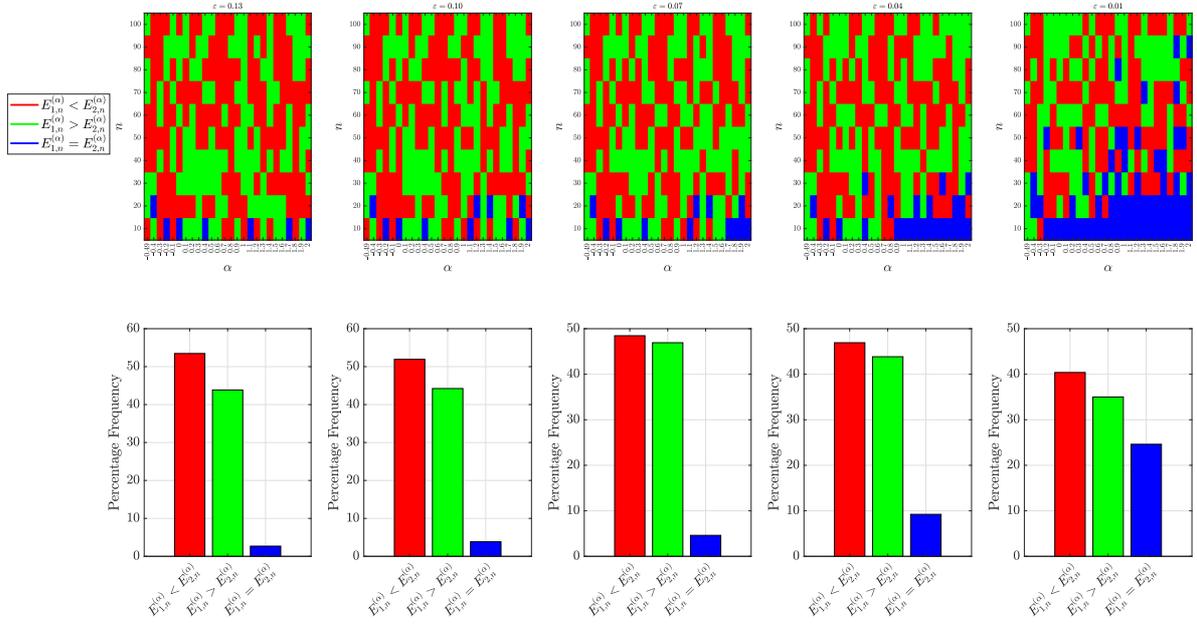}
\caption{The first row displays the color maps corresponding to the switching parameter values $\varepsilon = 0.13, 0.1, 0.07, 0.04$, and $0.01$. Each color map shows areas delineated by red, green, and blue colors. Each color specifies whether $E_{1,n}^{(\alpha)}$ is smaller/larger/equal to $E_{2,n}^{(\alpha)}$, for $n = 10(10)100$ and $\alpha = -0.49, -0.4(0.1)2$. The second row shows the percentage frequency that $E_{1,n}^{(\alpha)}$ occurs as smaller/larger/equal to $E_{2,n}^{(\alpha)}$ in each color map. All computations were carried out using MATLAB in double-precision floating-point arithmetic.}
\label{fig:E1E2_vareps1}
\end{figure}

\begin{figure}[h]
\includegraphics[scale=0.35]{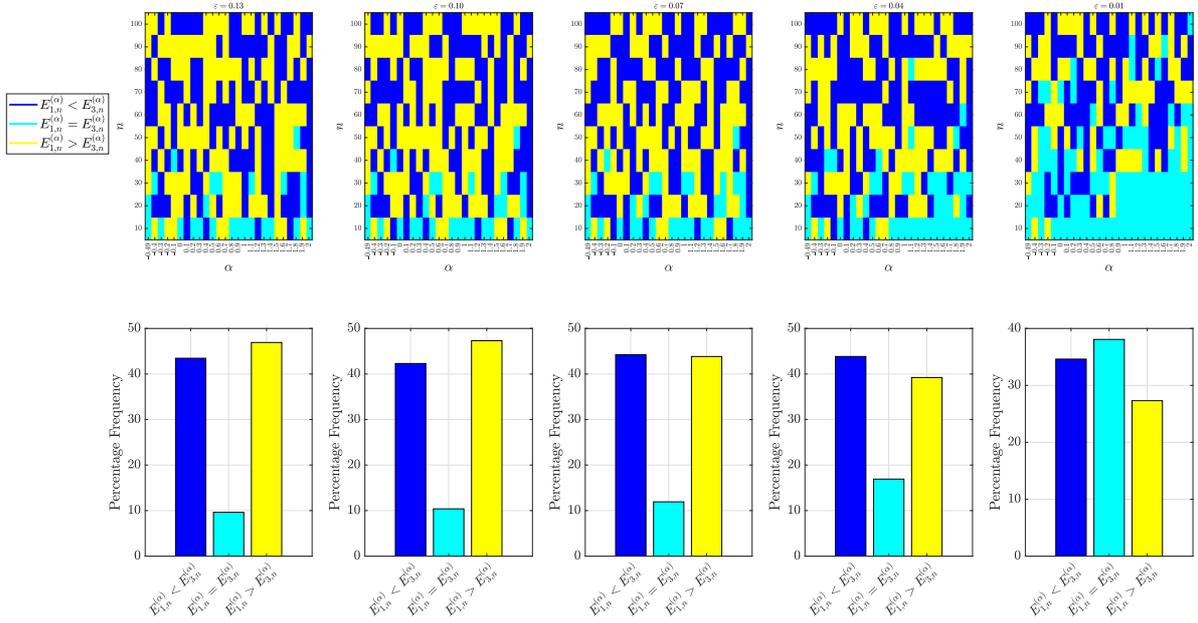}
\caption{The first row displays the color maps corresponding to the switching parameter values $\varepsilon = 0.13, 0.1, 0.07, 0.04$, and $0.01$. Each color map shows areas delineated by blue, cyan, and yellow colors. Each color specifies whether $E_{1,n}^{(\alpha)}$ is smaller/larger/equal to $E_{3,n}^{(\alpha)}$, for $n = 10(10)100$ and $\alpha = -0.49, -0.4(0.1)2$. The second row shows the percentage frequency that $E_{1,n}^{(\alpha)}$ occurs as smaller/larger/equal to $E_{3,n}^{(\alpha)}$ in each color map. All computations were carried out using MATLAB in double-precision floating-point arithmetic.}
\label{fig:E1E2_vareps2}
\end{figure}

\subsubsection{Stability and Sensitivity Analyses of GGR-Based SR-Interpolation/Collocation}
\label{subsubsec:TLcfggr1}
A valuable device for measuring the quality and numerical stability of polynomial interpolations is Lebesgue constant, as it provides a measure of how close the interpolant of a function is to the best polynomial approximant of the function. The Lebesgue constant is also very useful in assessing the quality of approximate solutions obtained through collocation (aka collocated solutions), as their accuracy is related to the rate at which the Lebesgue constant increases. 

Let $\left\| f \right\|_{\MBS}=\sup{\lbrace |f(x)|:x \in \MBS\rbrace}$ be the uniform norm (or supremum norm) of a real-valued, bounded function $f$ defined on a set $\MBS \subseteq  \MB{R}$. Suppose that $y_i, i = 0, \ldots ,n$ and $y_{c,n}(\tau)$ denote the exact solution values at the GGR points $\tau_{i}, i = 0, \ldots ,n,$ and the corresponding collocated solution, respectively. Let also $p_n^*y$ and $p_{n}y$ be the best polynomial approximation to the exact solution $y$ on $[-1,1)$ and the Lagrange interpolating polynomial of degree at most $n$ that interpolate the data set $\left\{ {\left( {\tau_{i}, y_i} \right)} \right\}_{i = 0}^n$ respectively. Through the uniqueness of Lagrange interpolation, one can easily show that
\begin{equation}
\left| {p_n^*y(\tau ) - {p_n}y(\tau )} \right| = \left| {\sum\limits_{i = 0}^n {\left[ {p_n^*y({\tau _i}) - {y_i}} \right]{\C{L}_{n,i}}(\tau )} } \right| \le \Lambda _n^{(\alpha )}{\left\| {y - p_n^*y} \right\|_{[ - 1,1)}},
\end{equation}
where $\Lambda _n^{(\alpha )} = \mathop {\max }\nolimits_{-1 \le x < 1} \sum\limits_{i = 0}^n {\left| {\C{L}_{n,i}(\tau)} \right|}$ denotes the Lebesgue constant associated with GGR-based SR-interpolation. Therefore,
\begin{gather}
{\left\| {y - {y_{c,n}}} \right\|_{[ - 1,1)}} = {\left\| {y - p_n^*y + p_n^*y - {p_n}y + {p_n}y - {y_{c,n}}} \right\|_{[ - 1,1)}}\nonumber\\
 \le {\left\| {y - p_n^*y} \right\|_{[ - 1,1)}} + {\left\| {p_n^*y - {p_n}y} \right\|_{[ - 1,1)}} + {\left\| {{y_{c,n}} - {p_n}y} \right\|_{[ - 1,1)}} \le \left( {1 + \Lambda _n^{(\alpha )}} \right){\left\| {y - p_n^*y} \right\|_{[ - 1,1)}} + {\left\| {\delta {y_{c,n}}} \right\|_{[ - 1,1)}},\label{eq:collocationnew1}
\end{gather}
where $\delta y_{c,n}(\tau)$ is the difference between Lagrange interpolating polynomial and the collocated solution. When ${\left\| {\delta {y_{c,n}}} \right\|_{[ - 1,1)}} \approx 0, \Lambda _n^{(\alpha )}$ roughly bounds the collocation error in the sense that it nearly quantifies how much larger the collocation error ${\left\| {y - {y_{c,n}}} \right\|_{[ - 1,1)}}$ is compared to the smallest possible error, ${\left\| {y - p_n^*y} \right\|_{[ - 1,1)}}$, in the worst case. In this case, it is obvious from Ineq. \eqref{eq:collocationnew1} that the smaller the Lebesgue constant, the better the predicted collocated solution is in the uniform norm. In other words, the collocation error is about at most a factor $1+\Lambda _n^{(\alpha )}$ worse than the best possible polynomial approximation. One can clearly see also that $\Lambda _n^{(\alpha )}$ depends on the location of the collocation points $\tau_i,\,i=0,\ldots, n$ but not on the solution values $y_i,\,i=0,\ldots,n$. Since the positions of the GGR points change as $\alpha$ and $n$ vary, we are interested to learn the  apt choices of $\alpha$ that makes $\Lambda _n^{(\alpha )}$ as small as possible while holding $n$ fixed. This can provide some useful insight on how should we select the candidate range of $\alpha$-values often used for collocations based on GGR points. In \cite{elgindy2018high1}, our findings uncovered that $\Lambda _n^{(\alpha )}$ for FGGR-based polynomial interpolation in Lagrange-basis form blows up as $\alpha \to -0.5$ and monotonically increase for increasing positive values of $\alpha$; cf. \cite[Eqs. (3.7) and (3.8)]{elgindy2018high1}. Moreover, it was noticed that `$\Lambda _n^{(\alpha )}$ does not decrease monotonically for increasing negative values of $\alpha$', indicating that $\Lambda _n^{(\alpha )}$ is not minimal for Chebyshev polynomials but rather attains its smallest value for Gegenbauer polynomials associated with some negative values of $\alpha$; cf. \cite[Figures 1 and 3]{elgindy2018high1}. It was roughly estimated in many earlier works through theoretical and numerical evidences that reasonably good $\alpha$-values for polynomial interpolation in basis form typically belong to the \textit{``Gegenbauer collocation interval of choice,''} $I_{\varepsilon ,r}^G$, defined by $I_{\varepsilon ,r}^G = \left\{ {\alpha | - 1/2 + \delta \le \alpha  \le r,0 < \delta  \ll 1,r \in [1,2]} \right\}$, for reasons pertaining to the stability and accuracy of numerical schemes employing Gegenbauer polynomials as basis polynomials; cf. \cite{elgindy2013optimal,elgindy2017high,elgindy2018high1,elgindy2019high,elgindy2020distributed} and the Refs. therein.

On the other hand, it was discovered in a number of works that Lebesgue constants for rational interpolation at equally-spaced nodes are much smaller than those associated with classical polynomial interpolation; cf. \cite{berrut1994linear,berrut1997lebesgue,carnicer2010weighted,wang2010rational,bos2013bounding}. Moreover, the Lebesgue constant of Berrut's rational interpolation \cite{berrut1988rational} at equidistant nodes is smaller than the Lebesgue constant for polynomial interpolation at Chebyshev nodes; cf. \cite{bos2011lebesgue}. Figure \ref{NceLebesgueRat} shows the surface of Lebesgue constant for GGR-based rational interpolation characterized by Eqs. \eqref{eq:new3} and \eqref{eq:new1} with barycentric weights obtained through the second switching formulas. The surface is constructed through least-squares approximation and is shown together with some of its cross-sections with the vertical planes $\alpha = -0.49,-0.2,-0.1,0,0.5,1,1.5$. A number of remarks deserve to be made at this point: (i) First, notice how small the Lebesgue constant values are for GGR-based rational interpolation compared with its values for the FGGR-based polynomial interpolation in basis form; cf. \cite[Figures 1]{elgindy2018high1}, (ii) $\Lambda _n^{(\alpha )}$ does not blow up as $\alpha \to -0.5$, but rather remains bounded, (iii) the associated Lebesgue constant grows logarithmically in the number of collocation nodes, (iv) it is interesting also to see how the Lebesgue constant drops monotonically as $\alpha$ increases while holding $n$ fixed. This suggests that Legendre polynomials are generally more suited for GGR-based SR-interpolations than Chebyshev polynomials. In fact, we can also observe from Figure \ref{NceLebesgueRat} that Gegenbauer polynomials with increasing $\alpha$-values are associated with smaller Lebesgue constant values. This suggests that Gegenbauer polynomials with $\alpha > 1/2$ may also be more plausible to employ in SR-interpolation for short/medium range of $n$-values. However, the work of \citet{elgindy2013optimal} manifests that Gegenbauer quadratures `may become sensitive to round-off errors for positive and large values of the parameter $\alpha$ due to the narrowing effect of the Gegenbauer weight function,' which drives the quadratures to become more extrapolatory with greater uncertainty in integral approximations; thus, the collocation is subject to a higher risk of producing meaningless results. In other words, ${\left\| {\delta {y_{c,n}}} \right\|_{[ - 1,1)}}$ may become large and the collocation error grows accordingly. It was observed also in \cite{elgindy2013optimal} that the weight function cease to exist near the boundaries $\tau = \pm 1$, and its support is nonzero only on a subinterval centered at $\tau = 0$ for increasing values of $\alpha > 1$. If we refer to GGR-based collocations employing SR-interpolations by the \textit{``SR-collocations,''} then this analysis suggests that, for a relatively large collocation mesh size, SR-collocations is expected to produce higher-order approximations for nonnegative $\alpha$-values with apparently optimal $\alpha$-values within/near the \textit{``Gegenbauer SR-collocation interval of choice (SRCIC)''} $\Upsilon_{\amin,\amax}^G = [\amin, \amax]: \amin \approx 1/2, 1/2 \le \amax \le 1$; in addition, Gegenbauer polynomials with positive and large $\alpha$-values are generally not apt for SR-interpolation/collocation. For small mesh sizes, however, there is no rule-of-thumb as to how should we select $\alpha$, since all Lebesgue constant curves converge to the same limit as $n \to 1$. The analysis in this section assume that the problem under study is well-conditioned. For sensitive problems, the interval of choice $\Upsilon_{\amin,\amax}^G$ may change depending on the sources of sensitivity. In Section \ref{sec:errb}, we shall expound that proper collocations of the FHOCI using any of the maps \eqref{eq:num1}-\eqref{garg2011a} entails shifting the right boundary, $\amax$, of $\Upsilon_{\amin,\amax}^G$ rightward as the mesh size grows large to reduce the divergence rate of the collocated solutions from the exact solutions.

\begin{figure}[h]
\centering
\includegraphics[scale=0.5]{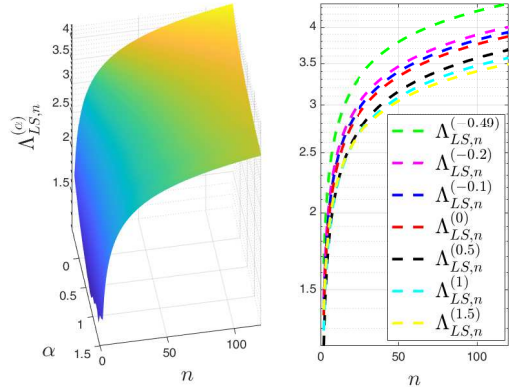}
\caption{The left plot shows the surface of Lebesgue constant for GGR-based SR-interpolation on the discrete rectangular domain $\{(n,\alpha): n = 2(1)120, \alpha = [-0.499,-0.49,-0.4(0.1)1.5]\}$. The surface is constructed through least-squares approximation using curves of the form $c_1 + c_2 \ln n$, for some real parameters $c_1$ and $c_2$ with logarithmic scale on the $z$ axis. The right plot shows the cross-sections of the surface with the vertical planes $\alpha = -0.49,-0.2,-0.1,0,0.5,1,1.5$.}
\label{NceLebesgueRat}
\end{figure}

From another perspective, it is interesting to mention that the Lebesgue constant is also a useful instrument in observing how the collocated solutions change as the input data are varied. By closely following the convention in \cite{elgindy2018high1}, suppose that $\tilde y_{i}, i = 0, \ldots ,n,$ and $\tilde y_{c,n}(\tau)$ are the perturbed solution values due to round-off or input data errors, and the perturbed collocated solution, respectively. Moreover, assume that $\tilde p_{n}y(\tau)$ is the Lagrange interpolating polynomial of degree at most $n$ that interpolate the data set $\left\{ {\left( {\tau_{i}, \tilde y_{i}} \right)} \right\}_{i = 0}^n$. Then we have
\begin{subequations}
\begin{gather}
{\left\| {y_{c,n} - \tilde y_{c,n}} \right\|_{[-1,1)} } = {\left\| {{p_{n}}y - \delta y_{c,n} - {{\tilde p}_{n}}y + \delta \tilde y_{c,n}} \right\|_{[-1,1)} } \le {\left\| {{p_{n}}y - {{\tilde p}_{n}}y} \right\|_{[-1,1)} } + {\left\| {\delta y_{c,n} - \delta \tilde y_{c,n}} \right\|_{[-1,1)} }\nonumber\\
= \mathop {{\text{max}}}\limits_{-1 \le \tau < 1} \left| {\sum\limits_{i = 0}^n {\left( {y_i - \tilde y_i} \right) \C{L}_{n,i}(\tau)} } \right| + {\left\| {\delta y_{c,n} - \delta \tilde y_{c,n}} \right\|_{[-1,1)} } \le \Lambda _n^{(\alpha )} \mathop {{\text{max}}}\limits_{0 \le i \le n} \left| {y_i - \tilde y_i} \right| + {\left\| {\delta y_{c,n} - \delta \tilde y_{c,n}} \right\|_{[-1,1)} },\label{eq:hi1}\\
\le \Lambda _n^{(\alpha )} {\left\| {y - \tilde y} \right\|_{[ - 1,1)}} + {\left\| {\delta y_{c,n} - \delta \tilde y_{c,n}} \right\|_{[-1,1)} },\label{eq:hi2}
\end{gather}
\end{subequations}
where $\delta \tilde y_{c,n}(\tau)$ is the difference between the perturbed Lagrange interpolating polynomial and the perturbed collocated solution. When ${\left\| {\delta y_{c,n} - \delta \tilde y_{c,n}} \right\|_{[-1,1)} }$ is relatively small, $\Lambda _n^{(\alpha )}$ nearly quantifies how much larger the perturbation error of the collocated solution, ${\left\| {y_{c,n} - \tilde y_{c,n}} \right\|_{[-1,1)} }$, is compared to the maximum possible perturbation error of the solution at the collocation points, $\mathop {{\text{max}}}\limits_{0 \le i \le n} \left| {y_i - \tilde y_i} \right|$, or to the maximum solution perturbation error, ${\left\| {y - \tilde y} \right\|_{[ - 1,1)}}$, in the worst case. 

\subsection{The Barycentric GRIM and Quadratures}
\label{sec:TGIM1}
Consider a real-valued function $f$ defined on the interval $[-1, 1]$ and its GGR-Based SR-interpolation given by Eqs. \eqref{eq:new3}, \eqref{eq:new1}, and the second switching formulas of the barycentric weights. Following the work of \citet{elgindy2017high1}, the formulas needed to construct the nonzero rows of the barycentric GRIM can be derived by integrating Eq. \eqref{eq:new3} on the successive intervals $\left[-1,\tau_{j}\right], j = 1, \ldots, n$, to obtain
\begin{equation}
\int_{-1}^{\tau_{j}}P_{n}f(\tau)\,d\tau = \sum_{i=0}^{n} f_{i}\int_{-1}^{\tau_{j}} \C{L}_{n,i}(\tau)\,d\tau,\quad  j=1,\ldots,n,
\label{n}
\end{equation}
where $f_i = f(\tau_i)\,\forall i$. With the change of variable
\begin{equation}
\tau = \frac{1}{2} \left[\left(\tau_{j}+1\right)t+\tau_{j}-1\right],
\label{eq:GIM2}
\end{equation}
we can rewrite Eq. \eqref{n} as 
\begin{equation}\label{eq:GIM3}
\int_{-1}^{\tau_{j}}P_{n}f(\tau)\,d\tau=\frac{\tau_{j}+1}{2}\sum_{i=0}^{n} f_{i}\int_{-1}^{1} \C{L}_{n,i}\left(t;-1,\tau_{j}\right) dt, \quad j=1,2,\ldots,n,
\end{equation}
where 
\[\C{L}_{n,i}\left(t;-1,\tau_{j}\right) = \C{L}_{n,i}\left(\frac{\left(\tau_{j}+1\right)t+\tau_{j}-1}{2}\right),\quad \forall i, j.\]
Since the polynomials $\C{L}_{n,i}\left(t;-1,\tau_{j}\right), i = 0, \ldots, n$, are of degree $n$, the integrals $\int_{-1}^{1} \C{L}_{n,i}\left(t;-1,\tau_{j}\right) dt$ can be computed exactly using an $N=\left\lceil(n+1)/2\right\rceil$-point LG quadrature, where $\left\lceil . \right\rceil$ denotes the ceiling function. Let $\lbrace \bar{\tau}_{k},\bar{\varpi}_{k}\rbrace_{k=0}^{N}$ be the set of LG quadrature nodes and weights, respectively, where
\begin{equation}
\bar{\varpi}_{k}=\frac{2}{\left(1-\bar{\tau}_{k}^2\right)\left(L'_{N+1}\left(\bar{\tau}_{k}\right)\right)^{2}}, \quad k=0,\ldots,N,
\label{eq:GIM4}
\end{equation}
and $L'_{N+1}$ denotes the derivative of the $(N+1)$st-degree Legendre polynomial $L_{N+1}$. Then
\begin{equation}
\int_{-1}^{1} \C{L}_{n,i}\left(t;-1,\tau_{j}\right)\,dt=\sum_{k=0}^{N} \bar{\varpi}_{k} \C{L}_{n,i}\left(\bar{\tau}_{k};-1,\tau_{j}\right).
\label{eq:GIM5}
\end{equation}
Hence, Eqs. \eqref{eq:GIM3} and \eqref{eq:GIM5} yield
\begin{equation}
\int_{ - 1}^{{\tau _j}} {f(\tau)\,d\tau} \approx \sum_{i=0}^{n} q_{j,i}  f_{i}=\F{Q} \bmf,\quad j = 0, \ldots, n,
\label{eq:GIM6}
\end{equation}
where $\bmf=\left[f_{0},f_{1},\ldots,f_{n}\right]^{t}$, and $q_{j,i},i,j=0,\ldots,n$, are the elements of the first-order barycentric GRIM $\F{Q}$ given by
\begin{equation}\label{eq:GIM7}
{q_{j,i}}=
\begin{cases}
0, \quad j = 0,\,i=0,\ldots,n,\\
\displaystyle{\frac{\tau_{j}+1}{2}\sum_{k=0}^{N} \bar{\varpi}_{k} \C{L}_{n,i}\left(\bar{\tau}_{k};-1,\tau_{j}\right)},\quad j=1,2,\ldots,n,\, i=0,\ldots,n. 
\end{cases}
\end{equation}
We denote the $j$th row of $\F{Q}$ by $\F{Q}_{j}\,\forall j$. The derivation of the formulas required to construct the GGR-based differentiation matrix (GRDM) in barycentric form is described in \ref{sec:TDIM}.
%Moreover, $\canczer{\F{Q}}$ stands for the matrix obtained by deleting the first row of $\F{Q}$.

\subsection{IPS Rational Collocation of the FHOC at the GGR Points}
\label{subsec:PSRC2404}
Let $\bmtau_{n} = \left[\tau_{0}, \tau_{1}, \ldots, \tau_{n}\right]^t, \tbmx_i = \tbmx\left(\tau_{i}\right), \tbmu_i = \tbmu\left(\tau_{i}\right), \bmf_i = \bmf(\tbmx_i,\tbmu_i), f_{i,j} = f_{i}(\tbmx_j,\tbmu_j)\,\forall i,j$, and
\[
\bm{\C{J}}_{n} = \left[ \int_{-1}^{\tau_{0}}  T'(\tau) \bmf(\tbmx(\tau),\tbmu(\tau))\,d\tau,\ldots,\int_{-1}^{\tau_{n}} T'(\tau)  \bmf(\tbmx(\tau),\tbmu(\tau))\,d\tau\right].
\]
Then collocating Eq. \eqref{eq:8} at the GGR nodes yields   
\begin{equation}
%\text{matrix form of \ref{eq:9}}\quad
\tbmx\left(\bmtau_{n}\right)= \text{vec}\left(\bm{\C{J}}_{n}\right) + \bmx_{0} \otimes \bmone_{n+1} \approx \hF{M} + \bmx_{0} \otimes \bmone_{n+1},
\label{eq:9mf}
\end{equation}
where 
\[\hF{M} = \text{vec}\left(\F{M}\right):\quad \F{M} = \F{Q} \left(\F{F}\circ \left[ T'\left(\bmtau_{n}\right) \otimes \bmone_{n_{x}}^{t}\right]\right),\]
$\F{F}=\left[\F{F}_{1,n},\ldots,\F{F}_{n_{x},n}\right], \F{F}_{i,n}=\left[f_{i,0},\ldots,f_{i,n}\right]^t\,\forall i, \bmone_{n}$ is the all ones column vector of size $n, [.,.], \text{``vec''}, \circ$, and $\otimes$ denote the horizontal matrix concatenation, the vectorization of a matrix, the Hadamard product, and the Kronecker product, respectively. Let $\tau_{n + 1} = 1$ and define ${{\F{Q}}_{n + 1}} = \left( {{q_{n+1,i}}} \right)_{0 \le i \le n}: {q_{n+1, i}} = \sum_{k=0}^{N} \bar{\varpi}_{k} \C{L}_{n,i}\left(\bar{\tau}_{k}\right)\,\forall i$, then, the discrete cost functional $J$ can be approximated numerically using the LG quadrature as follows:
\begin{equation}
J \approx J_n = \F{Q}_{n+1}\displaystyle{\left(T'\left(\bmtau_{n}\right) \circ \bm{g}\right)},
\label{eq:13}
\end{equation} 
where $\bm{g}=\left[g_{0},\ldots,g_{n}\right]^t: g_{i}=g\left(\tbmx_i,\tbmu_i\right)\,\forall i$. Hence, the FHOCI \eqref{eq:5}, \eqref{eq:7}, and \eqref{eq:8} is now converted into a nonlinear programming problem (NLP) in which the goal is to minimize the discrete cost functional \eqref{eq:13} subject to the nonlinear system of equations \eqref{eq:9mf}. If we define the image of the collocation points set $\MBS_n$ under the transformation $T$ by $\MBS_n^T = \{t_k: t_{k} = T\left(\tau_{k}\right),\; k = 0, \ldots, n\}$, and denote $\bmx(t_i)$ and $\bmu(t_i)$ by $\bmx_i$ and $\bmu_i\,\forall i$, respectively, then the NLP can be solved using well-developed optimization software for the unknowns $\tbmx_i = \bmx_i,\, i = 1, \ldots, n$, and $\tbmu_j = \bmu_j,\, j = 0,\ldots, n$. The approximate optimal state and control variables can then be calculated at any point $t \in [0, \infty)$ through the PS expansions
\begin{subequations}
\begin{align}
\bmx(t) &= \bmx\left(\bmt_n^t\right) \boldsymbol{\C{L}}_{n}(t),\quad \text{and}\label{eq:17}\\
\bmu(t) &= \bmu\left(\bmt_n^t\right) \boldsymbol{\C{L}}_{n}(t),\label{eq:18}
\end{align}
\end{subequations}
where $\bmt_{n} = \left[t_{0}, t_{1}, \ldots, t_{n}\right]^{t}$ and $\boldsymbol{\C{L}}_{n}(t) = \left[\C{L}_{n,0}(t),\C{L}_{n,1}(t),\ldots,\C{L}_{n,n}(t)\right]^{t}.$
In the special case, when $T = T_{1,L}^{(\alpha)}(\tau)$, one can easily show that the NLP can be written as follows:
\begin{subequations}
\begin{equation}
\min J_n = 2 L \F{Q}_{n+1}\displaystyle{\left[ \bm{g} \oslash \left(\bmone_{n+1}-\bmtau_{n}\right)_{2}   \right]}
\label{eq:T1}
\end{equation} 
subject to
\begin{equation}
\tbmx\left(\bmtau_{n}\right) \approx   2 L\,\hF{M}_{1} + \bmx_{0} \otimes \bmone_{n+1},
\label{eq:T1c}
\end{equation}
\end{subequations}
where $\oslash$ denotes the Hadamard division, $(\bmv)_{r}  = \underbrace {\bmv \circ \bmv \circ \ldots \circ \bmv}_{r\text{ times}}$, for any vector $\bmv$, and
\begin{equation}
\hF{M}_{1} = \text{vec}\left(\F{M}_{1}\right):\quad \F{M}_{1} = \F{Q}\left(\F{F}  \oslash \left[ \left(\bmone_{n+1}-\bmtau_{n}\right)_{2}  \otimes \bmone_{n_{x}}^{t}\right]\right).
\label{eq:NOTDIV1}
\end{equation}
Furthermore, when $T = T_{2,L}^{(\alpha)}(\tau)$, the NLP can be formulated as follows:
\begin{subequations}
\begin{equation}
\min J_n = L \F{Q}_{n+1}\displaystyle{\left[ \bm{g} \oslash \left(\bmone_{n+1}-\bmtau_{n}\right)   \right]}
\label{eq:T2}
\end{equation} 
subject to
\begin{equation}
%\text{matrix form of \ref{eq:9}}\quad
\tbmx\left(\bmtau_{n}\right) \approx   L\,\hF{M}_{2} + \bmx_{0} \otimes \bmone_{n+1},
\label{eq:T2c}
\end{equation}
\end{subequations}
where 
\begin{equation}
\hF{M}_{2} = \text{vec}\left(\F{M}_{2}\right):\quad \F{M}_{2} = \F{Q}\left(\F{F} \oslash \left[ \left(\bmone_{n+1}-\bmtau_{n}\right) \otimes \bmone_{n_{x}}^{t}\right]\right).
\label{eq:NOTDIV2}
\end{equation}
We refer to the NLPs described by Eqs. \eqref{eq:T1}, \eqref{eq:T1c}, \eqref{eq:T2}, and \eqref{eq:T2c} by NLP1 and NLP2, respectively. We also refer to the present collocation method by the \textit{``GGR-IPS''} method; the acronyms \textit{``GGR-IPS1''} and \textit{``GGR-IPS2''} stand for the GGR-IPS method performed using the parametric maps $T_{1,L}^{(\alpha)}$ and $T_{2,L}^{(\alpha)}$, respectively, while \textit{``GGR-IPS12''} stands for the GGR-IPS method performed using either maps $T_{1,L}^{(\alpha)}$ and $T_{2,L}^{(\alpha)}$.

\section{Necessary Conditions of Optimality for the NLP}
\label{sec:NCOOFTNP}
Consider the NLP described by Eqs. \eqref{eq:9mf} and \eqref{eq:13}. The Lagrangian associated with the NLP is defined by
\begin{equation}
\MF{L} = \F{Q}_{n+1} \left(\bmT'_n \circ \bm{g}\right)+\bm{r}^{t}\left(\hF{M} + \bmx_{0} \otimes \bmone_{n+1} - \tbmx\left(\bmtau_{n}\right)\right),
\label{eq:NLPnec1}
\end{equation}
where $\bmT'_n = T'\left(\bmtau_{n}\right)$, and $\bm{r}=\left[r_{10},\ldots,r_{1n},\ldots,r_{n_{x}0},\ldots,r_{n_{x}n}\right]^{t}$ is the vector of Lagrange multipliers. Therefore, the KKT necessary conditions of optimality are given by
\begin{align}
\mathop {\nabla}\limits_{\tilde {\bmx}} \MF{L} &= \F{Q}_{n+1} \left[\left(\bmT'_n \otimes {\bmone}_{n_x n}^t\right) \circ \mathop {\nabla}\limits_{\tilde {\bmx}} \bm{g}\right] + \bm{r}^t \left[\left({\F{I}}_{n_x} \otimes \left[{\F{Q}} \circ \left({\bmT'_n}^t \otimes {\bmone}_{n+1}\right) \right]\right) \mathop {\nabla}\limits_{\tilde {\bmx}} \hF{F} - \F{I}_{n_x} \otimes \F{E}\right] = \bmzer,\label{eq:NLPnec2}\\
\mathop {\nabla}\limits_{\tilde {\bmu}} \MF{L} &= \F{Q}_{n+1}\left[\left( \bmT'_n \otimes \bmone_{n_{u}(n+1)}^{t} \right) \circ \mathop {\nabla}\limits_{\tilde {\bmu}} \bm{g}\right]+ \bm{r}^{t}  \left[\left({\F{I}}_{n_x} \otimes \left[{\F{Q}} \circ \left({\bmT'_n}^t \otimes {\bmone}_{n+1}\right) \right]\right) \mathop {\nabla}\limits_{\tilde {\bmu}} \hF{F}\right] = \bmzer,\label{eq:NLPnec3}
\end{align}
where the operators $\displaystyle{ \mathop {\nabla}\limits_{\tilde {\bmx}} =\left[ \frac{\partial }{\partial \tilde{x}_{11}}  \ldots \frac{\partial }{\partial \tilde{x}_{1n}} \ldots  \frac{\partial }{\partial \tilde{x}_{n_{x}1}} \ldots \frac{\partial }{\partial \tilde{x}_{n_{x}n}} \right], \mathop {\nabla}\limits_{\tilde {\bmu}} =\left[ \frac{\partial }{\partial \tilde{u}_{10}}  \ldots \frac{\partial }{\partial \tilde{u}_{1n}} \ldots  \frac{\partial }{\partial \tilde{u}_{n_{u}0}} \ldots \frac{\partial }{\partial \tilde{u}_{n_{u}n}} \right]}$, $\F{I}_{n}$ is the identity matrix of size $n, \hF{F} = \text{vec}\left(\F{F}\right), \F{E} = [\bm{\mathit{0}}_{n}^{t};\F{I}_{n}]$, and $[.;.]$ denotes the vertical matrix concatenation.

\section{Error and Convergence Analyses}
\label{sec:errb}
In this section we derive the truncation error bounds for Eqs. \eqref{eq:9mf} and \eqref{eq:13} and their convergence rates. 

\begin{thm} 
\label{therr0}
Let $f \in C^{n+1}[-1, 1)$ be approximated by a Gegenbauer interpolant $P_{n}f$ based upon the GGR points set $\MBS_n$ as defined by Eq. \eqref{eq:new3}. Then there exist some numbers $\xi_{i} \in (-1,1), i=0,\ldots,n,$ such that the truncation error of Approximation \eqref{eq:GIM6} is given by 
\begin{equation}
{}_fE_{n}\left(\tau_{i},\xi_{i}\right)=\frac{f^{(n+1)}\left(\xi_{i}\right)}{(n+1)!K_{n+1}^{(\alpha)}  } \int_{-1}^{\tau_{i}} \C{G}_{n+1}^{(\alpha)}(\tau)\, d\tau \quad \forall i,
\label{eqerrnew}
\end{equation}
where $K_{n}^{(\alpha)} = 2^{n-1} \displaystyle{\frac{\Gamma(n+\alpha) \Gamma(2 \alpha+1)}{\Gamma(n+2 \alpha)\Gamma(\alpha+1)}},\, n = 0, 1, \ldots$. 
\end{thm}
\begin{proof}
By definition, we can write
\begin{equation}\label{eqerr:2}
f(\tau) = \sum_{k=0}^{n} f_{k}\, \C{L}_{n,k}(\tau) + {}_fE_{n}\left(\tau,\xi \right),\quad \forall \tau \in [-1,1),
\end{equation}
for some $\xi \in (-1, 1)$, where ${}_fE_n$ is the interpolation truncation error at the GGR points given by
\begin{equation}\label{eq:29}
{}_fE_{n}\left(\tau,\xi \right)=\frac{f^{(n+1)}\left(\xi\right)}{(n+1)!}\prod_{k=0}^{n}\left(\tau-\tau_{k}\right).
\end{equation}
The proof is established by realizing that $\C{G}_{n+1}^{(\alpha)}(\tau) = K_{n+1}^{(\alpha)}\prod_{k=0}^{n}\left(\tau-\tau_{k}\right)$, and integrating Eq. \eqref{eqerr:2} on $\left[-1,\tau_{i} \right)\,\forall i$. 
\end{proof}
The following result is a direct corollary of Theorem \ref{therr0} by letting $\eta: [-1, 1) \to \MBR$ and $\psi_{k}: [-1, 1) \to \MBR$ such that $\eta(\tau)=T'(\tau) g(\tbmx(\tau),\tbmu(\tau))$ and $\psi_{k}(\tau)=T'(\tau) f_{k}\left( \tbmx(\tau),\tbmu(\tau)\right)$, for each $k = 1,\ldots, n_x$.

\begin{cor} 
\label{therr00}
The truncation errors of Approximation \eqref{eq:13} and each approximation of System \eqref{eq:9mf} for each equation 
\begin{equation}
{{\tilde x}_k}({\tau _j}) = \int_{ - 1}^{{\tau _j}} {{\psi _k}(\tau )\,d\tau }  + {x_{k,0}},
\label{eq:cor}
\end{equation}
of the Integral Constraints System \eqref{eq:8} at each point $\tau_j \in \MBS_n$ are given by
\begin{equation}
{}_\eta {E_n}\left(\zeta \right)=\frac{\eta^{(n+1)}(\zeta)}{(n+1)!K_{n+1}} \int_{-1}^{1}  \C{G}_{n+1}^{(\alpha)}(\tau) d\tau,   
\label{eq:errbounds4}
\end{equation} 
and
\begin{equation}
{}_{\psi_k}E_{n}\left(\tau_{j},\xi_{j}\right)=\frac{\psi^{(n+1)}_{k}\left(\xi_{j}\right)}{(n+1)!K_{n+1}}  \int_{-1}^{\tau_{j}} \C{G}_{n+1}^{(\alpha)}(\tau)\,d\tau,\quad k=1,\ldots,n_{x},\quad j=0,\ldots,n,   
\label{eq:errbounds1}
\end{equation} 
respectively, where $\zeta, \xi_{j} \in(-1,1)\,\forall j$.
\end{cor}

The following upper bounds on the truncation errors of Approximations \eqref{eq:9mf} and \eqref{eq:13} can be deduced from \citet[Theorem 5.1]{elgindy2018high1}.
\begin{thm}
\label{therr1}
Let $\psi_{k} \in C^{n+1}[-1,1)$ and $\left\| \psi_{k}^{(n + 1)} \right\|_{[-1,1)} = A_{\psi_k, n} \in \MB{R}^{+}\,\forall k$, for some constant $A_{\psi_k, n}$ dependent on $n$ and $k$. Then there exist some positive constants ${B_{1}^{(\alpha )}}$ and ${C_{1}^{(\alpha )}}$ dependent on $\alpha$ and independent of n such that the truncation errors of System \eqref{eq:9mf} at each point $\tau_j \in \MBS_n$ are bounded by the following inequalities:
\begin{equation}\label{eq:errbounds2}
\left|{}_{\psi_k}E_{n}\left(\tau_{j},\xi_{j}\right)\right| \leq \frac{A_{\psi_k, n} \Gamma{(n+2\alpha+1)} \Gamma{(\alpha+1)} \left(\tau_{j}+1\right) }{2^{n} (n+1)! \Gamma{(n+\alpha+1)} \Gamma{(2 \alpha+1)}}{\left\| \C{G}_{n+1}^{(\alpha)} \right\|_{[-1,1)}},\quad k = 1, \ldots, n_x,\quad j = 0, \ldots, n,
\end{equation}
where $\xi_j \in (-1,1)\,\forall j$, and
\begin{equation}\label{eq:errbounds3}
{\left\| \C{G}_{n+1}^{(\alpha)} \right\|_{[-1,1)}}=
\begin{cases}
2, \quad n \geq 0 ,\;\alpha \geq 0, \\
\displaystyle{\frac{\Gamma{\left(\alpha+\frac{1}{2}\right)} \Gamma{\left(\frac{n+1}{2}\right)}}{\sqrt{\pi} \Gamma{\left(\alpha+\frac{n+1}{2}\right)}}\left(1+\sqrt{\frac{n+1}{2\alpha+n+1}}\right)},\quad \frac{n}{2} \in \MB{Z}_0^ +  \wedge \frac{-1}{2}<\alpha < 0,\\
\displaystyle{\frac{{\Gamma \left( {\alpha  + \frac{1}{2}} \right)\left( {\sqrt {n\left( {2\alpha  + n} \right)}  + n} \right)\Gamma \left( {\frac{n}{2}} \right)}}{{2 \sqrt \pi \, \Gamma \left( {\frac{n}{2} + \alpha  + 1} \right)}}}
,\quad \frac{{n + 1}}{2} \in {\MB{Z}^ + } \wedge \frac{-1}{2}<\alpha < 0.
\end{cases}
\end{equation}
Moreover, 
\begin{equation}\label{ineq:asymineqalphanonneg1}
\left| {_{\psi_k} {E_n}\left( {{\tau _j},{\xi _j}} \right)} \right| \le B_1^{(\alpha )}{\left( {\frac{e}{2}} \right)^n}\frac{{1 + {\tau _j}}}{{{n^{n + \frac{3}{2} - \alpha }}}},\quad \forall \alpha  \ge 0,
\end{equation}
and
\begin{equation}\label{ineq:asymineqalphaneg1}
\left| {}_{\psi_k}E_{n}\left(\tau_{j},\xi_{j}\right) \right|\mathop  < \limits_ \sim  {C_{1}^{(\alpha )}}{\left( {\frac{e}{2}} \right)^n}\frac{{1 + \tau_{j}}}{{{n^{n + \frac{3}{2}}}}},\quad \forall - 1/2 < \alpha  < 0, 
\end{equation}
as $n \rightarrow \infty$, where $\mathop  < \limits_ \sim$ means \textit{``less than or asymptotically equal to.''} 
\end{thm}

\begin{thm}
\label{therr2}
Let $\eta \in C^{n+1}[-1,1)$ and $\left\| \eta^{(n + 1)} \right\|_{[-1,1)} = A_{\eta, n} \in \MB{R}^{+}$, for some constant $A_{\eta, n}$ dependent on $n$. Then there exist some positive constants ${B_{2}^{(\alpha )}}$ and ${C_{2}^{(\alpha )}}$ dependent on $\alpha$ and independent of n such that the truncation error of Eq. \eqref{eq:13} is bounded by the following inequality:
\begin{equation}\label{eq:errbounds2_1}
\left|{}_\eta {E_n}\left(\zeta\right)\right| \leq \frac{A_{\eta, n} \Gamma{(n+2\alpha+1)} \Gamma{(\alpha+1)} }{2^{n-1} (n+1)! \Gamma{(n+\alpha+1)} \Gamma{(2 \alpha+1)}}{\left\| \C{G}_{n+1}^{(\alpha)} \right\|_{[ - 1,1)}},
\end{equation}
where $\zeta \in (-1,1)$. Moreover, 
\begin{equation}
\left| {_\eta {E_n}\left( \zeta  \right)} \right| \le B_2^{(\alpha )}{\left( {\frac{e}{2}} \right)^n}\frac{1}{{{n^{n + \frac{3}{2} - \alpha }}}},\quad \forall \alpha \ge 0,
\end{equation}
and
\begin{equation}
\left| {}_\eta {E_n}\left(\zeta\right) \right|\mathop  < \limits_ \sim  C_{2}^{(\alpha )} {\left( {\frac{e}{2}} \right)^n}\frac{1}{{{n^{n + \frac{3}{2}}}}},\quad \forall - 1/2 < \alpha  < 0,
\end{equation}
as $n \rightarrow \infty$. 
\end{thm}

\subsection{Divergence of Typical IPS Collocation Schemes of the FHOCI at Any Large Mesh Grid of Gauss Type When $T = T_{1,L}^{(\alpha)}$ or $T = T_{2,L}^{(\alpha)}$}
\label{subsec:DOTPCS}
In this section we derive some striking results regarding the convergence of typical collocation schemes of the FHOCI described by Eqs. \eqref{eq:5}, \eqref{eq:7}, and \eqref{eq:8} when $T \in \{T_{1,L}^{(\alpha)}, T_{2,L}^{(\alpha)}\}$ and the mesh grid is large and of Gauss-type. While the proof is pertained to the FHOCI and it employs the GGR points as the collocation points, it can be generalized to the usual form of the FHOC described by Eqs. \eqref{eq:5}-\eqref{eq:7} and any large collocation points of Gauss type, whence it becomes of considerably greater interest. We derive these interesting divergence results in the following two corollaries.

\begin{cor}
\label{div}
Let $T \in \lbrace T_{1,L}^{(\alpha)}, T_{2,L}^{(\alpha)} \rbrace$, and suppose that $\exists\, \hk \in \{1, \ldots, n_x\}: \psi_{{\hk}} \in C^{n}[-1,1)$, $0 < \displaystyle{\left\| f_{{\hk}} \right\|_{[-1,1)}} < \infty$, and  $0 \le \left\| \displaystyle{\frac{d^{j}}{d \tau^{j}}}f_{{\hk}} \right\|_{[-1,1)} <\infty \quad \forall j=1,\ldots,n+1$, then the upper truncation error bounds of the Approximations System \eqref{eq:9mf} diverge at each collocation point as $n \rightarrow \infty$, for any map scaling parameter value $L$.
\end{cor}

\begin{proof}
By the General Leibniz Rule, the $(n+1)$st-derivative of $\psi_{{\hk}}$ is given by
\begin{equation}
\psi_{{\hk}}^{(n+1)}(\tau)=\sum_{j=0}^{n+1} \binom{n+1}{j} T^{(n+2-j)}(\tau) \frac{d^{j}}{d \tau^{j}} f_{{\hk}}\left(\tbmx(\tau),\tbmu(\tau)\right),
\label{eq:div1}
\end{equation}
whence, 
\begin{equation}
\left\| \psi_{{\hk}}^{(n+1)}\right\|_{[-1,1)}=\sum_{j=0}^{n+1} \binom{n+1}{j} \left\|T^{(n+2-j)}\right\|_{[-1,1)} \left\|\frac{d^{j}}{d \tau^{j}} f_{{\hk}}\right\|_{[-1,1)}.
\label{eq:div2}
\end{equation}
Let $T=T_{1,L}^{(\alpha)}$, and notice that $\left(T_{1,L}^{(\alpha)}\right)^{(m)}(\tau)=\displaystyle{\frac{2L(m)!}{(1-\tau)^{m+1}}}\, \forall m \in \MB{Z}^{+}$, which is a monotonically increasing function for increasing values of $\tau$ as clearly seen from Figure \ref{fig1div}. Therefore, $A_{\psi_{\hk}, n} = O\left( \left\|\left(T_{1,L}^{(\alpha)}\right)^{(n+2)}\right\|_{[-1,1)} \right)$. From Theorem \ref{therr1}, there exist some positive constants $\hat B_{1}^{(\alpha)}$ and $\hat C_{1}^{(\alpha)}$ dependent on $\alpha$ and independent on $n$ such that
\begin{subequations}
\begin{equation}\label{ineq:asymineq1}
\left| {{}_{\psi_{\hk}} {E_n}\left(\tau_{j},\xi_{j}\right)} \right| \le L {\hat B_{1}^{(\alpha)}} (n+2)!\,\left(\frac{e}{2}\right)^n \frac{1+\tau_{j}}{{{n^{n + \frac{3}{2} - \alpha }}}}\left\| (1-\tau)^{-n-3}\right\|_{[-1,1)},\quad \forall \alpha \ge 0,
\end{equation}
and
\begin{equation}\label{ineq:asymineq2}
\left| {{}_{\psi_{\hk}} {E_n}\left(\tau_{j},\xi_{j}\right)} \right| \mathop  < \limits_ \sim L {\hat C_{1}^{(\alpha)}} (n+2)!\,\left(\frac{e}{2}\right)^n \frac{1+\tau_{j}}{{n^{n + \frac{3}{2}}}}\left\| (1-\tau)^{-n-3}\right\|_{[-1,1)},\quad \forall -1/2 < \alpha  < 0,
\end{equation}
\end{subequations}
whence we realize that the upper bound of $\left|{}_{\psi_{\hk}} {E_n}\right|$ at each collocation point $\tau_{j}$ diverges as $n\rightarrow \infty$. Consider now the case when $T = T_{2,L}^{(\alpha)}$. By a similar argument, notice first that $\left(T_{2,L}^{(\alpha)}\right)^{(m)}(\tau) = \displaystyle{\frac{L(m-1)!}{(1-\tau)^{m}}}\; \forall m \in \MB{Z}^{+}$ is also a monotonically increasing function for increasing values of $\tau$ as shown by Figure \ref{fig1div}. Therefore, $A_{\psi_{\hk}, n} = O\left( \left\|\left(T_{2,L}^{(\alpha)}\right)^{(n+2)}\right\|_{[-1,1)} \right)$. From Theorem \ref{therr1}, there exist some positive constant $\hat B_{2}^{(\alpha)}$ and $\hat C_{2}^{(\alpha)}$ dependent on $\alpha$ and independent on $n$ such that  
\begin{subequations}
\begin{equation}\label{ineq:asymineq3}
\left| {{}_{\psi_{\hk}} {E_n}\left(\tau_{j},\xi_{j}\right)} \right| \le L \hat B_2^{(\alpha )} (n + 1)!\,{\left( {\frac{e}{2}} \right)^n}\frac{{1 + {\tau _j}}}{{{n^{n + \frac{3}{2} - \alpha }}}}{\left\| {{{(1 - \tau )}^{ - n - 2}}} \right\|_{[ - 1,1)}},\quad \forall \alpha  \ge 0,
\end{equation}
and
\begin{equation}\label{ineq:asymineq4}
\left| {{}_{\psi_{\hk}} {E_n}\left(\tau_{j},\xi_{j}\right)} \right| \mathop  < \limits_ \sim L \hat C_2^{(\alpha )}(n + 1)!\,{\left( {\frac{e}{2}} \right)^n} \frac{{1 + {\tau _j}}}{{{n^{n + \frac{3}{2}}}}}{\left\| {{{(1 - \tau )}^{ - n - 2}}} \right\|_{[ - 1,1)}},\quad \forall -1/2 < \alpha  < 0,
\end{equation}
\end{subequations}
from which we observe that the upper bound of $\left|{}_{\psi_{\hk}} {E_n}\right|$ at each collocation point $\tau_j$ diverges as $n \rightarrow \infty$.
\end{proof}

\begin{figure}[h]
\centering
\includegraphics[scale=0.5]{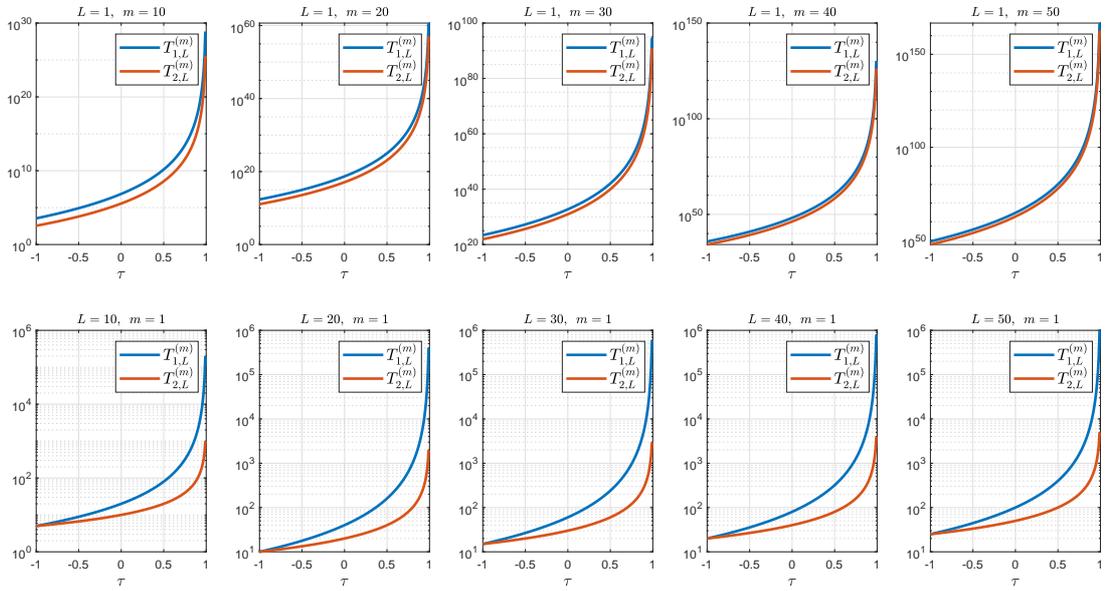}
\caption{The $m$th-order derivatives of $T_{1,L}^{(\alpha)}$ and $T_{2,L}^{(\alpha)}$ versus $\tau$ in log-lin scale for several values of $L$ and $m$. The superscript of $T_{i,L}^{(\alpha)}, i=1,2$ has been omitted in the plots.}
\label{fig1div}
\end{figure}

Theorem \ref{therr1} and Corollary \ref{div} reveal an interesting fact: Under their assumptions, the proposed method is expected to converge with an exponential rate to near-optimal solutions for increasing $n$-values within a relatively small $n$-values range as indicated by Inequality \eqref{eq:errbounds2}, but as $n$ grows large, the constant $A_{\psi_{\hk}, n}$ grows exponentially fast and ultimately dominates the error bounds when $n \to \infty$, as implied by Inequalities (\ref{ineq:asymineq1})-(\ref{ineq:asymineq4}), regardless of how well we choose the map scaling parameter value $L$. In fact, the asymptotic results of Corollary \ref{div} manifest that for increasing large $n$-values, reducing the $L$-value abates the divergence of the approximations at the outset, but as $n$ grows larger, this approach fails to cope with the soaring values of $n$ powers of the factors $1/(1-\tau_j)$ at the mesh points $\tau_j$, for sufficiently close mesh points values to $\tau = 1$; thus, divergence is inevitable! 

While the above forward error analysis may be too pessimistic and may reject solutions that are sufficiently accurate, another concern arise when we analyze the sensitivity of NLP1 and NLP2 associated with the maps $T_{1,L}^{(\alpha)}$ and $T_{2,L}^{(\alpha)}$ to input data errors. Observe that both problems require the computations of the maps $T'_{1,L}$ and $T'_{2,L}$ which are ill-conditioned for arguments near $1$. In particular, suppose that $\tau \approx 1$ with a small perturbation $h$ to $\tau$. Then the absolute errors in computing $T'_{1,L}(\tau)$ and $T'_{2,L}(\tau)$ are given by 
\[\left|T'_{1,L}(\tau+h) - T'_{1,L}(\tau)\right| \approx \frac{4 L\,|h|}{(1-\tau)^3}\quad \text{and}\quad \left|T'_{2,L}(\tau+h) - T'_{2,L}(\tau)\right| \approx \frac{L\,|h|}{(1-\tau)^2},\]
and hence the relative errors are $\displaystyle{\frac{2 |h|}{1-\tau}}$ and $\displaystyle{\frac{|h|}{1-\tau}}$, respectively, which blow up as $\tau \to 1$. Recall that GGR points cluster near $\pm 1$ as $n \to \infty$, so the sensitivity of the problem of calculating the maps derivative functions $T'_{1,L}$ and $T'_{2,L}$ at arguments near $1$ increases for increasing values of $n$. For example, let $\tau = 0.9999999999999$ be an exact argument value and consider its approximation $\hat \tau = 0.9999999999998$ with a small perturbation of about $9.99 \times 10^{-14}$ to $\tau$. Then the relative error in the input value is about $10^{-11} \%$. However, the relative errors in computing $T'_{1,L}(\tau)$ and $T'_{2,L}(\tau)$ are nearly $75 \%$ and $50 \%$. Hence, the relative changes in evaluating $T'_{1,L}(\tau)$ and $T'_{2,L}(\tau)$ are about $7.5$ and $5$ trillion times larger than the relative change in the input value in respective order! This example shows that increasing the mesh size shifts the positive collocation points closer and closer towards $\tau = 1$ and wild ill-conditioning ultimately rears its ugly head, as the sensitivity of NLP2 progressively stiffens for arguments near $1$. Therefore, one should keep in mind that reducing $L$ may still improve the approximations for a certain range of $n$-values, nonetheless this strategy is not susceptible to produce accurate approximations for relatively large values of $n$, in general, since both NLP1 and NLP2 are ill-conditioned near $\tau = 1$. It is noteworthy to mention here that this sensitivity of NLP1 and NLP2 near $\tau = 1$ is foreseen to relax or disappear if $g$ and $f_k\,\forall k$ decay exponentially fast such that $\mathop {\lim }\nolimits_{\tau  \to 1} {T'(\tau) g\left(\tbmx(\tau),\tbmu(\tau)\right) } = 0$ and $\mathop {\lim }\nolimits_{\tau  \to 1} {T'(\tau) \bmf(\tbmx(\tau),\tbmu(\tau))} = \bmzer$. %with rates faster than the growth rates of $T'_{1,L}$ and $T'_{2,L}$ as $t \to \infty$, respectively. 
Under a similar proof to that of Corollary \ref{div}, one can derive the following second divergence result.

\begin{cor}
\label{div2}
Let $T \in \lbrace T_{1,L}^{(\alpha)}, T_{2,L}^{(\alpha)} \rbrace$, $\eta \in C^{n}[-1,1)$, $0 < \displaystyle{\left\| g \right\|_{[-1,1)}} < \infty$, and  $0 \le \left\| \displaystyle{\frac{d^{j}}{d \tau^{j}}}g \right\|_{[-1,1)} <\infty\,\forall j = 1,\ldots,n+1$, then the upper truncation error bound of Approximation \eqref{eq:13} diverges as $n \rightarrow \infty$, for any map scaling parameter value $L$.
\end{cor}

The present analysis begs another interesting question: Which map should we use if we desire to implement the proposed method? For small/medium range of $n$-values, the answer is a bit elusive; however, for large values, it seems we have a crystal clear answer as shown by the following corollary.
\begin{cor}\label{cor:GDivnew1}
A Gegenbauer-Gauss collocation of the FHOCI using the map $T_{1,L}^{(\alpha)}$ generally diverges faster than applying the method lumped with the map $T_{2,L}^{(\alpha)}$ when $n  \to \infty$.
\end{cor}
\begin{proof}
The faster divergence exhibited using the map $T_{1,L}^{(\alpha)}$ compared with $T_{2,L}^{(\alpha)}$ as $n  \to \infty$ can be easily justified through Inequalities \eqref{ineq:asymineq1}-\eqref{ineq:asymineq4}. Moreover, since $T'_{1,L}(\tau)$ grows faster than $T'_{2,L}(\tau)$ by a factor of $2/(1-\tau)$, which blows up as $\tau \to 1$, the ill-conditioning of $T'_{1,L}$ is clearly more severe than that of $T'_{2,L}$ for values of $\tau \approx 1$. 
\end{proof}
Corollary \ref{cor:GDivnew1} manifests that the map $T_{2,L}^{(\alpha)}$ is more likely a better choice than $T_{1,L}^{(\alpha)}$ for large $n$ values. We end this section by drawing the attention of the reader to the fact that integral reformulations of various mathematical models have received considerable attention in the literature because they often produce well-conditioned linear systems. While numerical quadratures and integration matrices are generally more stable than numerical differentiation operators and matrices, there is no strong reason to expect that standard PS collocations of the FHOC in its strong differential-form using a single mesh grid of Gauss-type and maps like $T_{1,L}^{(\alpha)}$ and $T_{2,L}^{(\alpha)}$ would exhibit any merits over the current method, and they would ultimately diverge for a large mesh grid size. These considerations lead naturally to the following interesting conjecture.
\begin{conj}
Classical Jacobi polynomial collocations of the FHOC in differential/integral-form obtained through maps like $T_{1,L}^{(\alpha)}$ and $T_{2,L}^{(\alpha)}$ will likely diverge as the mesh size grows large, if the computations are carried out using floating-point arithmetic and the discretizations use a single mesh grid whether they are of Gauss/Gauss-Radau (GR) type or equally-spaced. The former divergence case is a direct result of the present divergence analysis, while the latter case is due to Runge's phenomenon and the ill-conditioning of polynomial interpolation at equally-spaced nodes as the degree of the polynomial grows.
\end{conj}

\section{Numerical Experiments}
\label{sec:NE1}
This section presents the results of some numerical experiments on two test examples which demonstrate the accuracy and efficiency of the proposed GGR-IPS12 methods for small/medium range mesh grid sizes, and verifies the inevitable divergence as the mesh size grows large. All numerical experiments were carried out using MATLAB R2022a software installed on a personal laptop equipped with a 2.9 GHz AMD Ryzen 7 4800H CPU and 16 GB memory running on a 64-bit Windows 11 operating system. The NLPs obtained through the GGR-IPS12 methods were solved using either (i) MATLAB fmincon solver with interior-point algorithm (fmincon-int) and sqp algorithm (fmincon-sqp), or (ii) the augmented Lagrange multiplier method \cite{hestenes1969multiplier,powell1969method} integrated with a modified BFGS method and a Chebyshev PS line search method (MBFGS-CPSLSM) \cite{elgindy2018optimization}; henceforth, referred to by the \textit{``EALMM.''} It should be clearly understood by the reader when we coin the name of the current collocation method with any NLP solver that we are implementing them both to solve the FHOCI. For example, the acronym GGR-IPS12-EALMM stands for the GGR-IPS12 methods combined with the EALMM. In all numerical tests, the exact optimal state and control variables were calculated using MATLAB with $30$ digits of precision maintained in internal computations. The fmincon solver was carried out using the stopping criteria TolFun $=$ TolX $= 10^{-12}, 10^{-15}$ for Examples 1 and 2, respectively; similarly, the augmented Lagrange multiplier method was terminated when the lower bound on the change in the augmented Lagrangian function value during a step does not exceed $10^{-12}$. All experiments were conducted using the parameters values $L \in \{0.25(0.25)10\}$ and $\alpha \in \{-0.4(0.1)2\}$. Most of the numerical simulations were performed using the two sets of initial guesses $\F{\Omega}_1 = \{(\tbmx_{0}, \tbmu_{0}): \tbmx_{0}= \bmone_{n_x}, \tbmu_{0} = \bmone_{n_u}\}$ and $\F{\Omega}_2 = \{(\tbmx_{0}, \tbmu_{0}): \tbmx_{0}= 0.5 \bmone_{n_x}, \tbmu_{0} = 0.5 \bmone_{n_u}\}$; henceforth, $\F{\Omega}_1 \cup \F{\Omega}_2$ is denoted by $\F{\Omega}$. Furthermore, by $AE_{J}$ and $MAE_{x,u}$, we mean the absolute error in the objective function value and the maximum absolute error of the state and control variables in respective order.\\

\noindent \textbf{Example 1.} Consider the IHOC \eqref{eq:2}–\eqref{eq:1} with $g\left(\bmx(t),\bmu(t)\right)=\left(\ln^{2} x(t) + u^{2}(t)\right)/2, f\left(\bmx(t),\bmu(t)\right)=x(t) \ln{x(t)}+x(t) u(t)$, and $\bmx_{0}=2$. The exact state and control variables are
\begin{subequations}
\begin{align}
x^{*}(t) &= \exp\left({y^{*}(t)}\right),\label{eq:exact_state}\\
u^{*}(t) &= -\left(1+\sqrt{2}\right) y^{*}(t),\label{eq:exact_control}
\end{align}
 where
\begin{equation}
 y^{*}(t) = \left(\ln{2}\right) \exp\left({-\sqrt{2} t}\right),
\label{eq:where}
\end{equation}
\end{subequations}
cf. \cite{garg2011pseudospectral}. The exact cost function ${J^*} = {(\ln 2)^2}{\mkern 1mu} \left( {\sqrt 2  + 1} \right)/2 \approx 0.5799580911421756$ rounded to $16$ significant digits. Through the change of variables 
\begin{equation}\label{eq:clever1}
z(t) = \ln{x(t)}, 
\end{equation}
the IHOC \eqref{eq:2}–\eqref{eq:1} can be rewritten in an equivalent linear-quadratic optimal control problem in an infinite horizon with $g\left(\bmz(t),\bmu(t)\right) = \left(z^2(t) + u^{2}(t)\right)/2, f\left(\bmz(t),\bmu(t)\right)=z(t) +u(t)$, and $\bmz_{0} = \ln 2$. We refer to the former and latter forms of the IHOC by Forms A and B, respectively. Form A of the example was previously solved by \citet{garg2011pseudospectral} using LG- and LGR-PS methods and the three maps \eqref{eq:fahroo2008p} and \eqref{garg2011a}; the obtained NLPs were solved using SNOPT \cite{doiS1052623499350013,gill2005snopt}. Table \ref{table6} shows a comparison between the LGR- and LG-PS methods and the GGR-IPS12-EALMM using the same initial guesses $\tx(\tau) = 2$ and $\tu(\tau) = \tau\,\forall \tau \in [-1, 1]$. Notice how the GGR-IPS2-EALMM generally enjoy superior stability properties and achieve higher-order approximations in this example for $n = 5(5)30$ compared with the other approaches, except for the LG-PS method, where they both achieve the same order of accuracy at $n = 30$. It is interesting here to recognize how the GGR-IPS2-EALMM defeats the LG-PS method for $n = 5(5)25$, although the latter employs a Gauss quadrature that is more accurate than the GR quadrature used by the former. One may connect the success of the GGR-IPS2-EALMM here to many reasons, namely (i) the clever change of variables \eqref{eq:clever1} that converts the NLP into a linear-quadratic optimal control problem which can be collocated more accurately, (ii) the integral form of the system dynamics allows for gaining more digits of accuracy via numerical quadratures which are well-known for their numerical stability, (iii) the highly-accurate built-in Algorithm \ref{algorithm:2} to the current methods which applies the latest technology of SR-interpolation, (iv) the parametric logarithmic map $T_{2,L}^{(\alpha)}$ that is favored over $T_{1,L}^{(\alpha)}$ for its slower growth and less sensitivity near $\tau = 1$, and (v) the map scaling parameter $L$ which permits for faster convergence rates when \textit{``optimally''} chosen. On the other hand, we observe that the errors of GGR-IPS12-EALMM generally decline gradually as the mesh grid size initially grow up to a certain limit, yet they bounce back beyond that limit as the mesh grid size continues to grow large in agreement with the theoretical results of Section \ref{sec:errb}. It is interesting to see similar phenomena with the control error profiles in \cite{garg2011pseudospectral} in the sense that (i) the control error plot of the LG-PS method does not appear as a (near) straight line in the shown log-scaled chart but rather a convex-shaped curve, as it curves outward; cf. \citet[Fig. 1(b)]{garg2011pseudospectral}, and (ii) the control error plot of the LGR-PS method suddenly increases at $n = 30$ much earlier before reaching the round-off plateau; cf. \citet[Fig. 2(b)]{garg2011pseudospectral}. Another interesting remark lies in the smallest errors of the current methods; they were all recorded at/near $\alpha = 0.5$ with $\alpha \in \Upsilon_{0.5,0.6}^G$, while several optimal values of $L$ were detected. Figures \ref{fig1} and \ref{fig2} show the plots of the exact state and control variables in addition to their collocated solutions and absolute errors obtained by the GGR-IPS12 methods integrated with three NLP solvers using the same initial guesses set and several values of $n, L$, and $\alpha$. We can observe from the shown graphical data that the GGR-IPS2-EALMM generally achieves better accuracy and stability properties compared with the other methods. 

\begin{table}[h]
\centering
\caption{The uncertainty intervals of the smallest $MAE_{x,u}$ obtained by the LGR- and LG-PS methods of \citet{garg2011pseudospectral} at the collocation points and the corresponding smallest $MAE_{x,u}$ obtained by the GGR-IPS12-EALMM using the same initial guesses $\tx(\tau) = 2$ and $\tu(\tau) = \tau\,\forall \tau \in [-1, 1]$. All approximations were rounded to $5$ significant digits.} 
\begin{tabular}{*{5}{c}}
\toprule
& LGR-PS \cite{garg2011pseudospectral} & LG-PS \cite{garg2011pseudospectral} & GGR-IPS1-EALMM  & GGR-IPS2-EALMM \\
& \multicolumn{2}{c}{Form A} & \multicolumn{2}{c}{Form B} \\
$n$ & \multicolumn{2}{c}{$MAE_{x,u}$ uncertainty interval} & $MAE_{x,u}/\alpha/L$ & $MAE_{x,u}/\alpha/L$\\
\midrule
5   &   (1e-03, 1e-02) & (1e-04, 1e-03) & 5.3830e-03/0.6/2.25 & 4.2453e-05/0.5/3.5 \\
10  &  (1e-06, 1e-05) & (1e-06, 1e-05) & 7.0615e-05/0.5/5.75 & 1.8735e-09/0.5/4.25 \\
15  &  (1e-07, 1e-06) & (1e-07, 1e-06) & 6.0288e-07/0.5/9.25 & 1.9736e-09/0.5/5 \\
20  &  (1e-08, 1e-07) & (1e-08, 1e-07) & 1.3181e-08/0.5/8.5 & 2.0583e-09/0.5/2.5 \\
25  &  (1e-08, 1e-07) & (1e-08, 1e-07) &  2.4368e-08/0.5/2.25 & 1.6175e-09/0.5/3 \\  
30  &  (1e-08, 1e-07) & (1e-09, 1e-08) & 2.7958e-08/0.5/2.75 & 3.6927e-09/0.5/2 \\  
\bottomrule
\end{tabular}
\label{table6}   
\end{table}

\begin{figure}[ht]
\centering
\includegraphics[scale=0.35]{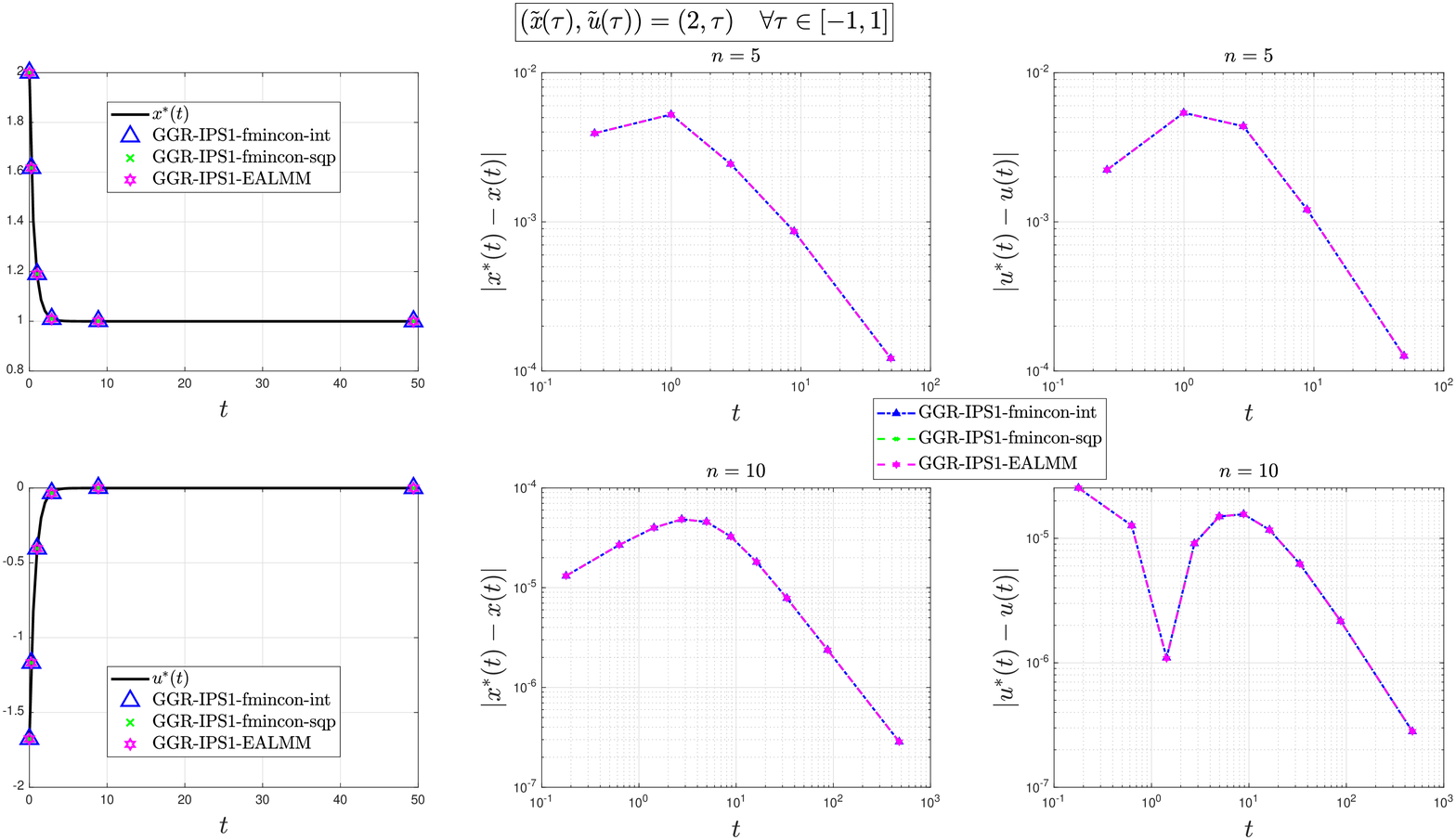}
\caption{The first column shows the plots of the exact optimal state and control variables of Example $1$ and the collocated solutions obtained by GGR-IPS1 method integrated with three distinct NLP solvers at the collocation points using $n = 5$ and the same initial guesses $\tx(\tau) = 2$ and $\tu(\tau) = \tau\,\forall \tau \in [-1, 1]$. The exact optimal state and control plots were generated using $101$ linearly spaced nodes from $0$ to $50$. The middle and last columns show the corresponding plots of the absolute errors of the state and control variables in log-log scale using $n = 5, 10$ and $(L,\alpha)$ ordered pairs as shown in Table \ref{table6}.}
\label{fig1}
\end{figure}

\begin{figure}[ht]
\centering
\includegraphics[scale=0.35]{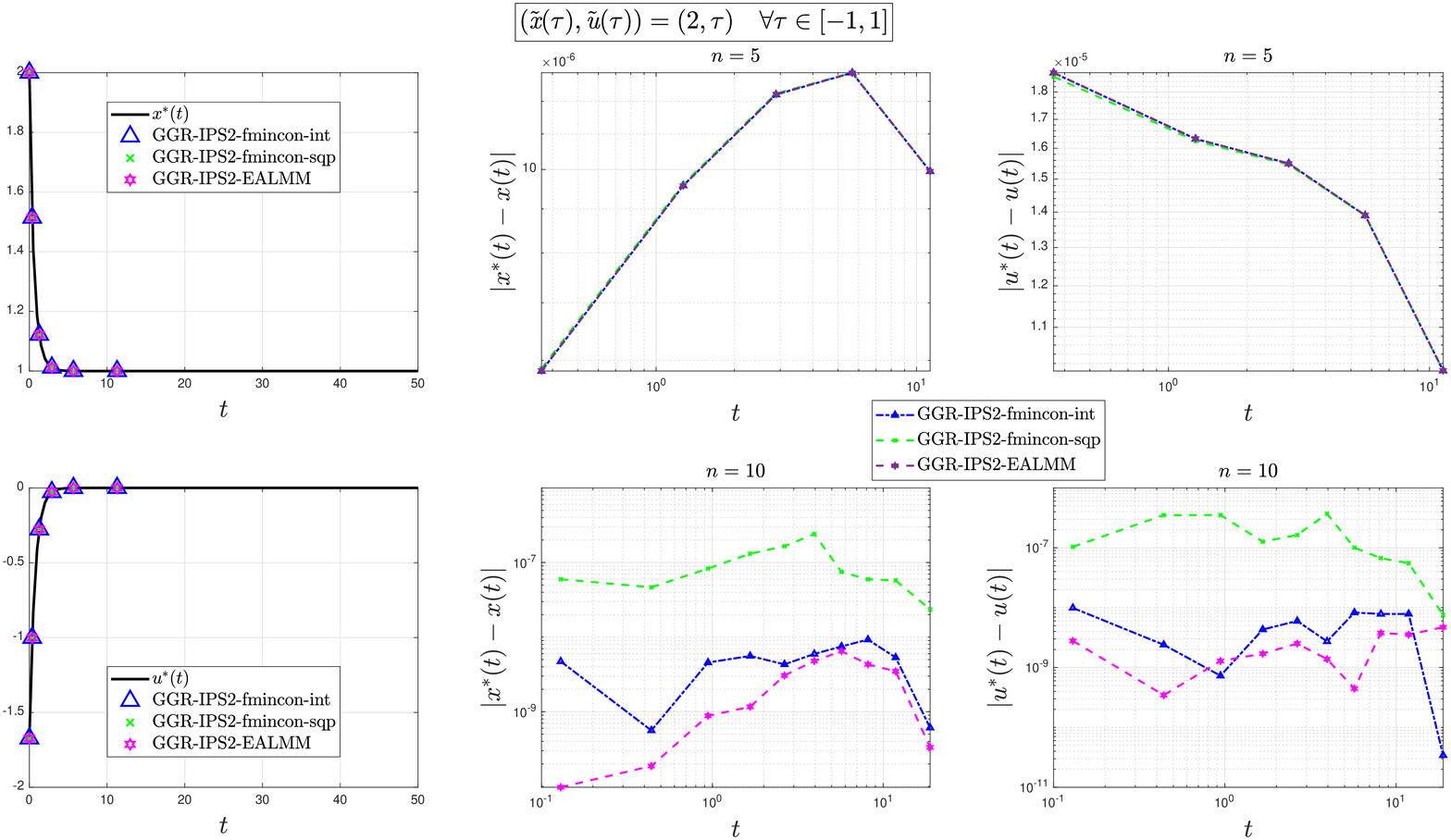}
\caption{The first column shows the plots of the exact optimal state and control variables of Example $1$ and the collocated solutions obtained by GGR-IPS2 method integrated with three distinct NLP solvers at the collocation points using $n = 5$ and the same initial guesses $\tx(\tau) = 2$ and $\tu(\tau) = \tau\,\forall \tau \in [-1, 1]$. The exact optimal state and control plots were generated using $101$ linearly spaced nodes from $0$ to $50$. The middle and last columns show the corresponding plots of the absolute errors of the state and control variables in log-log scale using $n = 5, 10$ and $(L,\alpha)$ ordered pairs as shown in Table \ref{table6}.}
\label{fig2}
\end{figure}

Figures \ref{fig3} and \ref{fig4} show the $MAE_{x,u}$ of the GGR-IPS1-EALMM using $n=10(10)50, \alpha=-0.4(0.1)2, L=1(1)4$, and $(\tbmx_{0}, \tbmu_{0}) \in \F{\Omega}$. It is interesting to observe here by visual inspection that, when holding $L$ fixed, the global minima of the error mesh surface plots occur near $\alpha = 0.5$ and the mesh surfaces rise up gradually as we move away, except when $\alpha \in \{-0.4,-0.3\}$, where sharp peaks may emerge suddenly for growing values of $n$. This suggests that Legendre polynomials seem an optimal choice among Gegenbauer basis polynomials when holding $L$ fixed, while Gegenbauer polynomials associated with $\alpha$-values near $-0.5$ may cause numerical instability as $n$ grows large. However, a different story emerges when the GGR-IPS2-EALMM is performed instead as can be seen from Figures \ref{fig3trans2} and \ref{fig4trans2}. Notice now that the errors \textit{``look''} monotonically decreasing for decreasing values of $\alpha$ when holding $L$ fixed at $1$ and $2$, and the error surface shoots up as $n$ and $\alpha$ increases. Therefore, Gegenbauer polynomials with some negative $\alpha$-values \textit{``seem''} optimal for relatively small values of $L$. Notice also that the sudden peaks observed before with the GGR-IPS1-EALMM in Figures \ref{fig3} and \ref{fig4} for $\alpha \in \{-0.4, -0.3\}$ and large $n$ values disappear. A further array of error mesh surface plots of the GGR-IPS1-EALMM are shown in Figures \ref{fig5} and \ref{fig5new} for $n=10(10)50, \alpha=-0.2,0,0.5,1, L=0.5(0.5)6$, and $(\tbmx_{0}, \tbmu_{0}) \in \F{\Omega}$. While holding $\alpha$ fixed, there seems no general rule of thumb can be laid down from the shown data. On the other hand, Figures \ref{fig5trans2} and \ref{fig5newtrans2} show the corresponding plots associated with the GGR-IPS2-EALMM, where the errors are very similar and can be clearly seen to surge as $L \to 0$, especially when $\alpha = 1$, but remain relatively small for $L  = 2(0.5)6$.

\begin{figure}[h]
\centering
\includegraphics[scale=0.37]{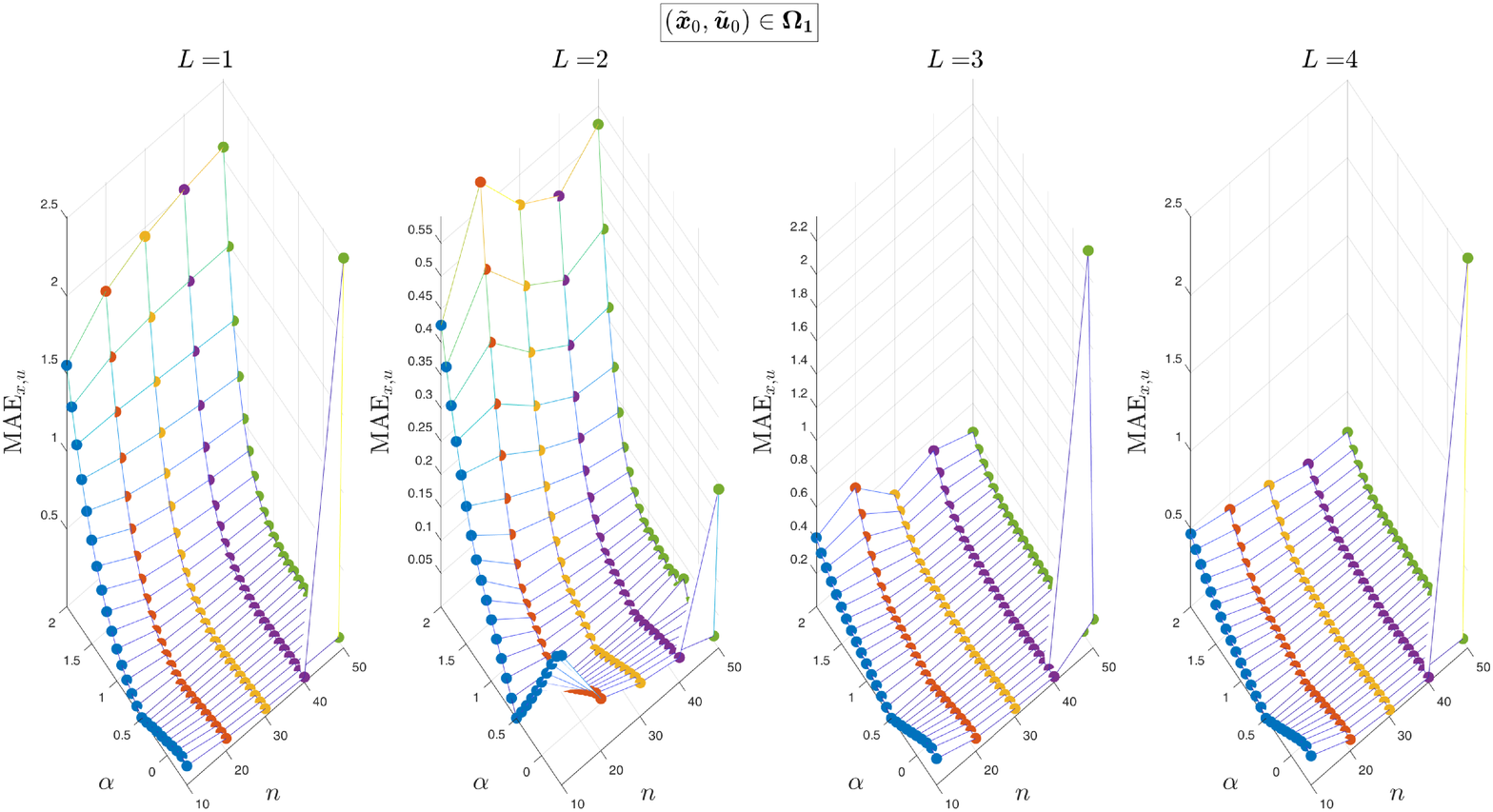}
\caption{The $MAE_{x,u}$ of the GGR-IPS1-EALMM at $101$ linearly spaced nodes between $0$ and $10$ using $n=10(10)50, \alpha=-0.4(0.1)2, L=1(1)4$, and $(\tbmx_{0}, \tbmu_{0}) \in \F{\Omega}_1$.} 
\label{fig3}
\end{figure}

\begin{figure}[h]
\centering
\includegraphics[scale=0.37]{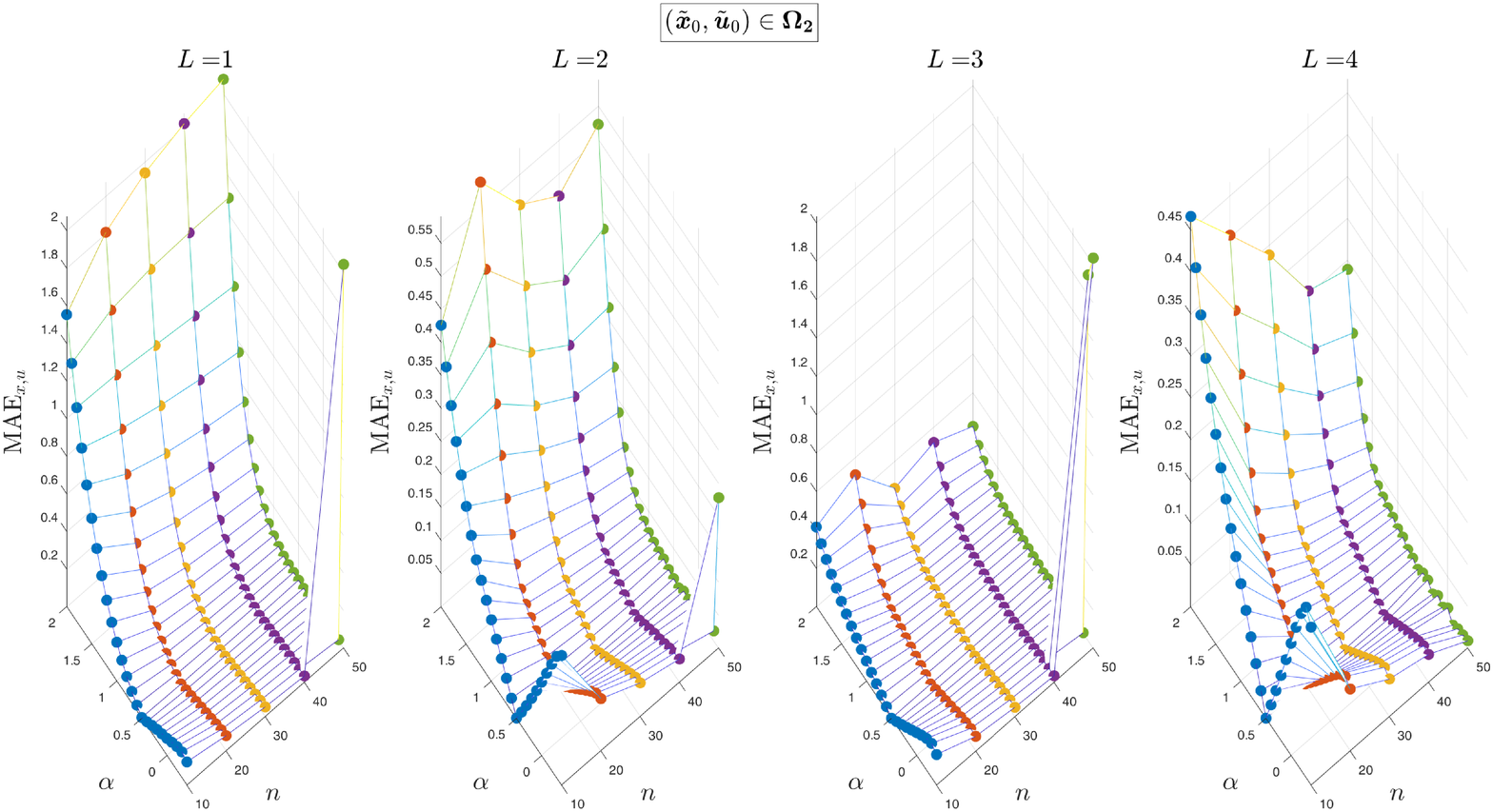}
\caption{The $MAE_{x,u}$ of the GGR-IPS1-EALMM at $101$ linearly spaced nodes between $0$ and $10$ using $n=10(10)50, \alpha=-0.4(0.1)2, L=1(1)4$, and $(\tbmx_{0}, \tbmu_{0}) \in \F{\Omega}_2$.} 
\label{fig4}
\end{figure}

\begin{figure}[h]
\centering
\includegraphics[scale=0.37]{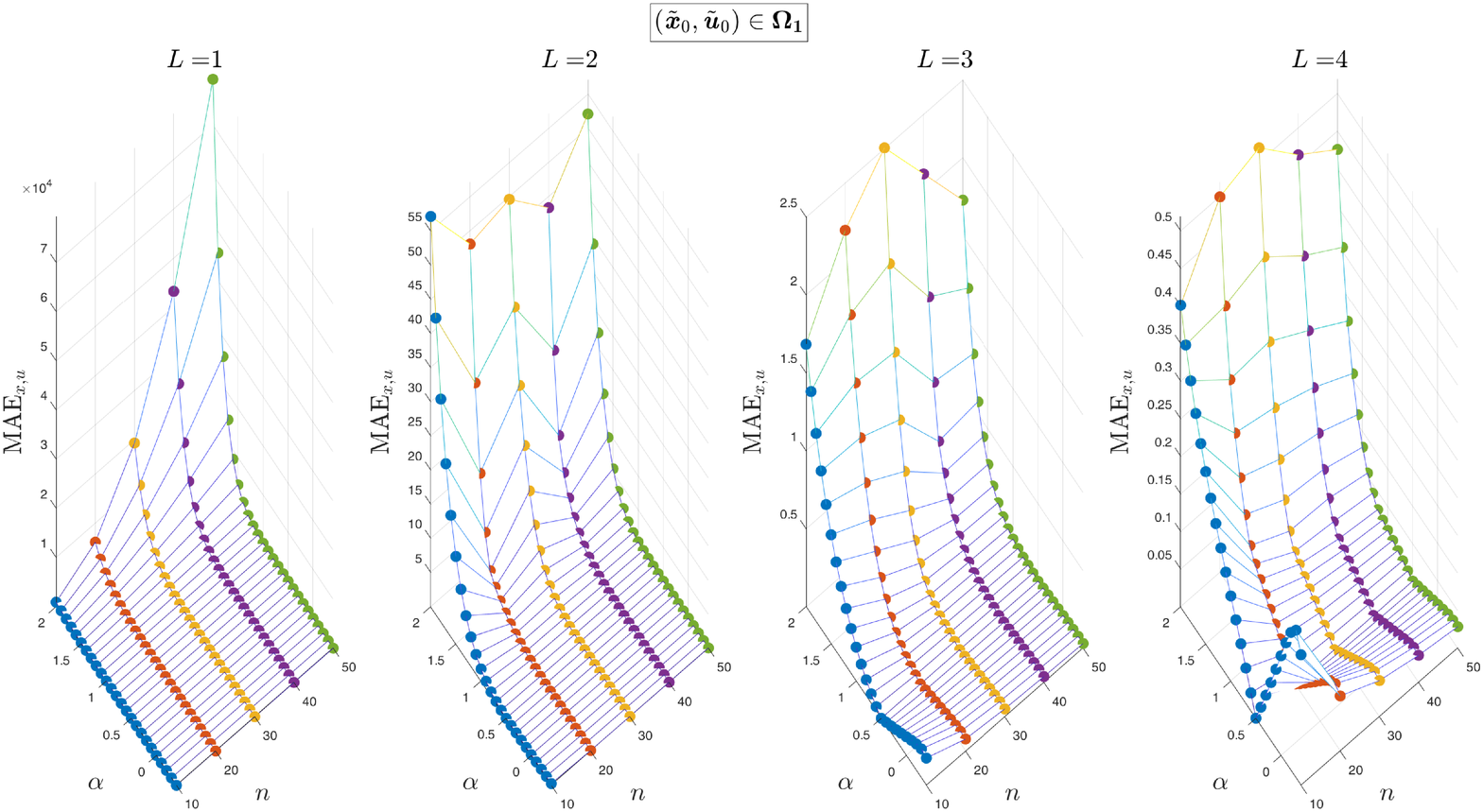}
\caption{The $MAE_{x,u}$ of the GGR-IPS2-EALMM at $101$ linearly spaced nodes between $0$ and $10$ using $n=10(10)50, \alpha=-0.4(0.1)2, L=1(1)4$, and $(\tbmx_{0}, \tbmu_{0}) \in \F{\Omega}_1$.} 
\label{fig3trans2}
\end{figure}

\begin{figure}[h]
\centering
\includegraphics[scale=0.37]{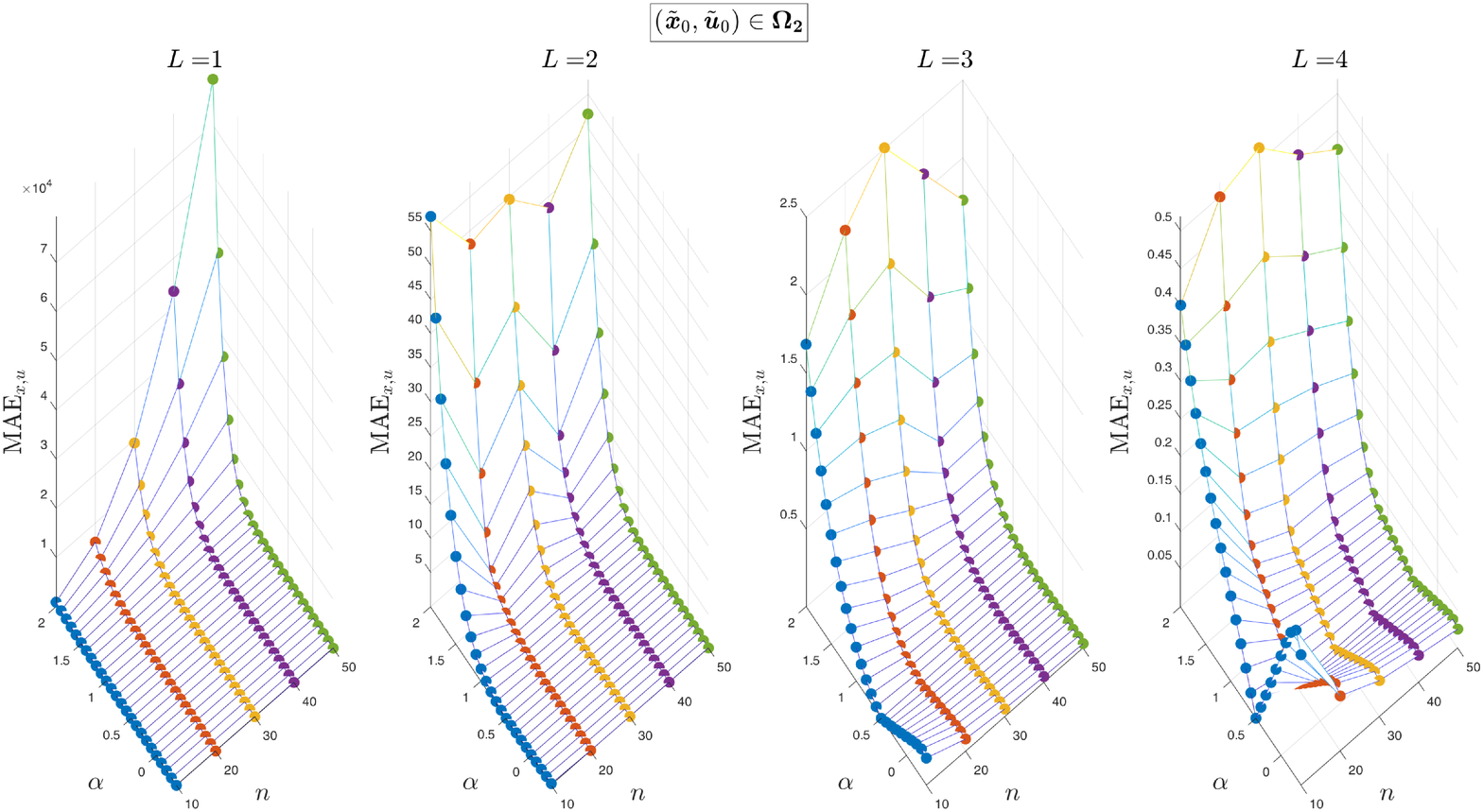}
\caption{The $MAE_{x,u}$ of the GGR-IPS2-EALMM at $101$ linearly spaced nodes between $0$ and $10$ using $n=10(10)50, \alpha=-0.4(0.1)2, L=1(1)4$, and $(\tbmx_{0}, \tbmu_{0}) \in \F{\Omega}_2$.} 
\label{fig4trans2}
\end{figure}

\begin{figure}[h]
\centering
\includegraphics[scale=0.37]{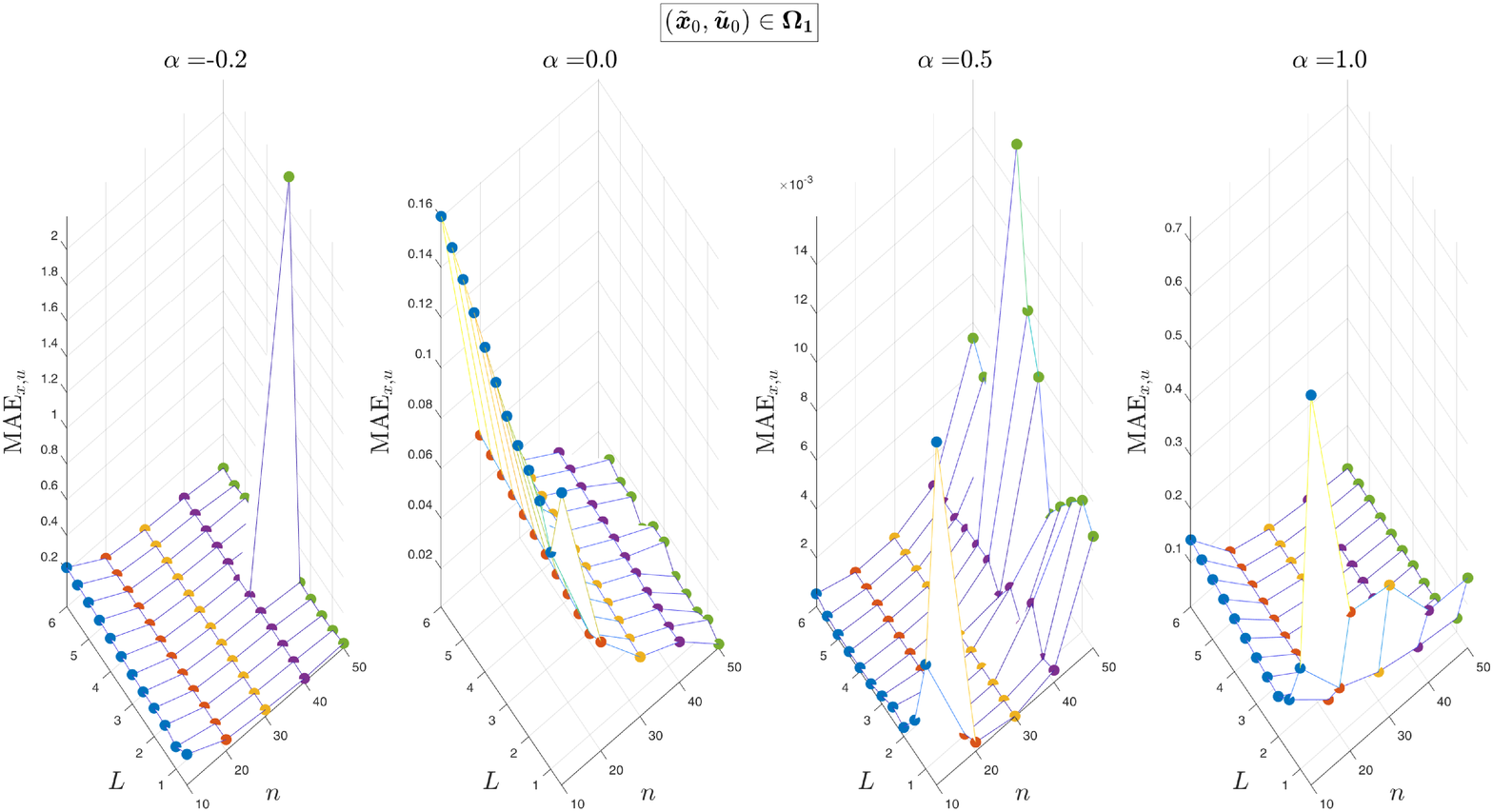}
\caption{The $MAE_{x,u}$ of the GGR-IPS1-EALMM at $101$ linearly spaced nodes between $0$ and $10$ using $n=10(10)50, \alpha=-0.2,0,0.5,1, L=0.5(0.5)6$, and $(\tbmx_{0}, \tbmu_{0}) \in \F{\Omega}_1$.} 
\label{fig5}
\end{figure}

\begin{figure}[h]
\centering
\includegraphics[scale=0.37]{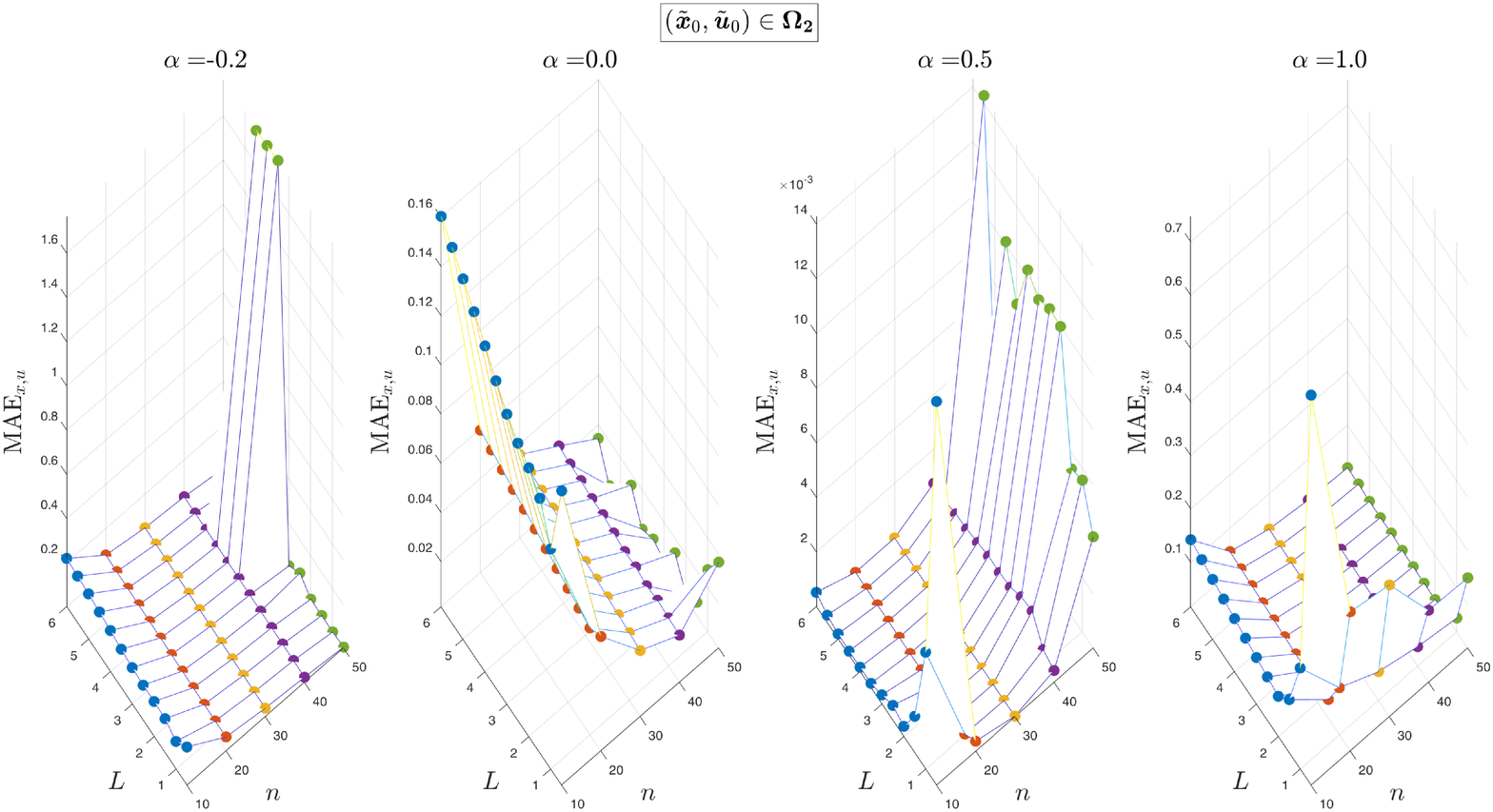}
\caption{The $MAE_{x,u}$ of the GGR-IPS1-EALMM at $101$ linearly spaced nodes between $0$ and $10$ using $n=10(10)50, \alpha=-0.2,0,0.5,1, L=0.5(0.5)6$, and $(\tbmx_{0}, \tbmu_{0}) \in \F{\Omega}_2$.} 
\label{fig5new}
\end{figure}

\begin{figure}[h]
\centering
\includegraphics[scale=0.37]{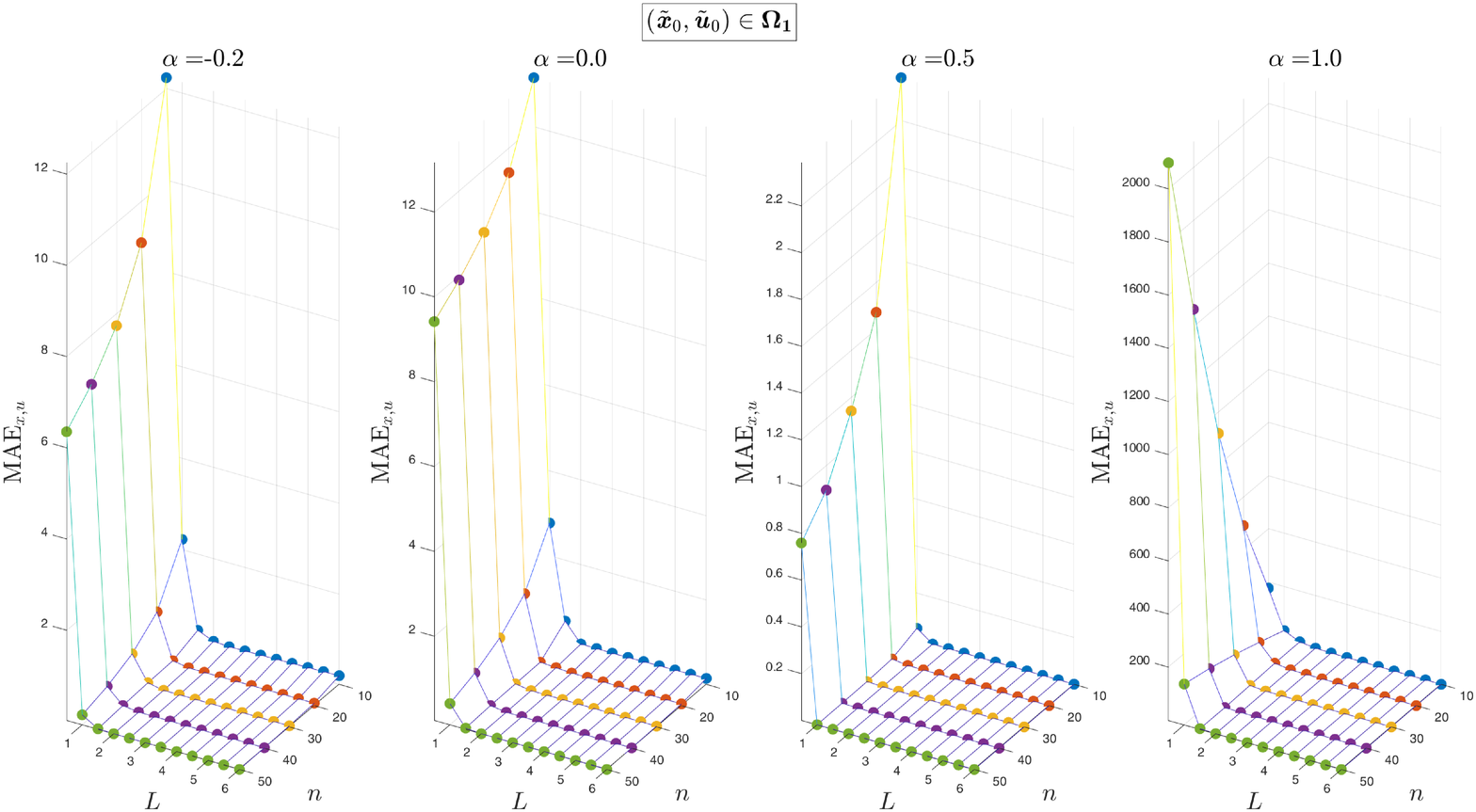}
\caption{The $MAE_{x,u}$ of the GGR-IPS2-EALMM at $101$ linearly spaced nodes between $0$ and $10$ using $n=10(10)50, \alpha=-0.2,0,0.5,1, L=0.5(0.5)6.5$, and $(\tbmx_{0}, \tbmu_{0}) \in \F{\Omega}_1$.} 
\label{fig5trans2}
\end{figure}

\begin{figure}[h]
\centering
\includegraphics[scale=0.37]{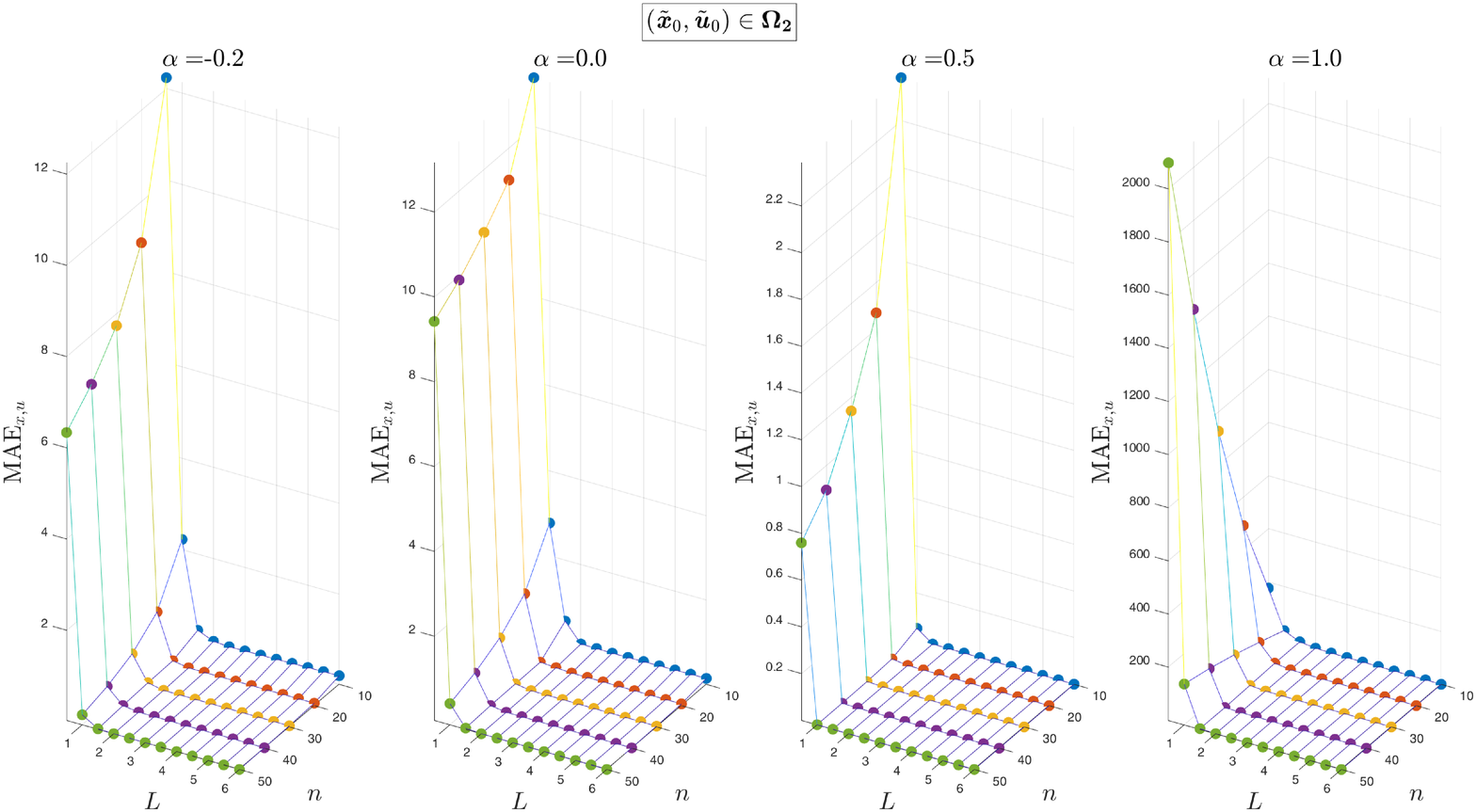}
\caption{The $MAE_{x,u}$ of the GGR-IPS2-EALMM at $101$ linearly spaced nodes between $0$ and $10$ using $n=10(10)50, \alpha=-0.2,0,0.5,1, L=0.5(0.5)6.5$, and $(\tbmx_{0}, \tbmu_{0}) \in \F{\Omega}_2$.} 
\label{fig5newtrans2}
\end{figure}

Figures \ref{fig6} and \ref{fig7} show comparisons of the number of iterations required by the GGR-IPS12 methods when combined with three distinct NLP solvers using $(\tbmx_{0}, \tbmu_{0}) \in \F{\Omega}$ and several $L$- and $\alpha$-values. Clearly, the integration of the GGR-IPS12 methods with the EALMM leads to a drastic reduction in the number of iterations in all cases. In fact, while the GGR-IPS12-EALMM often converged in only four/five iterations, other methods usually require many more iterations to converge; for example, the GGR-IPS12 methods combined with fmincon-int and fmincon-sqp took more than $200$ iterations to converge to the solutions of the problem when $(n, L, \alpha) = (48, 1, -0.2)$ starting with any initial guess $(\tbmx_{0}, \tbmu_{0}) \in \F{\Omega}$.

\begin{figure}[h]
\centering
\includegraphics[scale=0.37]{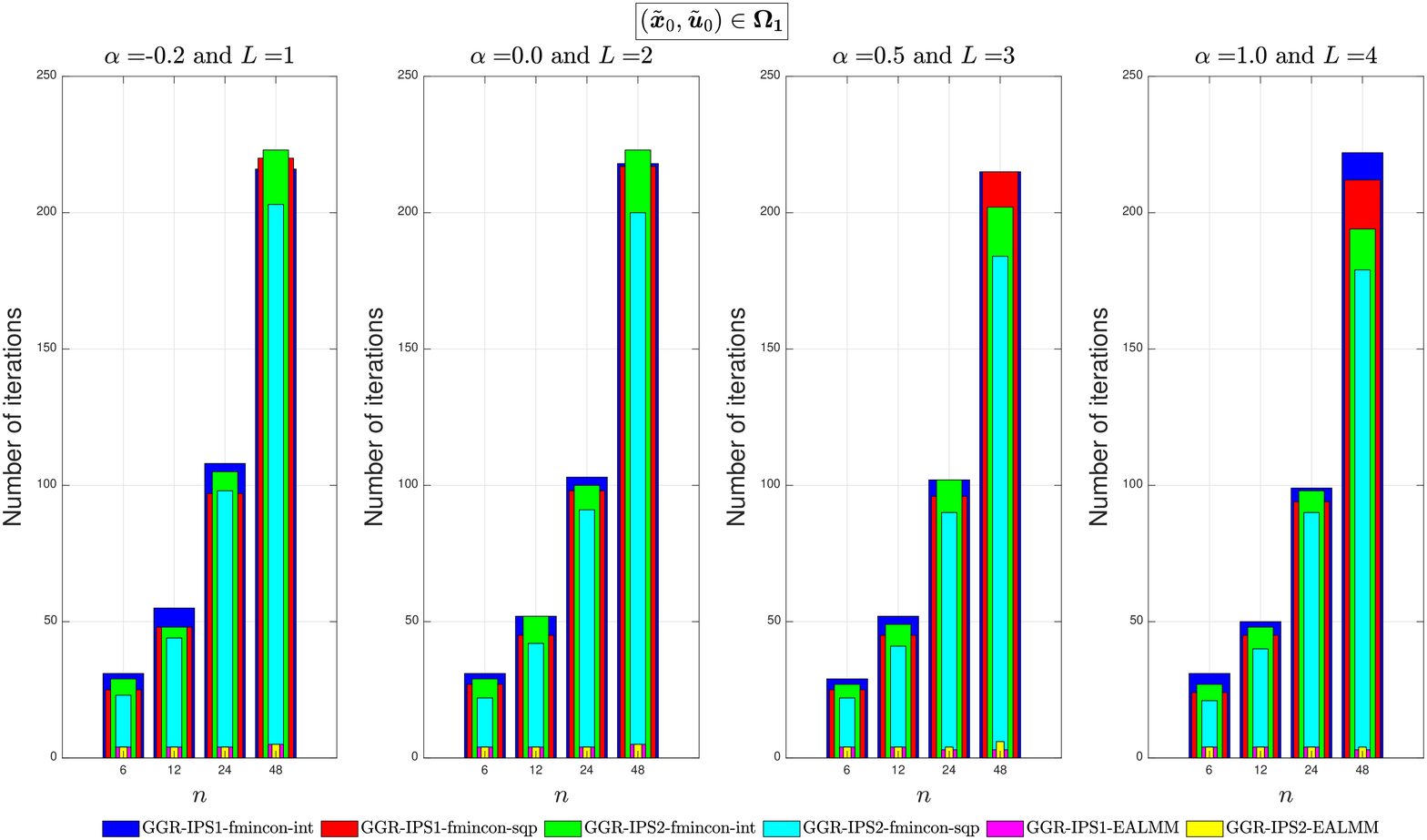}
\caption{The number of iterations required by the GGR-IPS12 methods performed with three distinct NLP solvers versus $n$, for $(\alpha,L) = (-0.2,1), (0,2), (0.5,3), (1,4)$, and $(\tbmx_{0}, \tbmu_{0}) \in \F{\Omega}_1$.} 
\label{fig6}
\end{figure}

\begin{figure}[h]
\centering
\includegraphics[scale=0.37]{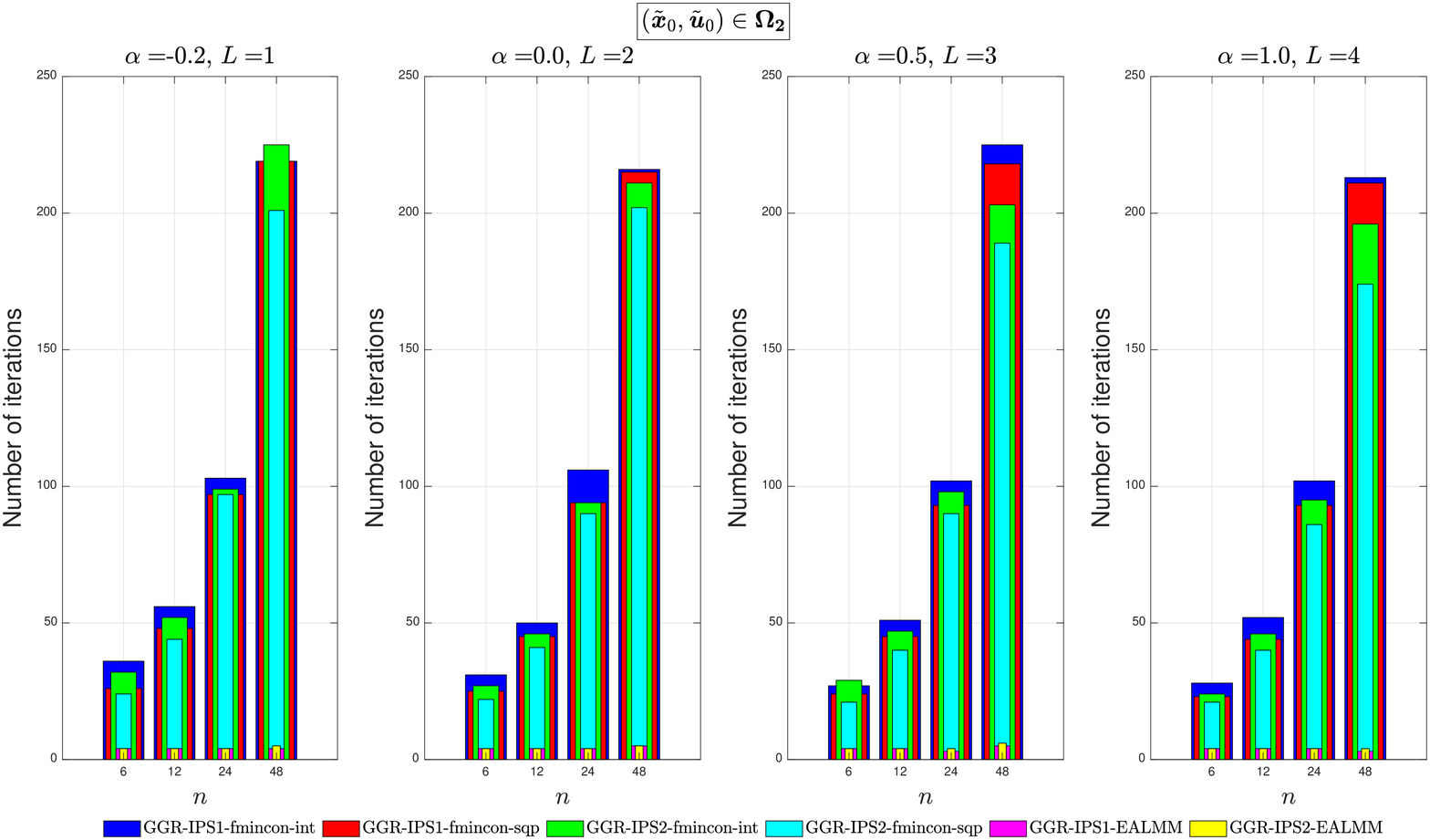}
\caption{The number of iterations required by the GGR-IPS12 methods performed with three distinct NLP solvers versus $n$, for $(\alpha,L) = (-0.2,1), (0,2), (0.5,3), (1,4)$, and $(\tbmx_{0}, \tbmu_{0}) \in \F{\Omega}_2$.}
\label{fig7}
\end{figure}

Table \ref{table1} shows the approximate cost function values obtained by the GGR-IPS2-EALMM for several parameter values. The fastest convergence was recorded at $\alpha = 0.5$ in all cases with $J \approx J_{16} = 0.579958091127$ in agreement with $J$ to $10$ significant digits. It is interesting to see through the tabulated data how Gegenbauer polynomials with $\alpha \in \{-0.4, 0.25\}$ exhibit faster convergence rates than Chebyshev polynomials (when $\alpha = 0$) capturing $4$ correct significant digits as early as $n = 12$, whereas Chebyshev polynomials are still lagging behind by one digit even when $n$ increases by $2$ units. Gegenbauer polynomials with $\alpha = -0.2$ also performed better than Chebyshev polynomials for $n \in \{14, 16\}$, while the poorest stability was that of Gegenbauer polynomials with $\alpha = 2$ scoring only one correct significant digit in all cases! Table \ref{table2} shows the smallest $MAE_{x,u}$ obtained by the GGR-IPS12-EALMM among the recorded errors for the parameter values $n = 10(10)80, \alpha = -0.4(0.1)2, L = 0.25(0.25)10$, and $(\tbmx_{0}, \tbmu_{0}) \in \F{\Omega}$. The table shows the capacity of the GGR-IPS12-EALMM to achieve improved near-optimal solutions for increasing values of $n$ within a small/medium range of mesh grid size; however, the accuracy deteriorates beyond a certain limit as the mesh grid size grows larger in agreement with the theoretical results proven in Section \ref{sec:errb}. The GGR-IPS2-EALMM is clearly superior to the GGR-IPS1-EALMM in terms of accuracy in all cases, and it is interesting here to see how the GGR-IPS1-EALMM diverges faster than the GGR-IPS2-EALMM for growing mesh sizes as prophesied earlier by the divergence analysis presented in Section \ref{subsec:DOTPCS}. The best approximations obtained experimentally by the GGR-IPS2-EALMM were recorded at/near $\alpha = 0.5$ with $\alpha \in \Upsilon_{0.4,0.6}^G$. The smallest errors of the GGR-IPS1-EALMM were also recorded at/near $\alpha = 0.5$ for $n = 5(5)55$ with $\alpha \in \Upsilon_{0.4,0.6}^G$; however, the algorithm tends to favor larger positive $\alpha$-values beyond $n = 55$, where we noticed the travel of the right boundary, $\amax$, of $\Upsilon_{0.4,\amax}^G$ rightward from $\amax = 0.6$ into $\amax = 1.8$ as $n$ reaches $80$. This is no surprise! In fact, recall that $T_{1,L}^{(\alpha)}$ increases monotonically for decreasing values of $\alpha$ while holding $n$ and $L$ fixed, and we can observe from Figure \ref{fig:IHtp1} that the mapping escalates wildly as we continue decreasing the $\alpha$-values. The byproduct of this behavior is that increasing the $\alpha$-values moves the collocation points associated with large values of $t$ leftward and relocate them closer to regions where the solution changes rapidly. This graphical interpretation consents with the fact that the interior GGR points move monotonically toward the center of the interval $(-1, 1)$ as the parameter $\alpha$ increases; cf. \cite{szeg1939orthogonal,elgindy2018optimal}. From another perspective, the leftward movement of the collocation points near the right boundary $\tau = 1$ mitigates the effect of the ill-conditioning of $T'_{1,L}$ for arguments near $1$, as $T'_{1,L}$ is evaluated at mesh points that are gradually departing the vicinity of $\tau = 1$. This argument adds more tenability for using Gegenbauer polynomials as a basis polynomials for numerical collocations of FHOCIs obtained from IHOCs via $T_{1,L}^{(\alpha)}$ or $T_{2,L}^{(\alpha)}$ in the sense that, while Chebyshev and Legendre polynomials cease to downgrade the errors as the mesh grid size grows large, Gegenbauer polynomials has the additional advantage to alleviate the growth rates of both $T_{1,L}^{(\alpha)}$ and $T_{2,L}^{(\alpha)}$ by increasing the $\alpha$-value whenever we wish while sustaining the luxury to apply either Chebyshev or Legendre polynomials with a push of a button: simply set $\alpha = 0$ or $0.5$ in the solver code! If we now turn our attention to the recorded $L$-values in the table, we can quickly spot that the smallest computed errors initially spans a wide range of numerically optimal $L$-values; however, the solvers ultimately have a bias towards smaller values of $L$ in attempt to damp the error in agreement with the forward error analysis presented in Section \ref{sec:errb}. Notice that the numerically optimal $L$-value for the GGR-IPS2-EALMM stays at $0.75$ for $n = 55(5)80$, while the corresponding values for the GGR-IPS1-EALMM occur at the smallest feasible $L$-value among the input range of experimental data. One may pin this peculiar behavior of the solvers to the fact that the logarithmic map $T_{2,L}^{(\alpha)}$ increases at a much slower rate than that of the algebraic map $T_{1,L}^{(\alpha)}$; cf. Figures \ref{fig:IHtp1} and \ref{fig:IHtp2}.
 
\begin{table}[ht]
\centering
\caption{The approximate cost function values obtained by the GGR-IPS2-EALMM for $(n,L) = (6,1), (8,2), (10,3), (12,4), (14,5), (16,6)$, and $\alpha=-0.4,-0.2,0,0.25,0.5,1,2$. All approximations were rounded to $12$ significant digits.} 
\resizebox{\columnwidth}{!}{%
\begin{tabular}{|c|c|*{7}{c}|}
\hline
 $n$ & $L$ & $\alpha=-0.4$ & $\alpha=-0.2$& $ \alpha=0$& $ \alpha=0.25$& $ \alpha=0.5$& $\alpha=1$ & $\alpha=2$\\\hline
6 & 1 & 0.579809073360 & 0.579627701619 & 0.579622685738 & 0.579789248084 & 0.579949642114 & 0.577727846201 & 0.503481602677\\ 
8 & 2 & 0.579848669619 & 0.579713782304 & 0.579730070685 & 0.579859689930 & 0.579958090977 & 0.578484510558 & 0.522640769897\\ 
10 & 3 & 0.579893894832 & 0.579797498143 & 0.579802788432 & 0.579889201985 & 0.579958091142 & 0.578845602933 & 0.530047022566\\
12 & 4 & 0.579918900010 & 0.579850348844 & 0.579850444796 & 0.579908959905 & 0.579958091143 & 0.579089227180 & 0.534636030121\\ 
14 & 5 & 0.579933051361 & 0.579883424648 & 0.579881314693 & 0.579922126600 & 0.579958091151 & 0.579263254798 & 0.538003914873\\ 
16 & 6 & 0.579941397724 & 0.579904638260 & 0.579901719586 & 0.579931066922 & 0.579958091127 & 0.579391311257 & 0.540689754740\\\hline
\end{tabular}
}
\label{table1}   
\end{table}

\begin{table}[ht]
\centering
\caption{The $MAE_{x,u}$ of the GGR-IPS12-EALMM obtained using $n=5(5)80$ and $(\tbmx_{0}, \tbmu_{0}) \in \F{\Omega}$. All approximations were rounded to $5$ significant digits.} 
\resizebox{\columnwidth}{!}{%
\begin{tabular}{*{13}{c}}
\toprule
& \multicolumn{6}{c}{GGR-IPS1-EALMM} & \multicolumn{6}{c}{GGR-IPS2-EALMM}\\
\cmidrule(lr){2-7} \cmidrule(lr){8-13}
		&          &   $(\tbmx_{0}, \tbmu_{0}) \in \F{\Omega}_1$ &       &      &  $(\tbmx_{0}, \tbmu_{0}) \in \F{\Omega}_2$    &       &   &  $(\tbmx_{0}, \tbmu_{0}) \in \F{\Omega}_1$ &  & & $(\tbmx_{0}, \tbmu_{0}) \in \F{\Omega}_2$   &\\    
\cmidrule(lr){2-4} \cmidrule(lr){5-7} \cmidrule(lr){8-10} \cmidrule(lr){11-13}				
$n$ & $\alpha$ & $L$&     $MAE_{x,u}$  &  $\alpha$  & $L$    & $MAE_{x,u}$ & $\alpha$   &     $L$ & $MAE_{x,u}$ & $\alpha$ & $L$ & $MAE_{x,u}$ \\\hline
5 & 0.6 & 2.25 & 5.3830e-03 & 0.6 & 2.25 & 5.3830e-03 & 0.5 & 3.5 & 4.2456e-05 & 0.5 & 3.5 & 4.2458e-05 \\
10 & 0.5 & 5.75 & 6.9439e-05& 0.5 & 5.75 & 7.0620e-05 & 0.5 & 5.25 & 8.6420e-09& 0.5 & 4.75 & 9.8263e-09 \\ 
15 & 0.5 & 7 & 6.3879e-07 & 0.5 & 7 & 6.7410e-07 & 0.5 & 4 & 6.1199e-09& 0.5 & 3.25 & 5.7262e-09 \\  
20 & 0.5 & 10 & 2.3194e-07 & 0.5 & 9 & 1.4376e-07 & 0.5 & 2.25 & 6.7398e-09 & 0.5 & 2.5 & 8.6293e-09\\ 
25 & 0.5 & 9 & 1.7060e-07 & 0.5 & 5.75 & 1.5636e-07 & 0.5 & 2.25 & 4.4629e-08 & 0.5 & 5.75 & 3.7128e-08\\ 
30 & 0.5 & 2 & 4.9161e-07 & 0.5 & 3.25 & 2.9685e-07 & 0.5 & 1.75 & 6.4234e-08 & 0.5 & 1.75 & 9.7104e-08\\ 
35 & 0.5 & 0.75 & 2.6309e-06 & 0.5 & 0.5 & 5.3234e-06 & 0.5 & 1.75 & 1.0386e-07 & 0.5 & 1.75 & 9.4232e-08\\ 
40 & 0.5 & 0.25 & 6.7082e-05 & 0.5 & 0.25 & 1.8574e-05 & 0.5 & 1.5 & 4.7475e-07 & 0.5 & 2 & 2.6221e-07\\
45 & 0.5 & 1 & 7.2641e-05 & 0.5 & 4.25 & 1.3092e-04 & 0.5 & 1.25 & 5.1101e-07 & 0.5 & 1.75 & 2.4520e-07\\
50 & 0.4 & 1.5 & 9.7024e-04 & 0.4 & 3.25 & 1.0915e-03 & 0.5 & 1 & 8.6070e-06 & 0.5 & 1.25 & 6.6680e-06\\
55 & 0.6 & 4.75 & 1.1300e-03 & 0.6 & 7.25 & 1.1785e-03 & 0.5 & 0.75 & 1.8330e-05 & 0.5 & 0.75 & 1.8329e-05\\
60 & 1 & 0.5 & 1.9681e-02 & 0.7 & 0.25 & 1.6967e-02 & 0.5 & 0.75 & 7.0032e-05 & 0.5 & 0.75 & 1.4334e-04 \\
65 & 1.2  & 0.25 & 2.2391e-02 & 1.2 & 0.25 & 2.2419e-02 & 0.5 & 0.75 & 3.1392e-04 & 0.5 & 0.75 & 2.1468e-04 \\ 
70 & 1.4  & 0.25 & 5.8462e-02 & 1.3 & 0.25 & 7.7659e-02 & 0.5 & 0.75 & 1.0067e-03 & 0.4 & 0.75 & 8.0880e-04 \\
75 & 1.4 & 0.25 & 1.5099e-01 & 1.4 & 0.25 & 1.5681e-01 & 0.6 & 0.75 & 8.6516e-04 & 0.5 & 0.75 & 8.9362e-04 \\
80 & 1.8 & 0.25 & 3.2058e-01 & 1.8 & 0.25 & 2.7602e-01 & 0.4 & 0.5 & 8.3497e-04 & 0.4 & 0.75 & 6.7363e-04\\
\bottomrule
\end{tabular}
}
\label{table2}  
\end{table}

\noindent \textbf{Example 2.} Consider the IHOC \eqref{eq:2}–\eqref{eq:1} with $g\left(\bmx(t),\bmu(t)\right)=x_{1}^{2}(t)+x_{2}^{2}(t)/2+u^{2}(t)/4, \bmf\left(\bmx(t),\bmu(t)\right)=\left[x_{2}(t), 2x_{1}(t)\right.\\\left. -x_{2}(t)+u(t)\right]^{t}$, and $\bmx_{0}=[-4,4]^{t}$. The exact state and control variables are
\begin{subequations}
\begin{align}
\bmx^{*}(t) &= \exp\left(\C{M}t\right)\bmx(0),\label{eq:exact_state_ex2}\\
u^{*}(t) &= -\bm{K} \bmx^{*}(t),\label{eq:exact_control_ex2}
\end{align}
where
\begin{align}
{\C{M}} &= \left[ {\begin{array}{*{20}{c}}
0&1\\
-2.82842712474619&-3.557647291327851
\end{array}} \right],\quad \text{and}\label{eq:C1}\\
{\bm{K}} &= \left[4.828427124746193;2.557647291327851\right],\label{eq:g1}
\end{align}
\end{subequations}
cf. \cite{kirk1970prentice,fahroo2008pseudospectral,garg2011direct}. This example is a linear quadratic regulator problem with an optimal cost functional value $J^{*}= 19.85335656362790$, rounded to $16$ significant digits, as obtained by MATLAB using the Symbolic Math Toolbox. Figure \ref{fig1ex2} shows the plots of the exact optimal state and control variables and their approximations obtained through GGR-IPS2-EALMM using some parameter values. Table \ref{table1_ex2} shows the $MAE_{x,u}$ of the LGR-PS method of \citet{garg2011direct} and the smallest corresponding $MAE_{x,u}$ and $AE_{J}$ pairs of the GGR-IPS2-EALMM at the collocation points using $(\tbmx_{0}, \tbmu_{0}) \in \F{\Omega}$ and several values of $n$. The GGR-IPS2-EALMM proves again to be superior in terms of accuracy for a small range of mesh size, as it converges rapidly to near optimal solutions at a much higher-rate than that of \citet{garg2011direct}. However, the superb accuracy of the method starts to decline when $n$ grows larger as anticipated earlier. Notice again here that the best accuracy in all cases was recorded at $\alpha = 0.5$ with SRCIC $\Upsilon_{0.5,0.5}^G = \{0.5\}$. 

\begin{table}[ht]
\centering
\caption{The uncertainty intervals of the smallest $MAE_{x,u}$ obtained by the method of \citet{garg2011direct} at the collocation points and the corresponding smallest $MAE_{x,u}$ and $AE_{J}$ pairs obtained by the GGR-IPS2-EALMM using $n = 4(5)34$ and $(\tbmx_{0}, \tbmu_{0}) \in \F{\Omega}$. All approximations were rounded to $5$ significant digits.} 
\begin{tabular}{c c c c}
\toprule
 & Method of \citet{garg2011direct}  & \multicolumn{2}{c}{GGR-IPS2-EALMM} \\
    &                                  & $(\tbmx_{0}, \tbmu_{0}) \in \F{\Omega}_1$  & $(\tbmx_{0}, \tbmu_{0}) \in \F{\Omega}_2$ \\  
$n+1$ &   $MAE_{x,u}$ uncertainty interval &  $MAE_{x,u}/AE_{J}/\alpha/L$                              &  $MAE_{x,u}/AE_{J}/\alpha/L$ \\
\midrule
5 &  (1e-01,1)  & 1.6959e-03/1.7032e-05/0.5/1.75 & 1.6959e-03/1.7032e-05/0.5/1.75 \\           
10 &  (1e-02,1e-01)  & 9.4155e-08/1.7870e-12/0.5/2.5 & 8.5595e-07/9.6705e-12/0.5/3.25 \\           
15  &  (1e-04, 1e-02)  &  2.0165e-08/4.9489e-12/0.5/5.5 & 1.2501e-08/2.1316e-13/0.5/2.5 \\ 
20  &  (1e-04, 1e-03)  &  8.2183e-09/1.6485e-12/0.5/3.5 & 6.2243e-09/1.0040e-11/0.5/2.5 \\
25  &  (1e-05, 1e-04)  &  1.7564e-07/3.6451e-12/0.5/3.5 & 8.4676e-08/1.7483e-11/0.5/3\\
30  &  (1e-06, 1e-05)  &  1.1714e-06/3.1175e-11/0.5/1.75 & 2.4535e-06/4.9347e-12/0.5/1.75\\
35  &  (1e-06, 1e-05)  &  6.9522e-06/9.9437e-11/0.5/1.5 & 4.5463e-06/1.0522e-10/0.5/1.5\\
\bottomrule
\end{tabular}
\label{table1_ex2} 
 \end{table} 
 
\begin{figure}[h]
\centering
\includegraphics[scale=0.7]{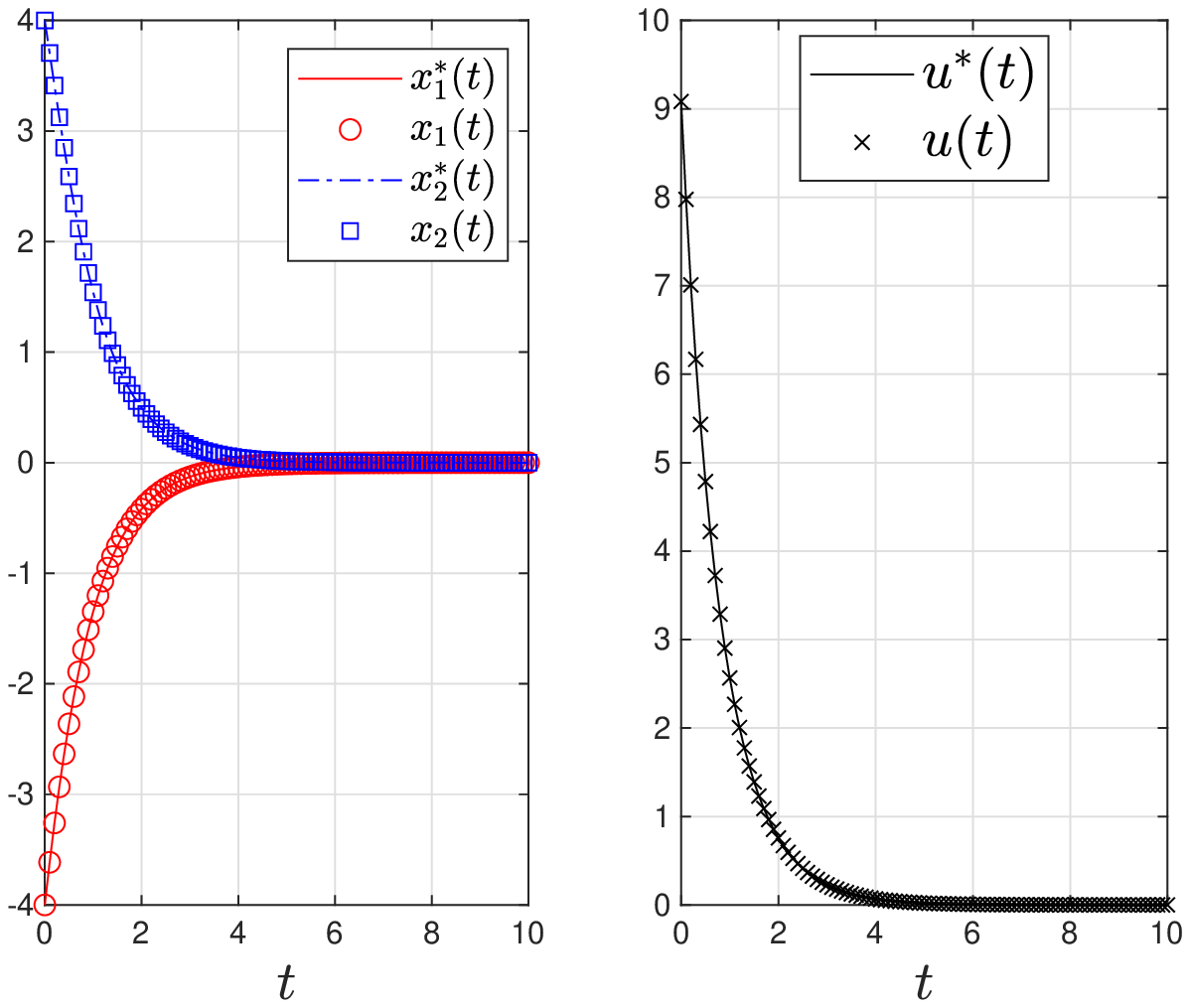}
\caption{The plots of the exact optimal states and control of Example $2$ and their collocated approximations obtained through GGR-IPS2-EALMM on the interval $[0,10]$ using $n=9, \alpha=0.5, L=2.5$, and $(\tbmx_{0}, \tbmu_{0}) \in \F{\Omega}_1$. All figures were generated using $101$ linearly spaced nodes from $0$ to $10$.}
\label{fig1ex2}
\end{figure}
 
Another comparison between the GGR-IPS2-fmincon-int, GGR-IPS2-fmincon-sqp, and the transformed LGR method of \citet{shahini2018transformed} is shown in Table \ref{table2_ex2}. We can clearly see that the former two methods generally yield smaller $AE_{J}$ values. The rise and fall of accuracy as the mesh size grows is again peculiar in the observed approximations in agreement with the presented divergence analysis in Section \ref{subsec:DOTPCS}. Remarkably, a match with the exact $J^*$ to full machine precision was recorded as early as $n = 20$ indicating an exceedingly accurate numerical scheme with exponential convergence for coarse meshes. All smallest errors reported by the current methods occurred at $\alpha = 0.5$, except for $n \in \{90, 100\}$, where collocations at $\alpha = 0$ and $0.8$ furnished higher accuracy. On the other hand, the rounded errors in \cite{shahini2018transformed} decay to $1.35 \times 10^{-09}$ as soon as $n = 30$ using the algebraic map, but surprisingly cease to vary any further for $n = 40(10)100$(!)

\begin{table}[ht]
\centering
\caption{The $AE_{J}$ of the method of \citet{shahini2018transformed} with the algebraic map $T_{1,L}^{(\alpha)}$ and the smallest $AE_{J}$ obtained by GGR-IPS2-fmincon-int and GGR-IPS2-fmincon-sqp using $(\tbmx_{0}, \tbmu_{0}) \in \F{\Omega}$. All approximations were rounded to $5$ significant digits.} 
\resizebox{\columnwidth}{!}{%
\begin{tabular}{*{7}{c}}
\toprule
 & \multicolumn{2}{c}{Method of \citet{shahini2018transformed}} & \multicolumn{4}{c}{GGR-IPS2}\\
 & Algebraic map $T_{1,L}^{(\alpha)}$ & Logarithmic map $T_{2,L}^{(\alpha)}$ & \multicolumn{2}{c}{fmincon-int}& \multicolumn{2}{c}{fmincon-sqp}\\ 
 & & & $(\tbmx_{0}, \tbmu_{0}) \in \F{\Omega}_1$ & $(\tbmx_{0}, \tbmu_{0}) \in \F{\Omega}_2$ & $(\tbmx_{0}, \tbmu_{0}) \in \F{\Omega}_1$ & $(\tbmx_{0}, \tbmu_{0}) \in \F{\Omega}_2$\\
$n$ & \multicolumn{2}{c}{$AE_{J}$} & $AE_{J}/\alpha/L$&  $AE_{J}/\alpha/L$ & $AE_{J}/\alpha/L$&  $AE_{J}/\alpha/L$\\
\midrule
10  &  4.82e-05 & 3.92e-05&5.3291e-14/0.5/2.5 &5.3291e-14/0.5/2.5 & 1.0658e-14/0.5/4.25 & 7.1054e-15/0.5/2.5  \\                               
20  &  4.88e-09 & 1.76e-06&0/0.5/6            &         0/0.5/5.25&               0/0.5/5.75 &      0/0.5/5.75\\ 
30  &  1.35e-09 & 2.73e-07&7.1054e-15/0.5/3   &1.4211e-14/0.5/5.25& 1.0658e-14/0.5/2.5  &1.7764e-14/0.5/5.5 \\
40  &  1.35e-09 & 7.24e-08&1.0658e-14/0.5/5.75&3.1974e-14/0.5/5.25& 1.0303e-13/0.5/6    &1.2079e-13/0.5/5   \\
50  &  1.35e-09 & 2.63e-08&8.8818e-14/0.5/10  &1.1013e-13/0.5/2.5 & 5.1514e-13/0.5/5    &5.9686e-13/0.5/5   \\
60  &  1.35e-09 & 1.19e-08&1.7053e-13/0.5/2.5 &4.0146e-13/0.5/10  & 8.3844e-13/0.5/5.5  &1.0409e-12/0.5/5.75\\
70  &  1.35e-09 & 6.45e-09&2.7001e-13/0.5/5   &4.0501e-13/0.5/10  & 1.2967e-12/0.5/10   &1.7337e-12/0.5/5   \\
80  &  1.35e-09 & 4.06e-09&8.2423e-13/0.5/5.5 &8.3134e-13/0.5/5.5 & 1.7977e-12/0.5/10   &2.7676e-12/0.5/2.5 \\
90  &  1.35e-09 & 2.90e-09&1.1072e-10/0/5.25  &2.1283e-09/0/2.5   & 1.4021e-09/0/3.5    & 1.4747e-10/0/2.5   \\   
100  & 1.35e-09 & 2.29e-09&9.5391e-08/0.8/4   &1.9779e-07/0.8/3.5 & 5.3783e-08/0.8/4    & 1.6219e-07/0.8/3.75\\
\bottomrule
\end{tabular}
}
\label{table2_ex2}   
 \end{table}

\section{Conclusion and Future Work}
\label{sec:CAFW1}
Direct IPS methods for solving IHOCS using the logarithmic mapping $T_{2,L}^{(\alpha)}$ and the developed SR-interpolation and barycentric quadrature formulas can produce excellent approximations to the optimal state and control variables for relatively small/medium mesh grids. However, this class of methods often suffer from numerical instability for fine meshes when endowed with any of the parametric maps $T_{i,L}^{(\alpha)}, i = 1, 2$; therefore, as the mesh size grows, they are not as useful as one might hope for computing the optimal state and control trajectories to within high precision. In fact, it has been shown in the current paper that two sources of difficulty arise in handling the horizon in IHOCs by a domain transformation that maps the infinite horizon to the finite horizon $[-1, 1)$ through the algebraic and logarithmic maps $T_{i,L}^{(\alpha)}, i = 1, 2$, namely (i) the exponential growth of the mappings surface slopes near the right boundary $\tau = 1$, which increase the truncation errors produced in the FHOCI discretization without bounds as $\tau \to 1$, and (ii) despite the fact that both mappings $T_{i,L}^{(\alpha)}, i = 1, 2$ have a singularity at $\tau = 1$, and we actually never evaluate them at the singularity, since the GGR collocation points are strictly less than 1, their derivatives are sensitive to input data errors for arguments near $\tau = 1$; thus, both NLP1 and NLP2 are ill-conditioned for $\tau \approx 1$. These theoretical facts as well as the observed empirical data are considerable reasons to say that typical direct spectral/PS- and IPS-methods based on classical Jacobi polynomials and the parametric maps $T_{i,L}^{(\alpha)}, i = 1, 2$ are foreseen to diverge as the mesh size grows large, if the computations are carried out using floating-point arithmetic and the discretizations use a single mesh grid whether they are of Gauss type or equally-spaced. 

While Gegenbauer polynomials associated with certain nonpositive $\alpha$-values are well suited for FGGR-based polynomial interpolations in Lagrange-basis form over fine meshes as shown by \citet{elgindy2018high1}, this paper asserts that Gegenbauer polynomials associated with certain nonnegative $\alpha$-values are more apt for GGR-based SR-interpolations over fine meshes. Moreover, for coarse mesh grids, Legendre polynomials are particularly (near) optimal basis polynomials for GGR-based SR-collocations of FHOCIs, as argued in Section \ref{subsubsec:TLcfggr1} and sustained through numerical simulations. On the other hand, Gegenbauer polynomials associated with certain positive values of $\alpha \in (1/2, 2]$ are optimal for IHOCIs collocations over fine mesh grids, as they can largely slow down the exponential growth of both parametric maps  $T_{i,L}^{(\alpha)}, i = 1, 2$, and their associated GGR collocation points are less dense near $\tau = 1$; thus, the sensitivity of computing $T'_{i,L}, i = 1, 2$ at arguments near $\tau = 1$ is significantly attenuated. The paper also shows that the parametric map $T_{1,L}^{(\alpha)}$ is more severely sensitive for $\tau \approx 1$ than $T_{2,L}^{(\alpha)}$ and the family $\left\{\left(T_{1,L}^{(\alpha)}\right)^m\right\}_{m=0}^\infty$ grows faster than $\left\{\left(T_{2,L}^{(\alpha)}\right)^m\right\}_{m=0}^\infty$ as $\tau \to 1$; therefore, $T_{2,L}^{(\alpha)}$ is more apt for the domain transformation of IHOCs than $T_{1,L}^{(\alpha)}$ for\\[0.5em] collocation points of Gauss/GR type. 

It is worthy to mention that direct IPS methods based on the proposed Gegenbauer SR-collocation can exhibit faster convergence rates for coarse meshes by regulating the map scaling parameter $L$ and the Gegenbauer parameter $\alpha$. In light of the stability analysis conducted in Section \ref{subsubsec:TLcfggr1}, GGR-based SR-collocations of well conditioned problems are generally endorsed for $\alpha$-values within/near the SRCIC $\Upsilon_{1/2,1}^G$; the current study also supports this rule of thumb for IHOCs when converted into FHOCIs through the parametric maps $T_{i,L}^{(\alpha)}, i = 1, 2$ and then collocated at relatively coarse mesh grids. However, the question of how can we find the optimal map scaling parameter $L^*$ for IHOCs remains open. An interesting direction for future works may involve a study of new mappings with smaller growth rates and derivatives of less sensitivity to input data errors. 

\appendix

\section{The Barycentric GRDM}
\label{sec:TDIM}
To construct the barycentric GRDM, we follow the derivation presented in \cite{elgindy2018high} and multiply both sides of Eq. \eqref{eq:new1} by $x-\tau_{j}\,\forall j$ to render them differentiable at $x=\tau_{j}$ such that
\begin{equation}
\C{L}_{n,i}(x) \sum_{k=0}^{n}\frac{\xi_{k}\left(x-\tau_{j}\right)}{x-\tau_{k}}=\frac{\xi_{i}\left(x-\tau_{j}\right)}{x-\tau_{i}},\quad i=0,\ldots,n.
\label{eq:newD1}
\end{equation}
Letting $\displaystyle{S(x)= \sum_{k=0}^{n}\frac{\xi_{k}\left(x-\tau_{j}\right)}{x-\tau_{k}}}$ and differentiating Eq. \eqref{eq:newD1} with respect to $x$ yields
\begin{equation}
S(x)\C{L}'_{n,i}(x)+\C{L}_{n,i}(x) S'(x)= \xi_{i} \left(\frac{x-\tau_{j}}{x-\tau_{i}}\right)',\quad i = 0, \ldots, n.
\label{eq:newD2}
\end{equation}
Since $\displaystyle{S(\tau_{j})=\xi_{j}, S'(\tau_{j})=\sum_{j \neq k}\frac{\xi_{k}}{
\tau_{j}-\tau_{k}}}$, and $\C{L}_{n,i}(\tau_{j})=0\, \forall i \neq j$, the off-diagonal elements of the differentiation matrix $\F{D} = (d_{j,i})_{0 \le j,i \le n}$ can be calculated by the following formula:
\begin{equation}
d_{j,i}= \C{L}'_{n,i}(\tau_{j})= \frac{\xi_{i}/\xi_{j}}{\tau_{j}-\tau_{i}},\quad \forall i \neq j.
\label{eq:newD3}
\end{equation}
For $i=j$, we have $\displaystyle{\sum_{i=0}^{n} \C{L}_{n,i}\left(x\right)=1}$, so $\displaystyle{\sum_{i=0}^{n} \C{L}'_{n,i}\left(x\right)=0}$, and 
\begin{equation}
d_{i,i}= \C{L}'_{n,i}\left(\tau_{i}\right) = -\sum_{j \neq i} \C{L}'_{n,j}\left(\tau_{i}\right)= - \sum\limits_{i \ne j} {{d_{i,j}}},\quad i = 0, \ldots, n. 
\label{eq:newD4}
\end{equation}
Hence, the derivative of a real-valued function $f \in C^{1}[-1,1]$ can be approximated at the GGR points by the following formula:
\begin{equation}
f'(\tau_{j}) \approx \sum_{i=0}^{n} d_{j,i} f_{i},\quad j = 0, \ldots, n.
\label{eq:newD5}
\end{equation}

\section{Computational algorithms}
\label{app:Alg1}
\begin{algorithm}[H]
\renewcommand{\thealgorithm}{B.1}
\caption{The First Switching Formula of the Barycentric Weights for the GGR Points}
\label{algorithm:1}
\begin{algorithmic}[1]
\Statex \textbf{Input}: Positive integer $n$; a real number $\alpha > -1/2$; the set of GGR points and quadrature weights $\lbrace\tau_{i},\varpi_{i}\rbrace_{i=0}^{n}$; a relatively small positive real number $\varepsilon$.
\Statex \textbf{Output}: Barycentric weights $\xi_{i},i=0,\ldots,n$.
\State $\xi_{0} \gets -\sqrt{(2\alpha+1)\varpi _{0}}$.
 \For{$i=1$ to $n$}
 \If{$\left|1-\tau_{i}\right| > \varepsilon$}
   \State  {$\xi_{i} \gets (-1)^{i-1}\sqrt{\left(1-\tau_{i}\right)\varpi _{i}}$}.
	\Else
	 \State {$\xi_{i} \gets (-1)^{i-1} \sin{\left(\frac{1}{2}\cos^{-1}{\tau_{i}}\right)} \sqrt{2 \varpi _{i}}$}.
	  \EndIf
	\EndFor
\State Stop.
\end{algorithmic}
\end{algorithm}

\begin{algorithm}[H]
\renewcommand{\thealgorithm}{B.2}
\caption{The Second Switching Formula of the Barycentric Weights for the GGR Points}
\label{algorithm:2}
\begin{algorithmic}[1]
\Statex \textbf{Input}: Positive integer $n$; a real number $\alpha > -1/2$; the set of GGR points and quadrature weights $\lbrace\tau_{i},\varpi_{i}\rbrace_{i=0}^{n}$; a relatively small positive real number $\varepsilon$.
\Statex \textbf{Output}: Barycentric weights $\xi_{i},i=0,\ldots,n$.
\State $\xi_{0} \gets -\sqrt{(2\alpha+1)\varpi _{0}}$.
 \For{$i=1$ to $n$}
 \If{$\left|1-\tau_{i}\right| > \varepsilon$}
   \State  {$\xi_{i} \gets (-1)^{i-1}\sqrt{\left(1-\tau_{i}\right)\varpi _{i}}$}.
	\Else
	 \State {$\xi_{i} \gets (-1)^{i-1} \sin{\left(\cos^{-1}{\tau_{i}}\right)} \displaystyle{\sqrt{\frac{\varpi _{i}}{1+\tau_{i}}}}$}.
	  \EndIf
	\EndFor
\State Stop.
\end{algorithmic}
\end{algorithm}
%% ---------------------------------------------
%% References
\bibliographystyle{elsarticle-num-names}
\bibliography{Bib}

\begin{thebibliography}{73}
\expandafter\ifx\csname natexlab\endcsname\relax\def\natexlab#1{#1}\fi
\providecommand{\url}[1]{\texttt{#1}}
\providecommand{\href}[2]{#2}
\providecommand{\path}[1]{#1}
\providecommand{\DOIprefix}{doi:}
\providecommand{\ArXivprefix}{arXiv:}
\providecommand{\URLprefix}{URL: }
\providecommand{\Pubmedprefix}{pmid:}
\providecommand{\doi}[1]{\href{http://dx.doi.org/#1}{\path{#1}}}
\providecommand{\Pubmed}[1]{\href{pmid:#1}{\path{#1}}}
\providecommand{\bibinfo}[2]{#2}
\ifx\xfnm\relax \def\xfnm[#1]{\unskip,\space#1}\fi
%Type = Article
\bibitem[{Orszag(1971)}]{orszag1971accurate}
\bibinfo{author}{S.~A. Orszag},
\newblock \bibinfo{title}{Accurate solution of the {O}rr--{S}ommerfeld
  stability equation},
\newblock \bibinfo{journal}{Journal of Fluid Mechanics} \bibinfo{volume}{50}
  (\bibinfo{year}{1971}) \bibinfo{pages}{689--703}.
%Type = Article
\bibitem[{Patterson~Jr and Orszag(1971)}]{patterson1971spectral}
\bibinfo{author}{G.~Patterson~Jr}, \bibinfo{author}{S.~A. Orszag},
\newblock \bibinfo{title}{Spectral calculations of isotropic turbulence:
  {E}fficient removal of aliasing interactions},
\newblock \bibinfo{journal}{The Physics of Fluids} \bibinfo{volume}{14}
  (\bibinfo{year}{1971}) \bibinfo{pages}{2538--2541}.
%Type = Article
\bibitem[{Kang and Bedrossian(2007)}]{kang2007pseudospectral}
\bibinfo{author}{W.~Kang}, \bibinfo{author}{N.~Bedrossian},
\newblock \bibinfo{title}{Pseudospectral optimal control theory makes debut
  flight, saves {NASA} {\$} 1{M} in under three hours},
\newblock \bibinfo{journal}{SIAM News} \bibinfo{volume}{40}
  (\bibinfo{year}{2007}).
%Type = Article
\bibitem[{Elgindy and Karas{\"o}zen(2019)}]{elgindy2019high}
\bibinfo{author}{K.~T. Elgindy}, \bibinfo{author}{B.~Karas{\"o}zen},
\newblock \bibinfo{title}{High-order integral nodal discontinuous
  {G}egenbauer-{G}alerkin method for solving viscous {B}urgers' equation},
\newblock \bibinfo{journal}{International Journal of Computer Mathematics}
  \bibinfo{volume}{96} (\bibinfo{year}{2019}) \bibinfo{pages}{2039--2078}.
%Type = Article
\bibitem[{Fornberg and Sloan(1994)}]{fornberg1994review}
\bibinfo{author}{B.~Fornberg}, \bibinfo{author}{D.~M. Sloan},
\newblock \bibinfo{title}{A review of pseudospectral methods for solving
  partial differential equations},
\newblock \bibinfo{journal}{Acta numerica} \bibinfo{volume}{3}
  (\bibinfo{year}{1994}) \bibinfo{pages}{203--267}.
%Type = Book
\bibitem[{Fornberg(1998)}]{fornberg1998practical}
\bibinfo{author}{B.~Fornberg}, \bibinfo{title}{A practical guide to
  pseudospectral methods}, \bibinfo{number}{1}, \bibinfo{publisher}{Cambridge
  university press}, \bibinfo{year}{1998}.
%Type = Book
\bibitem[{Hesthaven et~al.(2007)Hesthaven, Gottlieb, and
  Gottlieb}]{hesthaven2007spectral}
\bibinfo{author}{J.~S. Hesthaven}, \bibinfo{author}{S.~Gottlieb},
  \bibinfo{author}{D.~Gottlieb}, \bibinfo{title}{Spectral methods for
  time-dependent problems}, volume~\bibinfo{volume}{21},
  \bibinfo{publisher}{Cambridge University Press}, \bibinfo{year}{2007}.
%Type = Article
\bibitem[{Canuto et~al.(1987)Canuto, Houssanini, Quarteroni, and
  Zang}]{canuto1987springer}
\bibinfo{author}{C.~Canuto}, \bibinfo{author}{M.~Houssanini},
  \bibinfo{author}{A.~Quarteroni}, \bibinfo{author}{T.~Zang},
\newblock \bibinfo{title}{Springer series in computational physics},
\newblock \bibinfo{journal}{Spectral Methods in Fluid Dynamics}
  (\bibinfo{year}{1987}).
%Type = Book
\bibitem[{Canuto et~al.(2007)Canuto, Hussaini, Quarteroni, and
  Zang}]{canuto2007spectral}
\bibinfo{author}{C.~Canuto}, \bibinfo{author}{M.~Y. Hussaini},
  \bibinfo{author}{A.~Quarteroni}, \bibinfo{author}{T.~A. Zang},
  \bibinfo{title}{Spectral methods: {F}undamentals in single domains},
  \bibinfo{publisher}{Springer Science \& Business Media},
  \bibinfo{year}{2007}.
%Type = Article
\bibitem[{Clenshaw and Curtis(1960)}]{clenshaw1960method}
\bibinfo{author}{C.~W. Clenshaw}, \bibinfo{author}{A.~R. Curtis},
\newblock \bibinfo{title}{A method for numerical integration on an automatic
  computer},
\newblock \bibinfo{journal}{Numerische Mathematik} \bibinfo{volume}{2}
  (\bibinfo{year}{1960}) \bibinfo{pages}{197--205}.
%Type = Article
\bibitem[{El-Gendi(1969)}]{el1969chebyshev}
\bibinfo{author}{S.~El-Gendi},
\newblock \bibinfo{title}{Chebyshev solution of differential, integral and
  integro-differential equations},
\newblock \bibinfo{journal}{The Computer Journal} \bibinfo{volume}{12}
  (\bibinfo{year}{1969}) \bibinfo{pages}{282--287}.
%Type = Article
\bibitem[{Lee and Greengard(1997)}]{lee1997fast}
\bibinfo{author}{J.-Y. Lee}, \bibinfo{author}{L.~Greengard},
\newblock \bibinfo{title}{A fast adaptive numerical method for stiff two-point
  boundary value problems},
\newblock \bibinfo{journal}{SIAM Journal on Scientific Computing}
  \bibinfo{volume}{18} (\bibinfo{year}{1997}) \bibinfo{pages}{403--429}.
%Type = Article
\bibitem[{Greengard(1991)}]{greengard1991spectral}
\bibinfo{author}{L.~Greengard},
\newblock \bibinfo{title}{Spectral integration and two-point boundary value
  problems},
\newblock \bibinfo{journal}{SIAM Journal on Numerical Analysis}
  \bibinfo{volume}{28} (\bibinfo{year}{1991}) \bibinfo{pages}{1071--1080}.
%Type = Article
\bibitem[{Elgindy and Smith-Miles(2013)}]{elgindy2013solving}
\bibinfo{author}{K.~T. Elgindy}, \bibinfo{author}{K.~A. Smith-Miles},
\newblock \bibinfo{title}{Solving boundary value problems, integral, and
  integro-differential equations using {G}egenbauer integration matrices},
\newblock \bibinfo{journal}{Journal of Computational and Applied Mathematics}
  \bibinfo{volume}{237} (\bibinfo{year}{2013}) \bibinfo{pages}{307--325}.
%Type = Article
\bibitem[{Elgindy(2016)}]{elgindy2016high}
\bibinfo{author}{K.~T. Elgindy},
\newblock \bibinfo{title}{High-order numerical solution of second-order
  one-dimensional hyperbolic telegraph equation using a shifted {G}egenbauer
  pseudospectral method},
\newblock \bibinfo{journal}{Numerical Methods for Partial Differential
  Equations} \bibinfo{volume}{32} (\bibinfo{year}{2016})
  \bibinfo{pages}{307--349}.
%Type = Article
\bibitem[{Elgindy and Dahy(2018)}]{elgindy2018high}
\bibinfo{author}{K.~T. Elgindy}, \bibinfo{author}{S.~A. Dahy},
\newblock \bibinfo{title}{High-order numerical solution of viscous {B}urgers'
  equation using a {C}ole-{H}opf barycentric {G}egenbauer integral
  pseudospectral method},
\newblock \bibinfo{journal}{Mathematical Methods in the Applied Sciences}
  \bibinfo{volume}{41} (\bibinfo{year}{2018}) \bibinfo{pages}{6226--6251}.
%Type = Article
\bibitem[{Elgindy and Refat(2018)}]{elgindy2018high1}
\bibinfo{author}{K.~T. Elgindy}, \bibinfo{author}{H.~M. Refat},
\newblock \bibinfo{title}{High-order shifted {G}egenbauer integral
  pseudo-spectral method for solving differential equations of {L}ane--{E}mden
  type},
\newblock \bibinfo{journal}{Applied Numerical Mathematics}
  \bibinfo{volume}{128} (\bibinfo{year}{2018}) \bibinfo{pages}{98--124}.
%Type = Article
\bibitem[{Ling and Caputo(2012)}]{ling2012envelope}
\bibinfo{author}{C.~Ling}, \bibinfo{author}{M.~R. Caputo},
\newblock \bibinfo{title}{The envelope theorem for locally differentiable
  {N}ash equilibria of discounted and autonomous infinite horizon differential
  games},
\newblock \bibinfo{journal}{Dynamic Games and Applications} \bibinfo{volume}{2}
  (\bibinfo{year}{2012}) \bibinfo{pages}{313--334}.
%Type = Article
\bibitem[{Barucci and Gozzi(2001)}]{barucci2001technology}
\bibinfo{author}{E.~Barucci}, \bibinfo{author}{F.~Gozzi},
\newblock \bibinfo{title}{Technology adoption and accumulation in a
  vintage-capital model},
\newblock \bibinfo{journal}{Journal of economics} \bibinfo{volume}{74}
  (\bibinfo{year}{2001}) \bibinfo{pages}{1--38}.
%Type = Article
\bibitem[{Ross and Karpenko(2012)}]{ross2012review}
\bibinfo{author}{I.~M. Ross}, \bibinfo{author}{M.~Karpenko},
\newblock \bibinfo{title}{A review of pseudospectral optimal control: {F}rom
  theory to flight},
\newblock \bibinfo{journal}{Annual Reviews in Control} \bibinfo{volume}{36}
  (\bibinfo{year}{2012}) \bibinfo{pages}{182--197}.
%Type = Article
\bibitem[{Gao et~al.(2019)Gao, Li, Shan, Xiao, Yuan, and Liu}]{gao2019online}
\bibinfo{author}{X.~Gao}, \bibinfo{author}{T.~Li}, \bibinfo{author}{Q.~Shan},
  \bibinfo{author}{Y.~Xiao}, \bibinfo{author}{L.~Yuan},
  \bibinfo{author}{Y.~Liu},
\newblock \bibinfo{title}{Online optimal control for dynamic positioning of
  vessels via time-based adaptive dynamic programming},
\newblock \bibinfo{journal}{Journal of Ambient Intelligence and Humanized
  Computing}  (\bibinfo{year}{2019}) \bibinfo{pages}{1--13}.
%Type = Article
\bibitem[{Wang et~al.(2022)Wang, Ha, and Zhao}]{wang2022intelligent}
\bibinfo{author}{D.~Wang}, \bibinfo{author}{M.~Ha}, \bibinfo{author}{M.~Zhao},
\newblock \bibinfo{title}{The intelligent critic framework for advanced optimal
  control},
\newblock \bibinfo{journal}{Artificial Intelligence Review}
  (\bibinfo{year}{2022}) \bibinfo{pages}{1--22}.
%Type = Article
\bibitem[{ParandehGheibi et~al.(2015)ParandehGheibi, Roozbehani, Dahleh, and
  Ozdaglar}]{parandehgheibi2015value}
\bibinfo{author}{A.~ParandehGheibi}, \bibinfo{author}{M.~Roozbehani},
  \bibinfo{author}{M.~A. Dahleh}, \bibinfo{author}{A.~Ozdaglar},
\newblock \bibinfo{title}{The value of storage in securing reliability and
  mitigating risk in energy systems},
\newblock \bibinfo{journal}{Energy Systems} \bibinfo{volume}{6}
  (\bibinfo{year}{2015}) \bibinfo{pages}{129--152}.
%Type = Article
\bibitem[{Janov{\'a} and Hampel(2016)}]{janova2016optimal}
\bibinfo{author}{J.~Janov{\'a}}, \bibinfo{author}{D.~Hampel},
\newblock \bibinfo{title}{Optimal managing of forest structure using data
  simulated optimal control},
\newblock \bibinfo{journal}{Central European Journal of Operations Research}
  \bibinfo{volume}{24} (\bibinfo{year}{2016}) \bibinfo{pages}{297--307}.
%Type = Article
\bibitem[{Pang et~al.(2022)Pang, Cui, and Jiang}]{pang2022human}
\bibinfo{author}{B.~Pang}, \bibinfo{author}{L.~Cui}, \bibinfo{author}{Z.-P.
  Jiang},
\newblock \bibinfo{title}{Human motor learning is robust to control-dependent
  noise},
\newblock \bibinfo{journal}{Biological Cybernetics}  (\bibinfo{year}{2022})
  \bibinfo{pages}{1--19}.
%Type = Article
\bibitem[{Baum(1976)}]{baum1976existence}
\bibinfo{author}{R.~F. Baum},
\newblock \bibinfo{title}{Existence theorems for {L}agrange control problems
  with unbounded time domain},
\newblock \bibinfo{journal}{Journal of Optimization Theory and Applications}
  \bibinfo{volume}{19} (\bibinfo{year}{1976}) \bibinfo{pages}{89--116}.
%Type = Article
\bibitem[{Bates(1978)}]{bates1978lower}
\bibinfo{author}{G.~Bates},
\newblock \bibinfo{title}{Lower closure and existence theorems for optimal
  control problems with infinite horizon},
\newblock \bibinfo{journal}{Journal of Optimization Theory and Applications}
  \bibinfo{volume}{24} (\bibinfo{year}{1978}) \bibinfo{pages}{639--649}.
%Type = Article
\bibitem[{Haurie(1980)}]{haurie1980existence}
\bibinfo{author}{A.~Haurie},
\newblock \bibinfo{title}{Existence and global asymptotic stability of optimal
  trajectories for a class of infinite-horizon, nonconvex systems},
\newblock \bibinfo{journal}{Journal of Optimization Theory and Applications}
  \bibinfo{volume}{31} (\bibinfo{year}{1980}) \bibinfo{pages}{515--533}.
%Type = Book
\bibitem[{Carlson and Haurie(1987)}]{carlson1987infinite}
\bibinfo{author}{D.~A. Carlson}, \bibinfo{author}{A.~Haurie},
  \bibinfo{title}{Infinite Horizon Optimal Control: Theory and Applications},
  volume \bibinfo{volume}{290}, \bibinfo{publisher}{Springer Verlag},
  \bibinfo{year}{1987}.
%Type = Article
\bibitem[{Balder(1983)}]{balder1983existence}
\bibinfo{author}{E.~Balder},
\newblock \bibinfo{title}{An existence result for optimal economic growth
  problems},
\newblock \bibinfo{journal}{Journal of mathematical analysis and applications}
  \bibinfo{volume}{95} (\bibinfo{year}{1983}) \bibinfo{pages}{195--213}.
%Type = Article
\bibitem[{Carlson(1986)}]{carlson1986existence}
\bibinfo{author}{D.~Carlson},
\newblock \bibinfo{title}{Existence of finitely optimal solutions for
  infinite-horizon optimal control problems},
\newblock \bibinfo{journal}{Journal of optimization theory and applications}
  \bibinfo{volume}{51} (\bibinfo{year}{1986}) \bibinfo{pages}{41--62}.
%Type = Article
\bibitem[{Wang(2005)}]{wang2005existence}
\bibinfo{author}{L.~Wang},
\newblock \bibinfo{title}{Existence and uniqueness of solutions for a class of
  infinite-horizon systems derived from optimal control},
\newblock \bibinfo{journal}{International journal of mathematics and
  mathematical sciences} \bibinfo{volume}{2005} (\bibinfo{year}{2005})
  \bibinfo{pages}{837--843}.
%Type = Article
\bibitem[{Pickenhain(2015)}]{pickenhain2015infinite}
\bibinfo{author}{S.~Pickenhain},
\newblock \bibinfo{title}{Infinite horizon optimal control problems in the
  light of convex analysis in hilbert spaces},
\newblock \bibinfo{journal}{Set-Valued and Variational Analysis}
  \bibinfo{volume}{23} (\bibinfo{year}{2015}) \bibinfo{pages}{169--189}.
%Type = Article
\bibitem[{Besov(2018)}]{besov2018balder}
\bibinfo{author}{K.~O. Besov},
\newblock \bibinfo{title}{On {B}alder’s existence theorem for
  infinite-horizon optimal control problems},
\newblock \bibinfo{journal}{Mathematical Notes} \bibinfo{volume}{103}
  (\bibinfo{year}{2018}) \bibinfo{pages}{167--174}.
%Type = Article
\bibitem[{Dmitruk and Kuz'kina(2005)}]{dmitruk2005existence}
\bibinfo{author}{A.~V. Dmitruk}, \bibinfo{author}{N.~Kuz'kina},
\newblock \bibinfo{title}{Existence theorem in the optimal control problem on
  an infinite time interval},
\newblock \bibinfo{journal}{Mathematical Notes} \bibinfo{volume}{78}
  (\bibinfo{year}{2005}) \bibinfo{pages}{466--480}.
%Type = Article
\bibitem[{Aseev(2018)}]{aseev2018existence}
\bibinfo{author}{S.~M. Aseev},
\newblock \bibinfo{title}{An existence result for infinite-horizon optimal
  control problem with unbounded set of control constraints},
\newblock \bibinfo{journal}{IFAC-PapersOnLine} \bibinfo{volume}{51}
  (\bibinfo{year}{2018}) \bibinfo{pages}{281--285}.
%Type = Article
\bibitem[{Basco and Frankowska(2019)}]{basco2019hamilton}
\bibinfo{author}{V.~Basco}, \bibinfo{author}{H.~Frankowska},
\newblock \bibinfo{title}{Hamilton--jacobi--bellman equations with
  time-measurable data and infinite horizon},
\newblock \bibinfo{journal}{Nonlinear Differential Equations and Applications
  NoDEA} \bibinfo{volume}{26} (\bibinfo{year}{2019}) \bibinfo{pages}{7}.
%Type = Article
\bibitem[{Halkin(1974)}]{halkin1974necessary}
\bibinfo{author}{H.~Halkin},
\newblock \bibinfo{title}{Necessary conditions for optimal control problems
  with infinite horizons},
\newblock \bibinfo{journal}{Econometrica: Journal of the Econometric Society}
  (\bibinfo{year}{1974}) \bibinfo{pages}{267--272}.
%Type = Inproceedings
\bibitem[{Garg et~al.(2002)Garg, Hager, and Rao}]{garg2002gauss}
\bibinfo{author}{D.~Garg}, \bibinfo{author}{W.~Hager},
  \bibinfo{author}{A.~Rao},
\newblock \bibinfo{title}{Gauss pseudospectral method for solving
  infinite-horizon optimal control problems},
\newblock in: \bibinfo{booktitle}{AIAA Guidance, Navigation, and Control
  Conference}, \bibinfo{year}{2002}, p. \bibinfo{pages}{7890}.
%Type = Article
\bibitem[{Garg et~al.(2011{\natexlab{a}})Garg, Hager, and
  Rao}]{garg2011pseudospectral}
\bibinfo{author}{D.~Garg}, \bibinfo{author}{W.~W. Hager},
  \bibinfo{author}{A.~V. Rao},
\newblock \bibinfo{title}{Pseudospectral methods for solving infinite-horizon
  optimal control problems},
\newblock \bibinfo{journal}{Automatica} \bibinfo{volume}{47}
  (\bibinfo{year}{2011}{\natexlab{a}}) \bibinfo{pages}{829--837}.
%Type = Article
\bibitem[{Garg et~al.(2011{\natexlab{b}})Garg, Patterson, Francolin, Darby,
  Huntington, Hager, and Rao}]{garg2011direct}
\bibinfo{author}{D.~Garg}, \bibinfo{author}{M.~A. Patterson},
  \bibinfo{author}{C.~Francolin}, \bibinfo{author}{C.~L. Darby},
  \bibinfo{author}{G.~T. Huntington}, \bibinfo{author}{W.~W. Hager},
  \bibinfo{author}{A.~V. Rao},
\newblock \bibinfo{title}{Direct trajectory optimization and costate estimation
  of finite-horizon and infinite-horizon optimal control problems using a
  {R}adau pseudospectral method},
\newblock \bibinfo{journal}{Computational Optimization and Applications}
  \bibinfo{volume}{49} (\bibinfo{year}{2011}{\natexlab{b}})
  \bibinfo{pages}{335--358}.
%Type = Article
\bibitem[{Shahini and Mehrpouya(2018)}]{shahini2018transformed}
\bibinfo{author}{M.~Shahini}, \bibinfo{author}{M.~Mehrpouya},
\newblock \bibinfo{title}{Transformed {L}egendre spectral method for solving
  infinite horizon optimal control problems},
\newblock \bibinfo{journal}{IMA Journal of Mathematical Control and
  Information} \bibinfo{volume}{35} (\bibinfo{year}{2018})
  \bibinfo{pages}{341--356}.
%Type = Article
\bibitem[{Gottlieb and Shu(1995)}]{gottlieb1995gibbs}
\bibinfo{author}{D.~Gottlieb}, \bibinfo{author}{C.-W. Shu},
\newblock \bibinfo{title}{On the {G}ibbs phenomenon. {IV}. {R}ecovering
  exponential accuracy in a subinterval from a {G}egenbauer partial sum of a
  piecewise analytic function},
\newblock \bibinfo{journal}{Mathematics of Computation} \bibinfo{volume}{64}
  (\bibinfo{year}{1995}) \bibinfo{pages}{1081--1095}.
%Type = Article
\bibitem[{Gottlieb and Shu(1997)}]{gottlieb1997gibbs}
\bibinfo{author}{D.~Gottlieb}, \bibinfo{author}{C.-W. Shu},
\newblock \bibinfo{title}{On the {G}ibbs phenomenon and its resolution},
\newblock \bibinfo{journal}{SIAM review} \bibinfo{volume}{39}
  (\bibinfo{year}{1997}) \bibinfo{pages}{644--668}.
%Type = Article
\bibitem[{Kamm et~al.(2010)Kamm, Williams, Brock, and Li}]{kamm2010application}
\bibinfo{author}{J.~R. Kamm}, \bibinfo{author}{T.~O. Williams},
  \bibinfo{author}{J.~S. Brock}, \bibinfo{author}{S.~Li},
\newblock \bibinfo{title}{Application of {G}egenbauer polynomial expansions to
  mitigate {G}ibbs phenomenon in {F}ourier--{B}essel series solutions of a
  dynamic sphere problem},
\newblock \bibinfo{journal}{International Journal for Numerical Methods in
  Biomedical Engineering} \bibinfo{volume}{26} (\bibinfo{year}{2010})
  \bibinfo{pages}{1276--1292}.
%Type = Article
\bibitem[{Elgindy and Smith-Miles(2013)}]{elgindy2013fast}
\bibinfo{author}{K.~T. Elgindy}, \bibinfo{author}{K.~A. Smith-Miles},
\newblock \bibinfo{title}{Fast, accurate, and small-scale direct trajectory
  optimization using a {G}egenbauer transcription method},
\newblock \bibinfo{journal}{Journal of Computational and Applied Mathematics}
  \bibinfo{volume}{251} (\bibinfo{year}{2013}) \bibinfo{pages}{93--116}.
%Type = Article
\bibitem[{Elgindy and Karas{\"o}zen(2020)}]{elgindy2020distributed}
\bibinfo{author}{K.~T. Elgindy}, \bibinfo{author}{B.~Karas{\"o}zen},
\newblock \bibinfo{title}{Distributed optimal control of viscous {B}urgers'
  equation via a high-order, linearization, integral, nodal discontinuous
  {G}egenbauer-{G}alerkin method},
\newblock \bibinfo{journal}{Optimal Control Applications and Methods}
  \bibinfo{volume}{41} (\bibinfo{year}{2020}) \bibinfo{pages}{253--277}.
%Type = Article
\bibitem[{Doha(1990)}]{doha1990accurate}
\bibinfo{author}{E.~Doha},
\newblock \bibinfo{title}{An accurate solution of parabolic equations by
  expansion in ultraspherical polynomials},
\newblock \bibinfo{journal}{Computers \& Mathematics with Applications}
  \bibinfo{volume}{19} (\bibinfo{year}{1990}) \bibinfo{pages}{75--88}.
%Type = Inproceedings
\bibitem[{Abd-Elhameed and Youssri(2014)}]{abd2014new}
\bibinfo{author}{W.~Abd-Elhameed}, \bibinfo{author}{Y.~Youssri},
\newblock \bibinfo{title}{New ultraspherical wavelets spectral solutions for
  fractional {R}iccati differential equations},
\newblock in: \bibinfo{booktitle}{Abstract and applied analysis}, volume
  \bibinfo{volume}{2014}, \bibinfo{organization}{Hindawi},
  \bibinfo{year}{2014}.
%Type = Article
\bibitem[{Elgindy and Smith-Miles(2013)}]{elgindy2013optimal}
\bibinfo{author}{K.~T. Elgindy}, \bibinfo{author}{K.~A. Smith-Miles},
\newblock \bibinfo{title}{Optimal {G}egenbauer quadrature over arbitrary
  integration nodes},
\newblock \bibinfo{journal}{Journal of Computational and Applied Mathematics}
  \bibinfo{volume}{242} (\bibinfo{year}{2013}) \bibinfo{pages}{82--106}.
%Type = Misc
\bibitem[{Elgindy(2017)}]{Elgindy2016}
\bibinfo{author}{K.~T. Elgindy}, \bibinfo{title}{Optimal control of a parabolic
  distributed parameter system using a fully exponentially convergent
  barycentric shifted {G}egenbauer integral pseudospectral method},
  \bibinfo{howpublished}{Journal of Industrial and Management Optimization.
  AMER INST MATHEMATICAL SCIENCES-AIMS. DOI: 10.3934/jimo.2017056.},
  \bibinfo{year}{2017}.
%Type = Article
\bibitem[{Light(1978)}]{light1978comparison}
\bibinfo{author}{W.~Light},
\newblock \bibinfo{title}{A comparison between {C}hebyshev and ultraspherical
  expansions},
\newblock \bibinfo{journal}{IMA Journal of Applied Mathematics}
  \bibinfo{volume}{21} (\bibinfo{year}{1978}) \bibinfo{pages}{455--460}.
%Type = Article
\bibitem[{Boyd(1987)}]{boyd1987orthogonal}
\bibinfo{author}{J.~P. Boyd},
\newblock \bibinfo{title}{Orthogonal rational functions on a semi-infinite
  interval},
\newblock \bibinfo{journal}{Journal of Computational Physics}
  \bibinfo{volume}{70} (\bibinfo{year}{1987}) \bibinfo{pages}{63--88}.
%Type = Article
\bibitem[{Fahroo and Ross(2008)}]{fahroo2008pseudospectral}
\bibinfo{author}{F.~Fahroo}, \bibinfo{author}{I.~M. Ross},
\newblock \bibinfo{title}{Pseudospectral methods for infinite-horizon nonlinear
  optimal control problems},
\newblock \bibinfo{journal}{Journal of Guidance, Control, and Dynamics}
  \bibinfo{volume}{31} (\bibinfo{year}{2008}) \bibinfo{pages}{927--936}.
%Type = Phdthesis
\bibitem[{Garg(2011)}]{garg2011advances}
\bibinfo{author}{D.~Garg}, \bibinfo{title}{Advances in global pseudospectral
  methods for optimal control}, Ph.D. thesis, University of Florida
  Gainesville, FL, \bibinfo{year}{2011}.
%Type = Book
\bibitem[{Szeg\"{o}(1939)}]{szeg1939orthogonal}
\bibinfo{author}{G.~Szeg\"{o}}, \bibinfo{title}{Orthogonal polynomials},
  volume~\bibinfo{volume}{23}, \bibinfo{publisher}{American Mathematical Soc.},
  \bibinfo{year}{1939}.
%Type = Article
\bibitem[{Wang et~al.(2014)Wang, Huybrechs, and Vandewalle}]{wang2014explicit}
\bibinfo{author}{H.~Wang}, \bibinfo{author}{D.~Huybrechs},
  \bibinfo{author}{S.~Vandewalle},
\newblock \bibinfo{title}{Explicit barycentric weights for polynomial
  interpolation in the roots or extrema of classical orthogonal polynomials},
\newblock \bibinfo{journal}{Mathematics of Computation} \bibinfo{volume}{83}
  (\bibinfo{year}{2014}) \bibinfo{pages}{2893--2914}.
%Type = Article
\bibitem[{Elgindy(2017)}]{elgindy2017high}
\bibinfo{author}{K.~T. Elgindy},
\newblock \bibinfo{title}{High-order adaptive {G}egenbauer integral spectral
  element method for solving non-linear optimal control problems},
\newblock \bibinfo{journal}{Optimization} \bibinfo{volume}{66}
  (\bibinfo{year}{2017}) \bibinfo{pages}{811--836}.
%Type = Inproceedings
\bibitem[{Berrut(1994)}]{berrut1994linear}
\bibinfo{author}{J.-P. Berrut},
\newblock \bibinfo{title}{Linear rational interpolation of continuous functions
  over an interval},
\newblock in: \bibinfo{booktitle}{Proceedings of Symposia in Applied
  Mathematics, W. Gautschi, ed., AMS, Providence, RI}, \bibinfo{year}{1994},
  pp. \bibinfo{pages}{261--264}.
%Type = Article
\bibitem[{Berrut and Mittelmann(1997)}]{berrut1997lebesgue}
\bibinfo{author}{J.-P. Berrut}, \bibinfo{author}{H.~D. Mittelmann},
\newblock \bibinfo{title}{Lebesgue constant minimizing linear rational
  interpolation of continuous functions over the interval},
\newblock \bibinfo{journal}{Computers \& Mathematics with Applications}
  \bibinfo{volume}{33} (\bibinfo{year}{1997}) \bibinfo{pages}{77--86}.
%Type = Article
\bibitem[{Carnicer(2010)}]{carnicer2010weighted}
\bibinfo{author}{J.~M. Carnicer},
\newblock \bibinfo{title}{Weighted interpolation for equidistant nodes},
\newblock \bibinfo{journal}{Numerical Algorithms} \bibinfo{volume}{55}
  (\bibinfo{year}{2010}) \bibinfo{pages}{223--232}.
%Type = Article
\bibitem[{Wang et~al.(2010)Wang, Moin, and Iaccarino}]{wang2010rational}
\bibinfo{author}{Q.~Wang}, \bibinfo{author}{P.~Moin},
  \bibinfo{author}{G.~Iaccarino},
\newblock \bibinfo{title}{A rational interpolation scheme with superpolynomial
  rate of convergence},
\newblock \bibinfo{journal}{SIAM Journal on Numerical Analysis}
  \bibinfo{volume}{47} (\bibinfo{year}{2010}) \bibinfo{pages}{4073--4097}.
%Type = Article
\bibitem[{Bos et~al.(2013)Bos, De~Marchi, Hormann, and Sidon}]{bos2013bounding}
\bibinfo{author}{L.~Bos}, \bibinfo{author}{S.~De~Marchi},
  \bibinfo{author}{K.~Hormann}, \bibinfo{author}{J.~Sidon},
\newblock \bibinfo{title}{Bounding the {L}ebesgue constant for {B}errut’s
  rational interpolant at general nodes},
\newblock \bibinfo{journal}{Journal of Approximation Theory}
  \bibinfo{volume}{169} (\bibinfo{year}{2013}) \bibinfo{pages}{7--22}.
%Type = Article
\bibitem[{Berrut(1988)}]{berrut1988rational}
\bibinfo{author}{J.-P. Berrut},
\newblock \bibinfo{title}{Rational functions for guaranteed and experimentally
  well-conditioned global interpolation},
\newblock \bibinfo{journal}{Computers \& Mathematics with Applications}
  \bibinfo{volume}{15} (\bibinfo{year}{1988}) \bibinfo{pages}{1--16}.
%Type = Article
\bibitem[{Bos et~al.(2011)Bos, De~Marchi, and Hormann}]{bos2011lebesgue}
\bibinfo{author}{L.~Bos}, \bibinfo{author}{S.~De~Marchi},
  \bibinfo{author}{K.~Hormann},
\newblock \bibinfo{title}{On the {L}ebesgue constant of {B}errut’s rational
  interpolant at equidistant nodes},
\newblock \bibinfo{journal}{Journal of Computational and Applied Mathematics}
  \bibinfo{volume}{236} (\bibinfo{year}{2011}) \bibinfo{pages}{504--510}.
%Type = Article
\bibitem[{Elgindy(2017)}]{elgindy2017high1}
\bibinfo{author}{K.~T. Elgindy},
\newblock \bibinfo{title}{High-order, stable, and efficient pseudospectral
  method using barycentric {G}egenbauer quadratures},
\newblock \bibinfo{journal}{Applied Numerical Mathematics}
  \bibinfo{volume}{113} (\bibinfo{year}{2017}) \bibinfo{pages}{1--25}.
%Type = Article
\bibitem[{Hestenes(1969)}]{hestenes1969multiplier}
\bibinfo{author}{M.~R. Hestenes},
\newblock \bibinfo{title}{Multiplier and gradient methods},
\newblock \bibinfo{journal}{Journal of optimization theory and applications}
  \bibinfo{volume}{4} (\bibinfo{year}{1969}) \bibinfo{pages}{303--320}.
%Type = Article
\bibitem[{Powell(1969)}]{powell1969method}
\bibinfo{author}{M.~J. Powell},
\newblock \bibinfo{title}{A method for nonlinear constraints in minimization
  problems},
\newblock \bibinfo{journal}{Optimization}  (\bibinfo{year}{1969})
  \bibinfo{pages}{283--298}.
%Type = Article
\bibitem[{Elgindy(2018)}]{elgindy2018optimization}
\bibinfo{author}{K.~T. Elgindy},
\newblock \bibinfo{title}{Optimization via {C}hebyshev polynomials},
\newblock \bibinfo{journal}{Journal of Applied Mathematics and Computing}
  \bibinfo{volume}{56} (\bibinfo{year}{2018}) \bibinfo{pages}{317--349}.
%Type = Article
\bibitem[{Gill et~al.(2002)Gill, Murray, and Saunders}]{doiS1052623499350013}
\bibinfo{author}{P.~E. Gill}, \bibinfo{author}{W.~Murray},
  \bibinfo{author}{M.~A. Saunders},
\newblock \bibinfo{title}{{SNOPT}: {A}n {SQP} algorithm for large-scale
  constrained optimization},
\newblock \bibinfo{journal}{SIAM Journal on Optimization} \bibinfo{volume}{12}
  (\bibinfo{year}{2002}) \bibinfo{pages}{979--1006}.
%Type = Article
\bibitem[{Gill et~al.(2005)Gill, Murray, and Saunders}]{gill2005snopt}
\bibinfo{author}{P.~E. Gill}, \bibinfo{author}{W.~Murray},
  \bibinfo{author}{M.~A. Saunders},
\newblock \bibinfo{title}{{SNOPT}: {A}n {SQP} algorithm for large-scale
  constrained optimization},
\newblock \bibinfo{journal}{SIAM review} \bibinfo{volume}{47}
  (\bibinfo{year}{2005}) \bibinfo{pages}{99--131}.
%Type = Article
\bibitem[{Elgindy(2018)}]{elgindy2018optimal}
\bibinfo{author}{K.~T. Elgindy},
\newblock \bibinfo{title}{Optimal control of a parabolic distributed parameter
  system using a fully exponentially convergent barycentric shifted
  {G}egenbauer integral pseudospectral method},
\newblock \bibinfo{journal}{Journal of Industrial \& Management Optimization}
  \bibinfo{volume}{14} (\bibinfo{year}{2018}) \bibinfo{pages}{473}.
%Type = Article
\bibitem[{Kirk(1970)}]{kirk1970prentice}
\bibinfo{author}{D.~O. C.~T. Kirk},
\newblock \bibinfo{title}{Prentice-hall},
\newblock \bibinfo{journal}{Optimal Control Theory: An Introduction}
  (\bibinfo{year}{1970}).

\end{thebibliography}
%% ---------------------------------------------
\end{document}